\documentclass[11pt]{amsart}
\usepackage[utf8]{inputenc}

\usepackage{microtype}
\usepackage{amsmath, amstext, amssymb, pgfkeys, graphicx, tikz, amsthm, amsfonts, hyperref,longtable, bm, bbm, nicefrac,color,bbm,pgfplots,cases,mathrsfs}
\usepackage{caption}
\usepackage{subcaption}
\usepackage[top=1in, bottom=1.5in, left=1in, right=1in]{geometry}
\usetikzlibrary{shapes.geometric}
 \usetikzlibrary{intersections}
 \usepackage[nocompress]{cite}
\usetikzlibrary{arrows, calc}

\usepackage{mdwlist}
\usepackage{todonotes}
\usepackage{enumerate}
\RequirePackage{color}
\definecolor{my-blue}{rgb}{0.0,0.0,0.6}
\definecolor{my-red}{rgb}{0.5,0.0,0.0}
\definecolor{my-green}{rgb}{0.0,0.5,0.0}
\definecolor{nicos-red}{rgb}{0.75,0.0,0.0}
\definecolor{light-gray}{gray}{0.6}
\definecolor{really-light-gray}{gray}{0.8}
\definecolor{sussexg}{rgb}{0.0,0.5,0.5}
\definecolor{sussexp}{rgb}{0.5,0.0,0.5}

\newtheorem{theorem}{\color{blue} \sc Theorem}[section]
\newtheorem{lemma}[theorem]{\color{blue} \sc Lemma}
\newtheorem{proposition}[theorem]{\color{blue} \sc Proposition}
\newtheorem{corollary}[theorem]{\color{blue} \sc Corollary}
\newtheorem{assumption}[theorem]{\color{nicos-red} \bf Assumption}
\newtheorem{definition}[theorem]{\color{nicos-red} \bf Definition}
\newtheorem{example}[theorem]{\color{nicos-red} \bf Example}
\numberwithin{equation}{section}
\theoremstyle{remark}
\newtheorem{remark}[theorem]{\color{nicos-red} \bf Remark}

\newcommand{\be}{\begin{equation}}
\newcommand{\ee}{\end{equation}}

\providecommand{\P}[1]{\langle#1\rangle}
\newcommand{\fl}[1]{\left\lfloor{#1}\right\rfloor}

\def\bN{\mathbb{N}}
\def\bP{\mathbb{P}}

\def\bR{\mathbb{R}}
\def\bZ{\mathbb{Z}}

\def\e{\varepsilon}

\def\om{\omega}

 \def\Z{\bZ}

\def\R{\bR}
\def\N{\bN}

\def\P{\bP}

\allowdisplaybreaks

\begin{document}
%\usdate
\title[LPP in discontinuous environment]
{Last passage percolation in an exponential environment with discontinuous rates}

\author[F.~Ciech]{Federico Ciech}
\address{Federico Ciech\\ University of Sussex\\ Department of  Mathematics \\ Falmer Campus\\ Brighton BN1 9QH\\ UK.}
\email{F.Ciech@sussex.ac.uk}
\urladdr{http://www.sussex.ac.uk/profiles/395447} 

\author[N.~Georgiou]{Nicos Georgiou}
\address{Nicos Georgiou\\ University of Sussex\\ Department of  Mathematics \\ Falmer Campus\\ Brighton BN1 9QH\\ UK.}
\email{N.Georgiou@sussex.ac.uk}
\urladdr{http://www.sussex.ac.uk/profiles/329373}

\thanks{N. Georgiou was partially supported by the EPSRC First Grant EP/P021409/1: The flat edge in last passage percolation.}

% \date{\today}
\begin{abstract}
	We prove a strong law of large numbers for directed last passage times in an independent but inhomogeneous exponential environment. Rates for the exponential random variables are obtained from a discretisation of a speed function that may be discontinuous on a locally finite set of discontinuity curves. The limiting shape is cast as a variational formula that maximizes a certain functional over a set of weakly increasing curves. 
	
	Using this result, we present two examples that allow for partial analytical tractability and show that the shape function may not be strictly concave, and it may exhibit points of non-differentiability, flat segments, and non-uniqueness of the optimisers of the variational formula. Finally, in a specific example, we analyse further the macroscopic optimisers and uncover a phase transition for their behaviour.
\end{abstract}

\maketitle

\section{Introduction}
We consider a model of directed last passage growth model in two dimensions, where each lattice site $(i,j)$ of $\Z_+^2$ is given a random weight $\tau_{i,j}$ according to some background measure $\P$. 

Given lattice points $(a,b),(u,v)\in \Z^2_+$, $\Pi_{(a,b),(u,v)}$ is the set of lattice paths $\pi=\{(a,b)=(i_0,j_0),(i_1,j_1),\dots,(i_p,j_p)=(u,v)\}$ whose admissible steps satisfy 
\begin{equation}\label{steps}
(i_\ell,j_\ell)-(i_{\ell-1},j_{\ell-1})\in\{(1,0),(0,1)\}.
\end{equation}
If $(a,b)=(0,0)$ we simply denote this set by $\Pi_{u,v}$. 

For $(u,v)\in\Z^2_+$ and $n\in \mathbb{N}$ denote the \textit{last passage time}
\begin{equation}\label{microLPT0}
G_{(a,b),(u,v)}=\max_{\pi\in\Pi_{(a,b), (u,v)}}\bigg\{\sum_{(i,j)\in\pi}\tau_{i,j}\bigg\}.
\end{equation}
Again, if $(a,b)=(0,0)$ and no confusion arises, we simply denote $G_{(0,0),(u,v)}$ with $G_{u,v}$.  In the homogeneous setting, $\{ \tau_{i,j} \}_{(i,j) \in \Z_2^+}$ are i.i.d.\ under $\P$ and standard subadditivity arguments give the existence of a point-to-point scaling limit 
\[
\lim_{n\to \infty} \frac{G_{\fl{nx},\fl{ny}}}{n} = g_{pp}(x,y).
\]
Generic properties of $g_{pp}(x,y)$ have been obtained in \cite{mar-04}, that are universal under some mild conditions on the distribution of $\tau_{i,j}$. In \cite{bodineau2005universality},  a distributional limit to a Tracy-Widom law was proven 
for passage times `near the edge', i.e. for passage times in thin rectangles of order $n \times n^a$. It is expected that several properties of the last passage models hold irrespective of the distribution of $\tau_{i,j}$; these include the fluctuation exponent of $G_{\fl{nx},\fl{ny}}$, limiting laws and fluctuations of the maximal path around its macroscopic direction. As far as the law of large numbers goes, an universal approach, under only some moment assumptions on the distribution of $\tau_{i,j}$, has been developed in \cite{georgiou2016variational, rassoul2014quenched, rassoul2013quenched, rassoul2017variational}, where the limiting shape is given in terms of variational formulas. 

When the environment $\tau_{i,j} \sim \text{Exp(1)}$, the last passage model is one of the exactly solvable models of the KPZ class (see \cite{Cor-12} for a survey).
The strong law of large numbers in the exponential model is explicitly computed in \cite{Ros-81}
\be\label{eq:homoLPP}
\lim_{n\to \infty} \frac{G_{\fl{nx},\fl{ny}}}{n} = \gamma(x,y) = (\sqrt{x} + \sqrt{y})^2,\quad  \P\text{-a.s.}
\ee

In this article we derive the limiting constant for a sequence of scaled last passage times on the lattice. The passage times themselves are coupled through a common realization of exponential random variables. However, the rates of these random variables will be chosen according to a discrete approximation of a macroscopic function 
\[
c: \R^2_+ \longrightarrow \R_+.
\]   
 Consider the lattice corner $\Z^2_+$. The environment $\tau = \{ \tau_{i,j} \}_{(i,j) \in \Z^2_+}$ is a collection of i.i.d.\ exponential random variables of rate 1. For any $n \in \N$ we alter the rate of each of these random variables by a scalar multiplication using the macroscopic speed function $c(x,y)$. Namely, define 
\begin{equation}\label{eq:DSF}
r_{i,j}^{(n)} = c\Big(\frac{i}{n},\frac{j}{n}\Big)^{-1},\qquad (i,j)\in\bZ^2_+,
\ee
and define $n$-scaled, inhomogeneous environment by 
\be
\tau^{(n)}_{i,j} = r_{i,j}^{(n)} \tau^n_{i,j}.
\ee      
The rate of the exponential random variable $\tau^{(n)}_{i,j}$ is now determined by the scalar  $c\big(\frac{i}{n},\frac{j}{n}\big)$. 
On each site the rate is completely determined by the speed function $c(\cdot, \cdot)$. We indicate the corresponding exponential 1 random variable as $\tau^n_{i,j}$.

For $(u,v)\in\Z^2_+$ and $n\in \mathbb{N}$ denote the last passage time
\begin{equation}\label{microLPT}
G^{(n)}_{u,v}=\max_{\pi\in\Pi_{u,v}}\bigg\{\sum_{(i,j)\in\pi}r_{i,j}^{(n)}\tau_{i,j}^n\bigg\}=\max_{\pi\in\Pi_{u,v}}\bigg\{\sum_{(i,j)\in\pi}\tau_{i,j}^{(n)}\bigg\}.
\end{equation}
 
 We impose several conditions on the function $c(x,y)$ and they are described in Section \ref{sec:model}. For the moment we emphasise that for any compact set $K \subseteq \R^2_+$ there exist finite constant $m_K$ and $M_K$ such that 
\[
m_K \le c(x, y) \le M_K \quad \text{ for all } (x,y )\in K
\]
and there are a finite number (that depends on $K$) of discontinuity curves of the function $c(x,y)$. These are to avoid degeneracies: If $c(x,y)$ can take the value $0$ then the environment could take the value $\infty$ which leads to trivial passage times. If 
$c(x,y)$ can be infinity, that region of space will never be explored by a path. If the discontinuities have an accumulation point, then no descritisation of $c(x,y)$ can capture that.  

We prove a strong law of large numbers for $n^{-1}G^{(n)}_{\fl{nx},\fl{ny}}$. The limiting last passage constant $\Gamma_c(x,y)$ has a variational characterization that naturally leads to a continuous version of a last passage time model (see Theorem \ref{thm:1}). We study the variational formula and discuss properties of the shape 
$\Gamma_c(x,y)$ and obtain explicit minimizers in two cases of interest. 

The first example is the \textit{shifted two-phase model} with speed function
\be\label{eq:cell0}
c_{\ell}(x,y) = 	\begin{cases}
			1, & \text{if } y>x-\lambda,\\
			r, & \text{if } y\le x-\lambda.
		\end{cases} 
\ee 
and the second model is the \textit{corner-inhomogeneous} model with speed function 
\be \label{eq:cfxy0}
c_{f}(x,y) = 
\begin{cases}
1, & f(x) > y,\\
r, & f(x) < y,\\
1\wedge r, & f(x) =y.
\end{cases}
\ee    
Precise assumptions on $f, r, \lambda$ can be found in Section \ref{sec:model}.

\subsection{Inhomogeneous growth models} We are concerned with directed last passage percolation on the lattice  in a discontinuous 
environment; weights $\om_{i,j}$ at each site $(i,j)$ are exponentially distributed but with different rates that depend on their position. Similar 
arguments can be repeated when the environment comes from geometrically distributed weights, and in this case the inhomogeneity will be 
captured by changing the values of $p$, the probability of success of the geometric weight. Such models do not have the super-additivity properties that guarantee the existence of a macroscopic shape, so other techniques must be utilised to first show existence of macroscopic limits and then compute a formula for them.

Several inhomogeneous models of last passage percolation exist, each one with different ways of assigning rates (or weights in general). One way is to fix two positive sequences $\{ a_{i} \} _{i \in \N}$ and $\{ b_j \}_{j \in \N}$ to assign to site $(i,j)$ an exponential weight $\om_{i,j}$ with rate $a_i + b_j$. Laws of large numbers for the last passage time for these model was obtained in \cite{Sep-Kru-99} when $a_i$ where i.i.d.\ and $b_j$ constant, and then generalised in \cite{Emr-16}. The model enjoys several aspects of integrability, and large deviations from the shape were obtained in \cite{Emr-Jan-15}. When admissible steps are not restricted to just $e_1, e_2$, \cite{Grav-Tra-Wid-01} studies an inhomogeneous model which generalises the one introduced in \cite{Sep-98-aop-2} and obtain explicit distributional limits for fluctuations of the passage time.   

Macroscopic inhomogeneities defined via the speed function have been already considered in the literature. 
When the speed function is continuous, \cite{Rol-Tei-08} showed the law of large numbers for the passage times and convergence of the microscopic maximal paths to a continuous curve conditioned on uniqueness of the macroscopic maximiser.

When the speed function is $c(x,y) = r \mathbbm1 \{ x =y \} +  \mathbbm1 \{ x \neq y \}$ the law of large numbers was studied in  \cite{Sep-01-jsp, Bes-Sid-Sly-14} and it was shown that for small values of $r$ the LLN disagrees with that of the 1-homogeneous model. When the discontinuity curves of $c(x,y)$ was a locally finite  set of lines of the form $\{ y = x + b_i \}_{i \in \N}$, the law of large numbers limits was obtained in  \cite{Geo-Kum-Sep-10} and an explicit limit for the shape function was obtained in the case of the two-phase model with $c(x,y) = r_1 \mathbbm1 \{ x \le y \} + r_2 \mathbbm1 \{ x > y \}$. In this case a flat edge was observed for the limiting shape function. A first passage (unoriented)  percolation two-phase model was studied in \cite{Ahl-Dam-Sid-16}, where the edge-weight distribution was different to the left and right half-planes and in certain cases proved the creation of a `pyramid' in the limiting shape, i.e. a polygonal segment with a point of non-differentiability at the peak. 

In \cite{Cal-15} the law of large numbers for directed last passage percolation was extended when the set of discontinuity curves for $c(x,y)$ was a locally finite set of piecewise Lipschitz strictly increasing curves. A PDE approach was used, bypassing the usual techniques of the totally asymmetric simple exclusion process (TASEP)  particle systems, used in the earlier articles. 

In general, several models of percolation with inhomogeneities can be understood by looking at particle systems with inhomogeneities. For directed  nearest neighbour LPP the corresponding particle system is TASEP. A standard coupling connects TASEP with corner growth. To visualize it, rotate the corner growth model of $\pi/4$ anti-clockwise and the resulting shape is the so called \textit{wedge}. Particles occupy sites of $\Z$, subject to the exclusion rule that does not allow for two particles to occupy the same site.

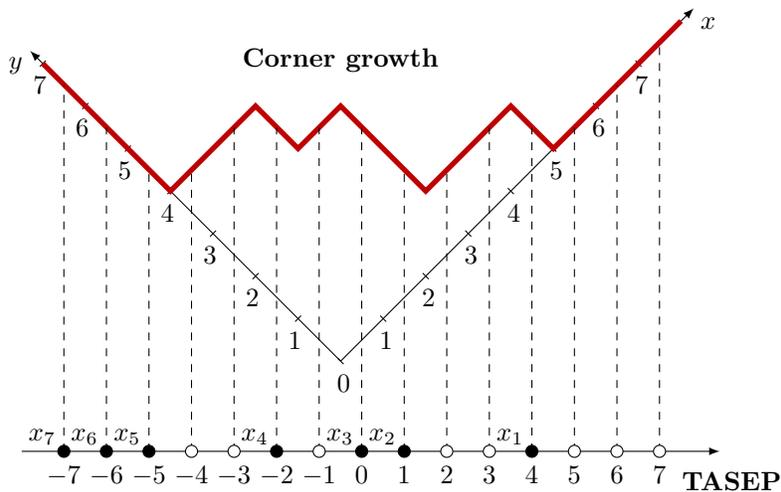
\begin{figure}[h!]
\begin{center}
\begin{tikzpicture}[>= latex, scale=0.8]
\draw[->,rotate=45]  (0,0) -- (8.3,0) node[xshift=0.5em,yshift=-0.5em] {\small$x$};
% vertical axis
\draw[->,rotate=45]   (0,0)-- (0,7.3) node[xshift=-0.5em,yshift=-0.5em]  {\small$y$};
  \foreach \x/\xtext in {0/0,1/1,2/2, 3/3,4/4,5/5,6/6,7/7}
    \draw[shift={(\x*0.707,\x*0.707)},rotate=45] (0pt,2pt) -- (0pt,-2pt) node[below] {\small$\xtext$};
  \foreach \x/\xtext in {1/1,2/2, 3/3,4/4,5/5,6/6,7/7}
    \draw[shift={(-\x*0.707,\x*0.707)},rotate=-45] (0pt,2pt) -- (0pt,-2pt) node[below] {\small$\xtext$};
  \foreach \x/\xtext in {-7/-7,-6/-6,-5/-5,-4/-4,-3/-3,-2/-2,-1/-1,0/0,1/1,2/2, 3/3,4/4,5/5,6/6,7/7}
    \draw[shift={(\x*0.707+0.35,-1.5)}] (0pt,2pt) -- (0pt,-2pt) node[below] {\small$\xtext$};
 \draw[line width=2 pt, color=nicos-red,rotate=45] (0,7)--(0,4)--(2,4)--(2,3)--(3,3)--(3,1)--(5,1)--(5,0)--(8,0);
\draw[->] (-5.3,-1.5)--(6.3,-1.5);
\draw[dashed] (-7*0.707+0.35,-1.5)--(-7*0.707+0.35,6*0.707+0.35);
\draw[dashed] (-6*0.707+0.35,-1.5)--(-6*0.707+0.35,5*0.707+0.35);
\draw[dashed] (-5*0.707+0.35,-1.5)--(-5*0.707+0.35,4*0.707+0.35);
\draw[dashed] (-4*0.707+0.35,-1.5)--(-4*0.707+0.35,4*0.707+0.35);
\draw[dashed] (-3*0.707+0.35,-1.5)--(-3*0.707+0.35,5*0.707+0.35);
\draw[dashed] (-2*0.707+0.35,-1.5)--(-2*0.707+0.35,5*0.707+0.35);
\draw[dashed] (-1*0.707+0.35,-1.5)--(-1*0.707+0.35,5*0.707+0.35);
\draw[dashed] (0*0.707+0.35,-1.5)--(-0*0.707+0.35,5*0.707+0.35);
\draw[dashed] (1*0.707+0.35,-1.5)--(1*0.707+0.35,4*0.707+0.35);
\draw[dashed] (2*0.707+0.35,-1.5)--(2*0.707+0.35,4*0.707+0.35);
\draw[dashed] (3*0.707+0.35,-1.5)--(3*0.707+0.35,5*0.707+0.35);
\draw[dashed] (4*0.707+0.35,-1.5)--(4*0.707+0.35,5*0.707+0.35);
\draw[dashed] (5*0.707+0.35,-1.5)--(5*0.707+0.35,5*0.707+0.35);
\draw[dashed] (6*0.707+0.35,-1.5)--(6*0.707+0.35,6*0.707+0.35);
\draw[dashed] (7*0.707+0.35,-1.5)--(7*0.707+0.35,7*0.707+0.35);
 \draw  [fill] (-7*0.707+0.35,-1.5) circle [radius=0.1]  node[xshift=-0.75em,yshift=0.5em] {\small$x_7$};
 \draw  [fill] (-6*0.707+0.35,-1.5) circle [radius=0.1]  node[xshift=-0.75em,yshift=0.5em] {\small$x_6$};
 \draw  [fill] (-5*0.707+0.35,-1.5) circle [radius=0.1]  node[xshift=-0.75em,yshift=0.5em] {\small$x_5$};
 \draw  [color=black,fill=white] (-4*0.707+0.35,-1.5) circle [radius=0.1];
 \draw   [color=black,fill=white]  (-3*0.707+0.35,-1.5) circle [radius=0.1] ;
 \draw  [fill] (-2*0.707+0.35,-1.5) circle [radius=0.1]  node[xshift=-0.75em,yshift=0.5em] {\small$x_4$};
 \draw   [color=black,fill=white]  (-1*0.707+0.35,-1.5) circle [radius=0.1] ;
 \draw  [fill] (0*0.707+0.35,-1.5) circle [radius=0.1]  node[xshift=-0.75em,yshift=0.5em] {\small$x_3$};
 \draw  [fill] (1*0.707+0.35,-1.5) circle [radius=0.1]  node[xshift=-0.75em,yshift=0.5em] {\small$x_2$};
 \draw   [color=black,fill=white]  (2*0.707+0.35,-1.5) circle [radius=0.1] ;
 \draw   [color=black,fill=white]  (3*0.707+0.35,-1.5) circle [radius=0.1] ;
 \draw  [fill] (4*0.707+0.35,-1.5) circle [radius=0.1]  node[xshift=-0.75em,yshift=0.5em] {\small$x_1$};
 \draw  [color=black,fill=white] (5*0.707+0.35,-1.5) circle [radius=0.1];
 \draw   [color=black,fill=white]  (6*0.707+0.35,-1.5) circle [radius=0.1] ;
 \draw  [color=black,fill=white] (7*0.707+0.35,-1.5) circle [radius=0.1];
\draw (0,5) node{\small$\text{\textbf{Corner growth}}$};
\draw (6.5,-2) node{\small$\text{\textbf{TASEP}}$};
\end{tikzpicture}
\caption{Graphical representation of the coupling between the corner growth model and TASEP.}
\label{tasep}
\end{center}
\end{figure}

The connection between the corner growth and TASEP comes via the height function $h_t$ that evolves with the particle system as time $t$ progresses. It is a piecewise linear curve, differentiable in intervals $(x - 1/2, x+1/2), x \in \Z$. For each such interval the derivative of $h_t$ exists and it is constant $1$ or $-1$. If the height function has a positive slope on $(x - 1/2, x+1/2)$, it means that the corresponding site $x$ on the line is not occupied by a particle at time $t$. Otherwise if the edge of the height function has a negative slope in $(x - 1/2, x+1/2)$ it means that the corresponding site on the line is occupied . Particles jump to the right at random exponential times subject to the exclusion rule. With each step, the height function updates. In particular, note that the height function $h_t$ corresponds to the level curves of the last passage time. 
 (see Figure \ref{tasep}).

Therefore understanding the height function in the wedge  which is the level curve of the last passage time, 
is equivalent in studying the exclusion process for the particle system. This coupling was utilised for example in \cite{Sep-01-jsp, Rol-Tei-08, Geo-Kum-Sep-10} to obtain results about hydrodynamic limits of the particle current and density, together with the results for the last passage times. 

Hydrodynamic limits for spatially inhomogeneous conservative systems for different versions of inhomogeneities have been extensively studied \cite{ Cov-Rez-97, Bah-98, Rez-02, Che-08, Geo-Kum-Sep-10, Bah-18}. An example where the discontinuity is microscopic in nature is the slow bond problem. This TASEP model was introduced in \cite{Ja-92} and \cite{Ja-94}, in which particles jump at the same rate 1 everywhere on $\Z$ except at site zero where the jump happens at a slower rate than the other sites. Results regarding the hydrodynamic limits (and by extension the last passage times) were obtained in \cite{Sep-01-jsp} and finally in \cite{Bes-Sid-Sly-14} the full conjecture was proven that a slow bond will always affect the hydrodynamics. Recently, In \cite{Bo-Pe-17} a totally asymmetric particle with blockage with spatial inhomogeneities was studied and limiting Tracy-Widom laws were obtained. 

\subsection{Organisation of the paper}
In Section \ref{sec:model} we describe the main theorems. First we state the law of large numbers limit for the passage time \eqref{microLPT}. This is Theorem \ref{thm:1}. The limiting shape function, denoted by $\Gamma(x,y)$ comes in the form of a variational formula, where a functional is maximised over a set of suitable functions. Coninuity properties of $\Gamma$ are proved in Section \ref{sec:continuity}.  The proof of the law of large numbers is in Section \ref{sec:5}.

We then state results for two explicitly analysable examples. The first one is the shifted-two phase model with speed function \eqref{eq:cell0}; here we study properties of the shape and show analytically that there are flat edges, convexity-breaking and points of non differentiability for the shape function $\Gamma(x,y)$. The related proofs are in Section \ref{sec:2shifted}. 

The other example is the corner-inhomogeneous model with a speed function \eqref{eq:cfxy0}. Under some regularity conditions on $f$, we are able to study properties of the maximisers of the variational formula for the shape and how their behaviour depends on $f$. For example, depending on $f$ we may have points $(x,y)$ for which the macroscopic maximiser follows the axes. For both studied examples we have cases where macroscopic maximisers are not unique. The proofs for this model can be found in Section \ref{sec:cornerinho}.     

\subsection{Commonly used notation}
$\N$ denotes the set of natural numbers. $\Z$ is the set of integers and $\Z_+ =  \N \cup\{0\}$. $\R$ denotes the real numbers and $\R_+$ the non-negative reals.  If a variable $\tau$ follows the exponential distribution with parameter $r > 0$ this means $\P\{\tau > t \} = e^{-r t}$, in other words $r$ is the rate.

Bold-face letters (e.g.\ $\bf v$) indicate two dimensional vectors (e.g.\ ${\bf v} = (v_1, v_2)$). In particular letter $\bf x$ is reserved for 
denoting two-dimensional curves; 
often we write ${\bf x} (s) = (x_1(s), x_2(s))$ to
 emphasise that the curve is parametrised and seen as a function. Inequalities of vectors ${\bf v} \le {\bf w}$ or $(v_1, v_2) \le (w_1, w_2)$ 
means the inequality holds coordinate-wise. For a vector ${\bf v} = (v_1, v_2)$, we denote by $\fl{{\bf v}} =  (\fl{v_1}, \fl{v_2})$.

Without any special mention, when we write $\|\cdot\|$ we mean $\|\cdot\|_{\infty}$ unless explicitly referring to a different norm.
For any continuous function $g$ we denote its modulus of continuity by $\om_g$ and we assume 
\[ \| g(z_1) - g(z_2) \| _{\infty} \le \om_g(| z_1 -z_2 |_{\infty}). \] In the sequence we use the fact that 
$\om_g$ is continuous at 0 and that $\om_g(0) = 0$ without particular mention.

For any set $A \subseteq  \R_2^+$, we denote the multiplication  $nA = \{ (nx, ny): (x,y) \in \R^2_+ \}$ and the floor 
$\fl{nA} = \{ (\fl{nx}, \fl{ny}): (nx, ny) \in nA \}$. The topological interior of the set is denoted by ${\rm int}(A)$. 
For vectors $\bf v, w$, ${\bf v} \le {\bf w}$, we denote by $R({\bf v, w})$ the rectangle with south-west corner $\bf v$ and north-east corner ${\bf w}$. 

Letter $G$ is reserved for last passage times. Often we use the notation $G_A$ to denote the last passage time in the set $A$, which is the maximum weight that can be collected on up-right paths that lie in the set $A$. If no such paths exist, $G_A = 0$. 

\section{Model and results} 
\label{sec:model}

At this point, we state the technical conditions on $c(x,y)$ that we are imposing. There will be no special mention to these in the sequence, unless absolutely necessary. We explain why these assumptions are used after the statement of Theorem \ref{thm:1}.

We assume the speed function $c(x,y)$ satisfies the following two assumptions:

\begin{assumption}\label{ass:c}[Discontinuity curves of $c(x,y)$] 
	Function $c(x,y)$ is discontinuous on a (potentially) countable set of curves $H_c = \{ h_i\}_{i \in \mathcal I}$ that is locally finite in all the following properties 
		\begin{enumerate}
			\item $h_i$ is either a linear segment or strictly monotone.  
			\item If $h_i$ is not a vertical line segment, it can be viewed as a graph			
			\[
			h_i: [ z_i, w_i] = {\rm Dom}(h_i) \to \R,
			\] 
			\item If $h_i$ is strictly increasing, then 
			\begin{enumerate}
				\item $h_i$ is $C^1((z_i, w_i), \R)$. At the boundary points $z_i, w_i$ the derivative may take the value $\pm \infty, 0$.
				\item The equation $h_i'(s) = 0$ has finitely many solutions in $[ z_i, w_i]$. 
			\end{enumerate}
				\item If $h_i$ is strictly decreasing, then $h_i$ is continuous.						\end{enumerate} 
\end{assumption}				

The discontinuity curves $\{h_i\}_{i \in \mathcal I}$ separate $\R^2_+$ into open regions in which $c(x,y)$ is assumed continuous. The number of regions is finite in any compact set of $\R^2_+$.	Denote the set of regions by $\mathcal Q$. 

There are two types of points on these discontinuity curves,
\begin{enumerate}
	\item(Interior points) These are points ${\bf w}$ that belong on a single discontinuity curve $h_i$. For any point ${\bf w}$ of this form, we can find an $\e > 0$ so that $h_i$ partitions $ B( {\bf w}, \e) $ in to three disjoint sets,  $U_{\e, \bf w }$ (above $h_i$),  $L_{\e, \bf w }$ (below $h_i$) and 
$(h_i \cap B( {\bf w}, \e))$. 
	
	\item(Intersection/terminal points) These are points ${\bf w}$ that either belong on more than one discontinuity curve or they are terminal for $h_i$. There are finitely many of these points in any compact set.			
\end{enumerate}
				
\begin{assumption}\label{ass:c2}[Further properties of $c(x,y)$]		
\begin{enumerate}		
		\item $c(x,y)$ is continuous on any $Q \in \mathcal Q$,  lower-semicontinuous on $\R^2_+$, that further satisfies the following  stability assumption:  
		
		For every $i \in \mathcal I$ and interior point ${ \bf w } \in h_i$, there exists $\e = \e(i, {\bf w})>0$ so that for all ${\bf y} \in B( {\bf w}, \e) \cap h_i $ 
		there exists open set $Q_{i, \bf w} \in \{ L_{\e, \bf w }, U_{\e, \bf w }\}$, so that 
		for any sequence ${\bf z}_n \in Q_{i, \bf w} \cap B( {\bf w}, \e)$ with ${\bf z}_n \to {\bf y}$,
		\be \label{eq:calder2}
		 \lim_{{{\bf z}_n} \to {\bf y}} c({\bf z}) = c( {\bf y}). 
		\ee
		\item For any compact set $K \subset \R^2_+$, there exist two constants $r_{\text{\rm low}}^{(K)} > 0$ and $r_{\text{\rm high}}^{(K)} < \infty$, so that 
		\[
			r_{\text{\rm low}}^{(K)} \le c(x,y) \le r_{\text{\rm high}}^{(K)} , \quad \forall (x, y) \in K.
		\]
	\end{enumerate}
\end{assumption}

\begin{remark} Assumption \ref{ass:c2}, (1) gives by a standard compactness argument that if $c(x,y)$ is never continuous on $h_i$ then it must be that in a strip around $h_i$ the values of $c(x,y)$ on one of the incident regions is always smaller than the values in all other incident regions. This is consistent with assumption F3, equation (1.12) in \cite{Cal-15}. Lower semi-continuity of $c(x,y)$ implies that the limiting value in \eqref{eq:calder2} is the smallest of all possible limits on sequences that approach $y$. However, the assumption of  \cite{Cal-15} that $c(x,y)$ is (at least locally) Lipschitz is now removed. 
\end{remark}

Fix an $(x,y)$ in $\R^{2}_+$ and a speed function $c(\cdot, \cdot)$. Define the function 
$\Gamma_c(x,y)$ via the variational formula 
\begin{equation}\label{macroLPT}
\Gamma_c(x,y)=\sup_{\mathbf{x}(\cdot)\in\mathcal{H}(x,y)}\bigg\{\int_0^1\frac{\gamma(\mathbf{x}'(s))}{c(x_1(s),x_2(s))}ds\bigg\},
\end{equation}
where $\gamma(x,y)=(\sqrt{x}+\sqrt{y})^2$ is the last-passage constant in a homogeneous rate 1 environment, $\mathbf{x}(s)=(x_1(s),x_2(s))$ denotes a path in $\mathbb{R}^2$ and set 
\begin{align*}
\mathcal{H}(x,y) &=\{{\bf x}\in C([0,1],\R^2_+):\mathbf{x} \text{ is piecewise }C^1,\mathbf{x}(0)=(0,0), \mathbf{x}(1)=(x,y), \\
& \phantom{xxxxxxxxxxxxxxxxxxxxxxxx}\mathbf{x}'(s)\in \R^2_+ \text{ wherever the derivative is defined}\}.
\end{align*}
When the speed function $c(x,y) = r$ constant, we can immediately compute 
\begin{align*}
\Gamma_r(x,y) &= \sup_{\mathbf{x}(\cdot)\in\mathcal{H}(x,y)} \int_0^1\frac{\gamma(\mathbf{x}'(s))}{c(x_1(s),x_2(s))}ds = \frac{1}{r} \sup_{\mathbf{x}(\cdot)\in\mathcal{H}(x,y)} \int_0^1\gamma(\mathbf{x}'(s))\, ds\\
& \le \frac{1}{r} \sup_{\mathbf{x}(\cdot)\in\mathcal{H}(x,y)} \gamma\Big( \int_0^1x_1'(s)\ ds,  \int_0^1 x_2'(s) \, ds\Big), \quad \text{by Jensen's inequality since $\gamma$ is concave}\\
&= \frac{1}{r} \gamma(x,y) \le \Gamma_r(x,y).
\end{align*}
The last inequality follows from the fact that the straight line from 0 to $(x,y)$ is an admissible candidate maximiser for \eqref{macroLPT}. The calculation shows two things that we use freely in the sequence, namely
\begin{enumerate} 
	\item Straight lines are optimisers of \eqref{macroLPT} in homogeneous (constant) regions of $c(x,y)$. In fact, because $\gamma$ is strictly concave, it is easy to show that the straight line will be the unique maximiser. We refer to this fact as `Jensen's inequality' in the sequence.
	\item $\Gamma_r(x,y)$ corresponds to the limiting shape function for last passage times in a homogeneous Exp$(r)$ environment. 
\end{enumerate}

Two more properties of $\Gamma_c$ can be immediately obtained:
\begin{enumerate}[(1)]
	\item (Independence from parametrization) For any $c > 0$, $\gamma(cx,cy) = c \gamma(x,y)$ so the value of the integral 
	\be\label{eq:I}
		I({\bf x}) = \int_{0}^1 \frac{\gamma(\mathbf{x}'(s))}{c(x_1(s),x_2(s))}ds
	\ee 
	is independent of the parametrisation we choose for the curve ${\bf x}$. 
	\item (Superadditivity) Define $\Gamma_c(x,y) : =  \Gamma_c((0,0), (x,y))$ and similarly define $\Gamma_c$ 
	from any starting point $(a,b)$ to any terminal point $(x,y)$, $(x,y) \ge (a,b)$ by 
	 \begin{equation}\label{macroLPT2}
\Gamma_c((a,b),(x,y))=\sup_{\mathbf{x}(\cdot)\in\mathcal{H}((a,b), (x,y))}\bigg\{\int_0^1\frac{\gamma(\mathbf{x}'(s))}{c(x_1(s),x_2(s))}ds\bigg\},
\end{equation}
where 
\begin{align*}
\mathcal{H}((a,b),(x,y)) &=\{{\bf x}\in C([0,1],\R^2_+):\mathbf{x} \text{ is piecewise }C^1,\mathbf{x}(0)=(a,b), \mathbf{x}(1)=(x,y), \\
& \phantom{xxxxxxxxxxxxxxxxxx}\mathbf{x}'(s)\in \R^2_+ \text{ wherever the derivative is defined}\}.
\end{align*}
Then, for any $(a,b) \le (z,w) \le (x,y)$ we have 
\be \label{eq:subadditive}
\Gamma_c((a,b), (x,y)) \ge \Gamma_c((a,b), (z,w)) + \Gamma_c((z,w), (x,y)).
\ee
\end{enumerate}

In this respect, function $\Gamma_c$ behaves like a `macroscopic last passage time' and the first theorem shows that it is a continuous function.

\begin{theorem}\label{cor:conti}[Continuity of $\Gamma$.] Let $c(x,y)$ satisfy Assumptions \ref{ass:c} and  \ref{ass:c2}. Fix $(a, b)$ and $(x,y) \in \R^2_+$. 
For any $\e > 0$ there exists a $\delta_0= \delta_0(\e)> 0$ so that 
for all $\delta_1, \delta_2, \delta_3, \delta_4 \in (-\delta_0, \delta_0)$, we have 
\be
| \Gamma_c((a + \delta_1, b + \delta_2), (x + \delta_3, y+ \delta_4)) - \Gamma_c((a,b),(x,y)) | < \e.
\ee
\end{theorem}

In the next theorem we obtain $\Gamma_c$ in \eqref{macroLPT} as the law of large number of the microscopic last passage time  \eqref{microLPT}.
\begin{theorem}\label{thm:1}
Recall \eqref{microLPT}. Let $c(x,y)$ a macroscopic speed function which satisfies Assumption \ref{ass:c}, and let $(x,y) \in \R^2_+$. Then we have the scaling limit
\begin{equation}\label{eq:6}
\lim_{n\to\infty}n^{-1}G^{(n)}_{\lfloor nx\rfloor,\lfloor ny \rfloor}=\Gamma_c(x,y)\quad \P-\text{a.s.}
\end{equation}
\end{theorem}

\begin{remark}[The conditions on the discontinuity curves] 
In \cite{Cal-15} the discontinuity curves are assumed strictly monotone, outside of compact set. As such, when viewed as graphs of continuous functions, they are differentiable almost everywhere. This is more general than the piecewise $C^1$ condition in Assumption   \ref{ass:c} 3-(a). In our case we cannot relax the piecewise $C^1$ assumption further; in Example \ref{ex:x2} we prove that for a certain speed function $c(x,y)$ 
the maximizing macroscopic path actually follows the discontinuity curve of $c(x,y)$ on a set of positive measure and the set of $\mathcal H$ contains only piecewise $C^1$ paths. 

We expect that under Assumptions \ref{ass:c} and \ref{ass:c2} $\Gamma_c(x,y)$ is in fact a maximum and not a supremum. 
\end{remark}

We use Theorem \ref{thm:1} to analyse two examples.
 
 \subsection{ The shifted two-phase model.}
The first one is the shifted two-phase model. It is a generalisation of the one provided in \cite{Geo-Kum-Sep-10}. We want to study an explicit description of the limit shape function for a two-phase corner growth model with a discontinuity of the speed function along the line $y=x-\lambda$. We assume 
$\lambda\in\bR_{+}$. For a fixed $r \in (0,1)$ we use the macroscopic speed function $c_\ell(s,t)$ on $\bR^2_+$ defined as 
\be\label{eq:cell}
c_{\ell}(x,y) = 	\begin{cases}
			1, & \text{if } y>x-\lambda,\\
			r, & \text{if } y\le x-\lambda.
		\end{cases} 
\ee
Subscript $\ell$ is to remind the reader that the small rate is \textit{lower} than the discontinuity line, i.e. $r < 1$ in this example. Since the speed function only takes two values, the set of optimal macroscopic paths from the origin to $(x,y)$ are piecewise linear paths.

\begin{theorem} \label{thm:2shifted} Let $c_{\ell}(x,y)$ as in \eqref{eq:cell}.
There exist explicitly computable functions $A(r), D(r)$ (see equation \eqref{eq:AD}) and some optimal point $a^*> \lambda$ so that for any $(x,y)\in\bR^2_+$ the limiting shape function is given by 
\[
\Gamma_{c_{\ell}}(x,y) = 
	\begin{cases}
	\gamma(x,y), &\text{ if } y\geq L(x,y),\\
	I(x,y)	, &\text{ if }  x-\lambda\leq y\leq L(x,y),\\	
	\gamma(a^*,a^*-\lambda)+r^{-1}\gamma(x-a^*,y-a^*+\lambda), &\text{ if }  y<x-\lambda,
	\end{cases}
\] 
where $I(x,y)$ is a linear section of $\Gamma_{c_{\ell}}(x,y)$, given by 
\[ 
I(x,y) = (1+A(r)) x + \Big( 1  +\frac{1}{A(r)}\Big)y -D(r),\] 
and $L(x,y)$ is described by the equation
\[ 
\Big(A(r)x-\frac{1}{A(r)}y\Big)^2-2 D(r) \Big(A(r)x+\frac{1}{A(r)}y\Big)+D(r)^2=0.
\] 
\end{theorem}

\subsection{The corner-discontinuous model.}
The other example is what we call the \textit{corner-discontinuous model}. We start with a $C^2$ convex decreasing function $f: [0, a_0] \longrightarrow [0, b_0]$ where
$f(0) = b_0 >0$ and $f(a_0) = 0$. Then we define the speed function 
\be \label{eq:cfxy}
c_{f}(x,y) = 
\begin{cases}
1, & f(x) > y,\\
r, & f(x) < y,\\
1\wedge r, & f(x) =y.
\end{cases}
\ee
In words, after a bounded region of rate 1 delineated by $f$ and the coordinate axes, the rate becomes $r$. Computing analytically the shape function $\Gamma_{c_f}(x,y)$ is challenging; it depends on properties of the function $f$. When $f$ takes the specific form
\[
f(x) = (1-\sqrt{x})^2, \quad x \in [0,1],
\] 
we can explicitly identify the shape function in Example \ref{ex:solvable} and the macroscopic maximisers of \eqref{macroLPT} are straight paths from $(0,0)$ to $(x,y)$, despite the discontinuity.  

Changing the function $f$, properties of macroscopic maximisers can be obtained. From the fact that $c(x,y)$ is piecewise constant, macroscopic maximisers of \eqref{macroLPT} exist and are piecewise linear segments, one in each of the two constant regions. 

For each point $(x,y)$ in the $r$-region variational formula will be maximised by either a piecewise linear path that crosses $f$ or by a piecewise linear path, with initial segment on one of the coordinate axes. 

\begin{definition}[Types of maximisers] There are two types of potential maximisers of \eqref{macroLPT} under speed function \eqref{eq:cfxy}:
\begin{enumerate}
\item[Type {\rm C}:] We say that the maximiser is of  \textit{crossing type}  when it crosses the function $f$ at some optimal \textit{ crossing point} $(a, f(a))$, $(0 < a <a_0)$ which depends on $(x,y)$. 
\item[Type {\rm B}:] We say that the maximiser is of \textit{ boundary type}, when the first linear segment of it follows one of the coordinate axes. 
\end{enumerate}
\end{definition}
Note that for $(x,y) \in (0, a_0)\times(0, f(0))$ we cannot have type B maximisers, and for $(x,y)$ in the 1-region the maximiser must be the straight line from $(0,0)$. 
Based on this definition we define 
\[ 
 R_{0, f(0)} = \{ (x,y) \in \R^2_+: \text{maximiser of \eqref{macroLPT} is of type B and goes through $(0,f(0)$} \}.
\]
Similarly define  $ R_{a_0, 0}$ for which maximisers go through the horizontal axis. We would like to know when $ R_{0, f(0)}$ have non-empty interior. 
As it turns out, this only depends on properties of the function $f$ and the value of $r$. 

A few definitions before stating the result. 
First we define a function $m_2$ of $a \in (0, a_0)$ by 
\be \label{eq:m_2f}
m_2(a) = \frac{4}{\Big(-\frac{1}{f'(a)}-1+D+\sqrt{\Big(-\frac{1}{f'(a)}-1+D\Big)^2-4\frac{1}{f'(a)}}\Big)^2},
\ee
where
\be\label{eq:dioff}
D = D_a = r \Big(1 + \sqrt{\frac{f(a)}{a}}\Big)\Big(\sqrt{\frac{a}{f(a)}} + \frac{1}{f'(a)}\Big).
\ee
In Section \ref{sec:cornerinho} we prove that for any points $(x,y) \in$ int$(R^2_+)$ which have a candidate maximiser of type C, i.e.\ for any point $(x,y)$ for which there exists at least one admissible crossing point $(a_{x,y}, f(a_{x,y}))$ with $0 <a_{x,y}<a_0$, the slope $m_2 = m_2(a_{x,y})$ of the second linear segment must satisfy the equation
\[
\frac{y - f(a_{x,y})}{x -a_{x,y}} = m_2(a_{x,y}). 
\]
It is not necessary that for each $(x,y)$ a unique $a_{x,y}$ will satisfy the equation above, but it will be true that $a_{x,y} < x$ and $f(a_{x,y}) < y$ (see Lemma \ref{lem:dense}).

Furthermore, we define  
\[
\alpha_0= \inf\big\{ s :  \varlimsup_{a \searrow 0} a^{s}|f'(a)| = 0 \big\}  \quad \text{and} \quad \alpha_\infty =  \sup\big\{ s :  \varlimsup_{a \searrow 0} a^{s}|f'(a)| =  \infty \big\}.
\]
Check that $\alpha_0 \ge \alpha_\infty$. The two values coincide when either of them is non-zero and finite. To check that the two give the same $\alpha$, reason by way of contradiction; Assume that there exists a $\gamma$ so that 
\[
\sup\big\{ s :  \varlimsup_{a \searrow 0} a^{s}|f'(a)| =  \infty \big\}< \gamma < \inf\big\{ s :  \varlimsup_{a \searrow 0} a^{s}|f'(a)| = 0 \big\}.
\]
Then $0 < \varlimsup_{a \to 0}  a^{\gamma}|f'(a)| < \infty$. Then for any $\e > 0$ small enough, we will have that the same condition is true for $\gamma + \e$ and that is a contradiction.  

These let us define the order of growth of $f'$ as 
\be\label{eq:order}
\alpha = 
\begin{cases}
 \inf\big\{ s :  \varlimsup_{a \searrow 0} a^{s}|f'(a)| = 0 \big\} = \sup\big\{ s :  \varlimsup_{a \searrow 0} a^{s}|f'(a)| =  \infty \big\} & \text{ if } \alpha_0 \in (0, \infty)\\
 \infty, &\text{ if } \alpha_\infty = \infty\\
 0, &\text{ if } \alpha_0 = 0.
\end{cases}
\ee

When the order of growth of $f'$ is specified to be $\alpha$, we further define
\be\label{eq:inforder}
0 \le c_{\alpha}^{(-)} =  \varliminf_{a \to 0} a^{\alpha}|f'(a)| \le \varlimsup_{a \to 0} a^{\alpha}|f'(a)| = c_{\alpha}^{(+)} \le \infty.
\ee
Similarly we define
\be\label{eq:order2}
\beta = 
\begin{cases}
\beta_0 =\sup\Big\{ s : \varliminf_{a \nearrow a_0} \frac{|f'(a)|}{(a_0 -a)^s} = 0 \Big\} = \inf\Big\{ s :  \varliminf_{a \nearrow a_0} \frac{|f'(a)|}{(a_0 -a)^s} = \infty \Big\} =\beta_{\infty} & \text{ if } \beta_0 \in (0, \infty),\\
 0, &\text{ if } \beta_\infty = 0,\\
 \infty, &\text{ if } \beta_0 = \infty.

\end{cases}
\ee
Again, at $\beta$, we similarly define $\eta_{\beta}^{(-)}, \eta_{\beta}^{(+)}$ by 
\be\label{eq:inforder2}
0 \le \eta_{\beta}^{(-)} =  \varliminf_{a \to a_0}\frac{|f'(a)|}{(a_0 -a)^\beta}   \le \varlimsup_{a \to a_0}  \frac{|f'(a)|}{(a_0 -a)^\beta} = \eta_{\beta}^{(+)} \le \infty.
\ee
Now we are ready to state a theorem for this model.

\begin{theorem} \label{thm:regionLM} Let $c_f(x,y)$ be given by \eqref{eq:cfxy}, for some $C^2((0, a_0),(0, f(0)))$ convex function $f$.  
Assume either that $\alpha \neq 1/2$ or that $\alpha = 1/2$ and $ r \notin \Big[ \frac{c_{1/2}^{(+)}}{c_{1/2}^{(+)}-\sqrt{f(0)}}, \frac{c_{1/2}^{(-)}}{c_{1/2}^{(-)} - \sqrt{f(0)}} \Big]$.
 Then 
the following are equivalent:
\begin{enumerate}
\item $\varlimsup_{a \to 0} m_2(a) = + \infty$, 
\item $R_{0, f(0)} = \{0\}\times [f(0), \infty)$. 
\end{enumerate}
Similarly, assume either that $\beta \neq 1/2$ or that  $\beta = 1/2$ and $ r \notin \Big[ \frac{1}{1- \eta_{1/2}^{(-)}\sqrt{a_0}}, \frac{1}{1- \eta_{1/2}^{(+)}\sqrt{a_0}} \Big]$ . 
Then 
the following are equivalent:
\begin{enumerate}
\item $\varliminf_{a \to a_0} m_2(a) = 0$, 
\item $R_{a_0, 0} =  [a_0, \infty)\times\{0\}$. 
\end{enumerate}
\end{theorem}

The situation when $\alpha = 1/2$  and $ r \in \Big[ \frac{c_{1/2}^{(+)}}{c_{1/2}^{(+)}-\sqrt{f(0)}}, \frac{c_{1/2}^{(-)}}{c_{1/2}^{(-)} - \sqrt{f(0)}} \Big)$  or respectively, $\beta = 1/2$ and $r \in \Big[ \frac{1}{1- \eta_{1/2}^{(-)}\sqrt{a_0}}, \frac{1}{1- \eta_{1/2}^{(+)}\sqrt{a_0}} \Big)$, is a bit more delicate. While Theorem \ref{thm:regionLM} is valid when we know the behaviour of $m_2(a)$ as a generic function of $a$, when $\alpha = 1/2$ we want the behaviour of $m_2(a)$ on \textit{crossing points:}

\begin{definition}\label{def:cross}(Crossing points) A point $(a, f(a))$ is a crossing point if and only if there exists $(x,y) \in \R^2_+$ so that a maximiser in \eqref{macroLPT} for 
$\Gamma_{c_f}(x,y)$ is the piecewise linear segment $(0,0) \to (a, f(a)) \to (x,y)$. 
\end{definition}

\begin{theorem}\label{thm:gap} Assume 
$\alpha = 1/2,  r \in \Big[ \frac{c_{1/2}^{(+)}}{c_{1/2}^{(+)}-\sqrt{f(0)}}, \frac{c_{1/2}^{(-)}}{c_{1/2}^{(-)} - \sqrt{f(0)}} \Big)$. Then the following are equivalent:
\begin{enumerate}
 \item There exists a sequence of crossing points $(a_k, f(a_k))$ so that $a_k \to 0$, $m_2(a_k) \to \infty$ and $\displaystyle\varliminf_{a_k \to 0}a_k^{1/2}|f'(a_k)| <\frac{r\sqrt{f(0)}}{r-1}$. 
 \item  $R_{0, f(0)} = \{0\}\times [f(0), \infty)$. 
\end{enumerate}
Similarly, assume that $\beta = 1/2$ and $r \in \Big[ \frac{1}{1- \eta_{1/2}^{(-)}\sqrt{a_0}}, \frac{1}{1- \eta_{1/2}^{(+)}\sqrt{a_0}} \Big)$. Then the following are equivalent:
\begin{enumerate}
 \item There exists a sequence of crossing points $(a_k, f(a_k))$ so that $a_k \to a_0$, $m_2(a_k) \to 0$ and $\displaystyle\varliminf_{a_k \to a_0}a_k^{1/2}|f'(a_k)| <\frac{r-1}{r\sqrt{a_0}}$. 
 \item  $R_{a_0, 0} = [a_0, \infty)\times\{0\}$. 
\end{enumerate}

\end{theorem}

We closely look at the case for which $\alpha = 1/2$ and $c^{(-)}_{1/2} = \frac{r}{r-1} \sqrt{f(0)}$ or $\eta^{(+)}_{1/2} = \frac{r-1}{r \sqrt{a_0}}$ and show that it is a phase transition; depending on how the limits are approached it may or may not lead to non-degenerate regions for type B maximisers. We include the details that justify this statement in Section \ref{sec:cornerinho}, Proposition \ref{prop:pt2}.

Finally, we obtain a partition of the parameter space $(\alpha, r)$ where we can a priori identify whether $\varlimsup_{a\to 0} m_2(a) = \infty$ or $\varliminf_{a\to a_0} m_2(a) = 0$ as the content of the next proposition. 

\begin{proposition}\label{prop:EX2}
 Let $\alpha, \beta$ and $c_{\alpha}^{(-)}, \eta_{\beta}^{(+)}$ as defined in equations \eqref{eq:order},\eqref{eq:order2}, \eqref{eq:inforder},\eqref{eq:inforder2}  and let $m_2(a)$ be given by equation \eqref{eq:m_2f}. Then, for $(\alpha, r) \in \R^2_+$, 
 \begin{enumerate}
 \item  For $\lim_{a\to 0} f'(a) = -\infty$, we have
 \be
 \varlimsup_{a\to 0} m_2(a) = \left\{
\begin{aligned}
&\frac{1}{(r-1)^2}\qquad&&\text{if } \alpha>\frac{1}{2}\text{ and } r>1, \\
&\frac{1}{\Big(r - 1-\frac{r\sqrt{f(0)}}{c_{1/2}^{(-)}}\Big)^2}\qquad&&\text{if } \alpha = \frac{1}{2},\, c_{1/2}^{(-)}> \sqrt{f(0)},\, r > \frac{c_{1/2}^{(-)}}{c_{1/2}^{(-)}-\sqrt{f(0)}},\\
&+\infty\qquad&&\text{ otherwise. }
\end{aligned}
\right.
\ee
\item For  $\varlimsup_{a\to 0} f'(a)= - c$
\be
\varlimsup_{a\to 0} m_2(a) = +\infty.
\ee 
\end{enumerate}
By interchanging the role of the coordinates, we can obtain the corresponding results for when $a \to a_0$, namely 
\begin{enumerate}
 \item  For $\lim_{a\to a_0} f'(a) = 0$, we have
 \be
 \varliminf_{a\to a_0} m_2(a) = \left\{
\begin{aligned}
&(r-1)^2\qquad&&\text{if } \beta>\frac{1}{2}\text{ and } r>1, \\
&{\big(r - 1- r \eta_{1/2}^{(+)} \sqrt{a_0} \big)^2}\qquad&&\text{if } \beta = \frac{1}{2},\, \eta_{1/2}^{(+)}< a_0^{-1/2},\, r > \frac{1}{1 - \eta_{1/2}^{(+)}\sqrt{a_0}},\\
&0\qquad&&\text{ otherwise. }
\end{aligned}
\right.
\ee
\item For  $\varliminf_{a\to a_0} f'(a)= - c$
\be
\varliminf_{a\to 0} m_2(a) = 0.
\ee 
\end{enumerate} 
 
\end{proposition}
\begin{figure}
\begin{subfigure}{.5\linewidth}
\centering
\begin{tikzpicture}[scale=1.5, >=latex, xscale=1.2]
 \fill[color=nicos-red!30](0,0)--(2.7,0)--(2.7,1)--(0,1)--(0,0);
  \fill[color=nicos-red!30](0,1)--(1,1)--(1,2.7)--(0,2.7)--(0,1);
\fill[color=blue!30](1,1)--(2.7,1)--(2.7,2.7)--(1,2.7)--(1,1);
   \draw[->,thick](0,0)--(3,0)node[right]{$\alpha$};
    \draw[->,thick](0,0)--(0,3)node[above]{$r$};
\draw[red](1,1)--(1,2.7);
\draw[my-green, ->](1,2)--(1.2, 2.9); 
 \draw[red](0,1)node[left,black]{$1$};
 \draw[red](1,1)--(2.7,1);
\draw(0.8,0.5)node{$\varlimsup m_2\to\infty$};
\draw(2,2)node{$\varlimsup m_2\to\frac{1}{(r-1)^2}$};
\draw[line width=2pt, my-green](1,1.5)--(1,2.7);
\fill[red!30](1,1.5)circle(1mm);
\draw[red](1,1.5)circle(1mm);
\draw(2,3)node{$\varlimsup m_2\to\frac{1}{(r(1- \sqrt{f(0)}/c_{1/2}^{(-)})-1)^2}$};
\draw(1, -0.2)node{$\frac{1}{2}$};
\draw(-0.5,1.5)node{$\frac{c_{1/2}^{(-)}}{c_{1/2}^{(-)}-\sqrt{f(0)}}$};
\draw[dashed, my-green](0,1.5)--(0.9,1.5);
\end{tikzpicture}
\caption{Behaviour for $\varlimsup m_2(a)$ when $\alpha$ and\\$r$ vary, when $a\to0$ and $f'(0)\to-\infty$, when \\$c_{1/2}^{(-)} > \sqrt{f(0)}$.}
\label{fig:sub1}
\end{subfigure}%
\begin{subfigure}{.5\linewidth}
\centering
\begin{tikzpicture}[scale=1.5, >=latex, xscale=1.2]
 \fill[color=nicos-red!30](0,0)--(2.7,0)--(2.7,1)--(0,1)--(0,0);
  \fill[color=nicos-red!30](0,1)--(1,1)--(1,2.7)--(0,2.7)--(0,1);
\fill[color=blue!30](1,1)--(2.7,1)--(2.7,2.7)--(1,2.7)--(1,1);
   \draw[->,thick](0,0)--(3,0)node[right]{$\beta$};
    \draw[->,thick](0,0)--(0,3)node[above]{$r$};
\draw[red](1,1)--(1,2.7);
\draw[my-green, ->](1,2)--(1.2, 2.9); 
 \draw[red](0,1)node[left,black]{$1$};
 \draw[red](1,1)--(2.7,1);
\draw(0.8,0.5)node{$\varliminf m_2\to0$};
\draw(2,2)node{$\varliminf m_2\to(r-1)^2$};
\draw[line width=2pt, my-green](1,1.5)--(1,2.7);
\fill[red!30](1,1.5)circle(1mm);
\draw[red](1,1.5)circle(1mm);
\draw(2,3)node{$\varliminf m_2\to\big(r - 1- r \eta_{1/2}^{(+)} \sqrt{a_0} \big)^2$};
\draw(1, -0.2)node{$\frac{1}{2}$};
\draw(-0.5,1.5)node{$ \frac{1}{1 - \sqrt{a_0}\eta_{1/2}^{(+)}}$};
\draw[dashed, my-green](0,1.5)--(0.9,1.5);
\end{tikzpicture}

\caption{Behaviour for $\varliminf m_2(a)$ when $\beta$ and\\ $r$ vary, when $a\to a_0$ and $f'(a_0)\to 0$, when\\ $\eta_{1/2}^{(+)} <a_0^{-1/2}$.}
\label{fig:sub2}
\end{subfigure}

\end{figure}

Proposition \ref{prop:EX2} in conjunction with Theorem \ref{thm:regionLM} classifies the cases for which non-trivial maximisers of type B exist when $\alpha \neq 1/2$. Theorem \ref{thm:gap} is weaker, so without further analysis, the proposition can only guarantee trivial type B maximisers from the vertical axis when $\alpha=1/2$ and $ r \notin \Big[ \frac{1}{1- \eta_{1/2}^{(-)}\sqrt{a_0}}, \frac{1}{1- \eta_{1/2}^{(+)}\sqrt{a_0}} \Big]$. When $\alpha =1/2$ and $ r \in \Big[ \frac{1}{1- \eta_{1/2}^{(-)}\sqrt{a_0}}, \frac{1}{1- \eta_{1/2}^{(+)}\sqrt{a_0}} \Big)$ one needs to verify that the optimal slopes tend to $+\infty$.

We showcase the above results by performing some Monte Carlo simulations to show the maximal paths in different cases. 
For all simulations we considered the curve $y= f(x)$ to be 
\[
f(x) = (c - x^{b/k})^k,
\]
and we varied the parameters $b,c,k$ with $b<k$. See Figure \ref{fig:sim}.
\begin{figure}[h!]
    \centering
    \begin{subfigure}[t]{0.3\textwidth}
        \centering
        \includegraphics[height=1.2in]{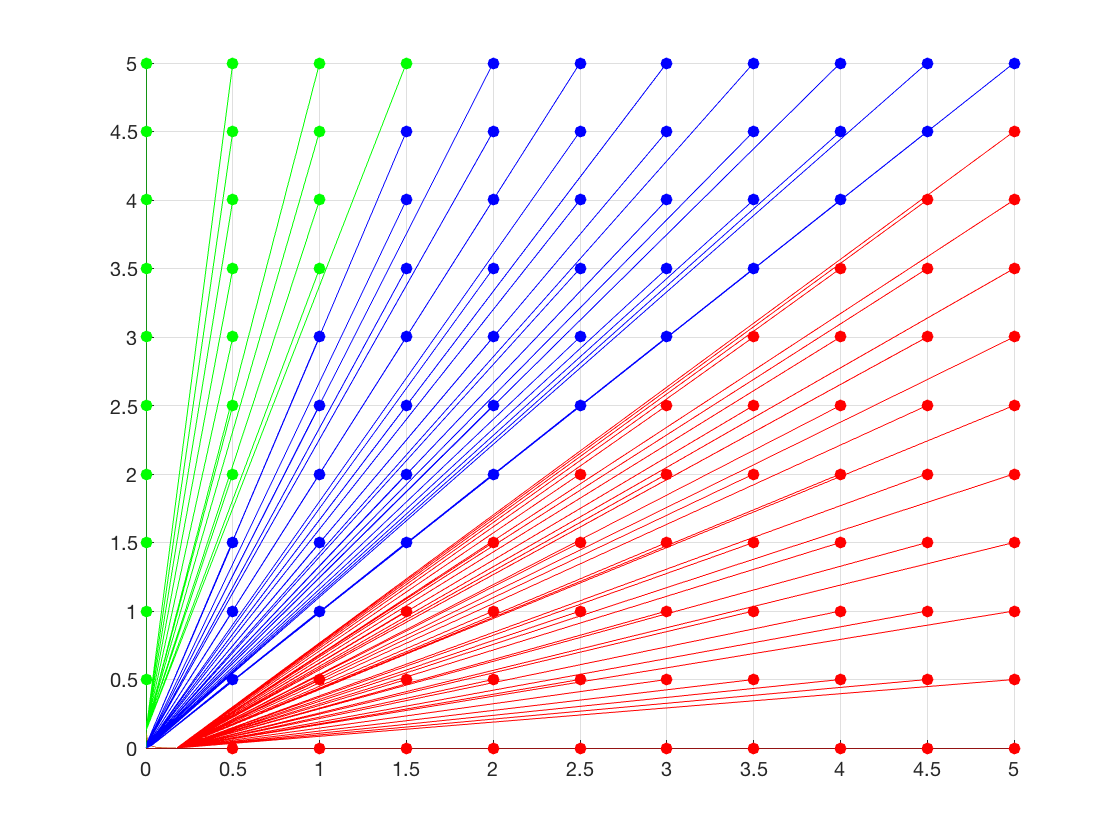}
        \caption{$c=0.5, b=1.2, k=3, r=3$.}
    \end{subfigure}
    ~
    \begin{subfigure}[t]{0.3\textwidth}
        \centering
        \includegraphics[height=1.2in]{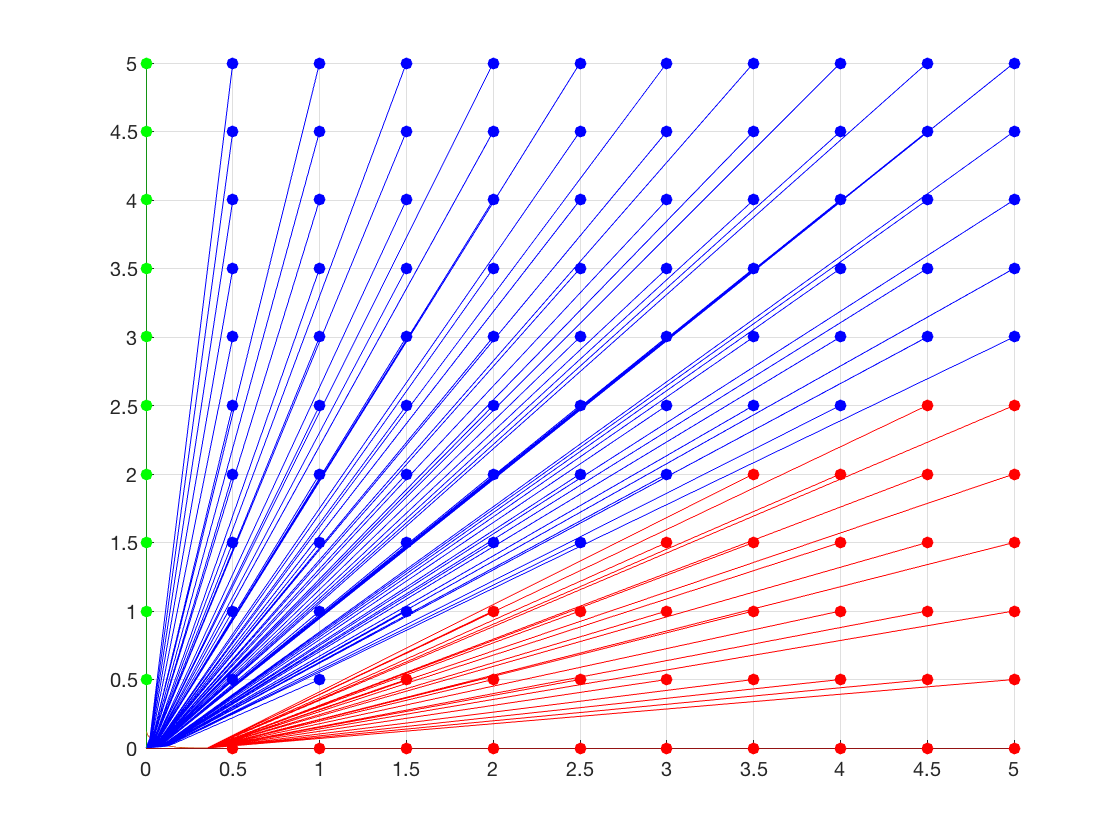}
        \caption{$c=0.5, b=2, k=3, r=3$.}
    \end{subfigure}
    ~    
    \begin{subfigure}[t]{0.3\textwidth}
        \centering
        \includegraphics[height=1.2in]{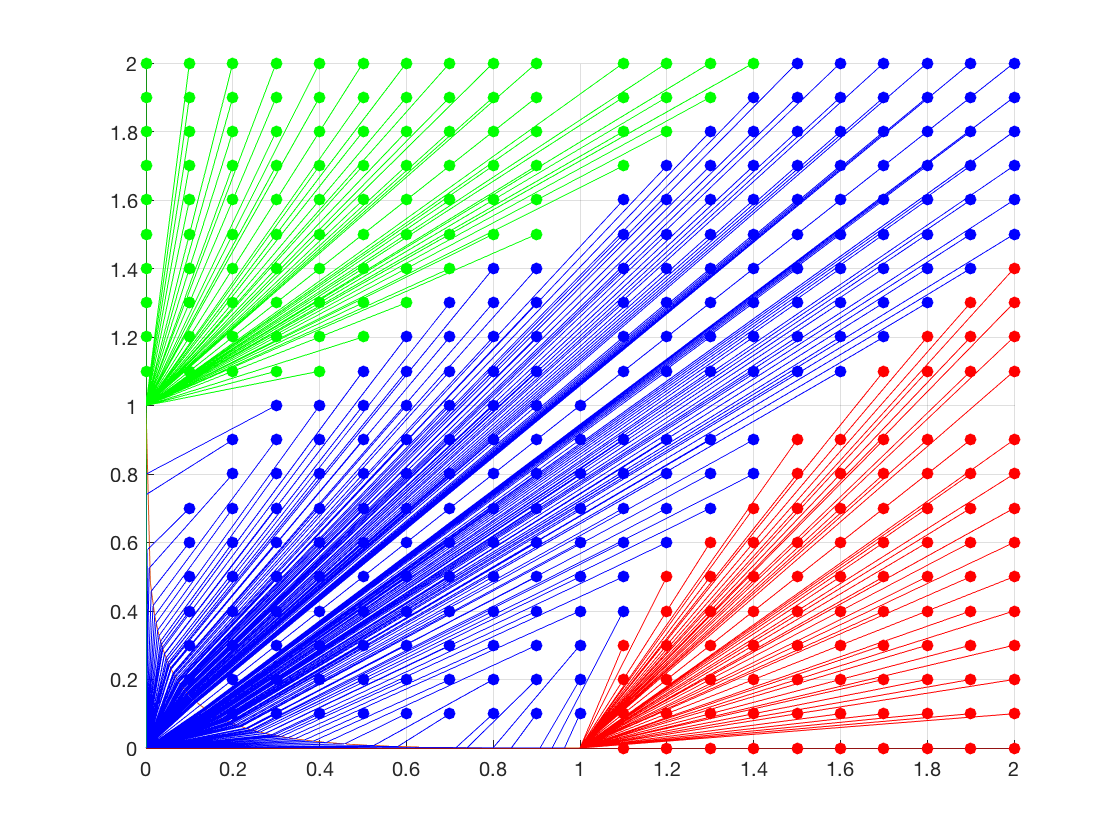}
        \caption{$c=1, b=1, k=3, r=3$.}
    \end{subfigure}
    
    \centering
    \begin{subfigure}[t]{0.3\textwidth}
        \centering
         \includegraphics[height=1.2in]{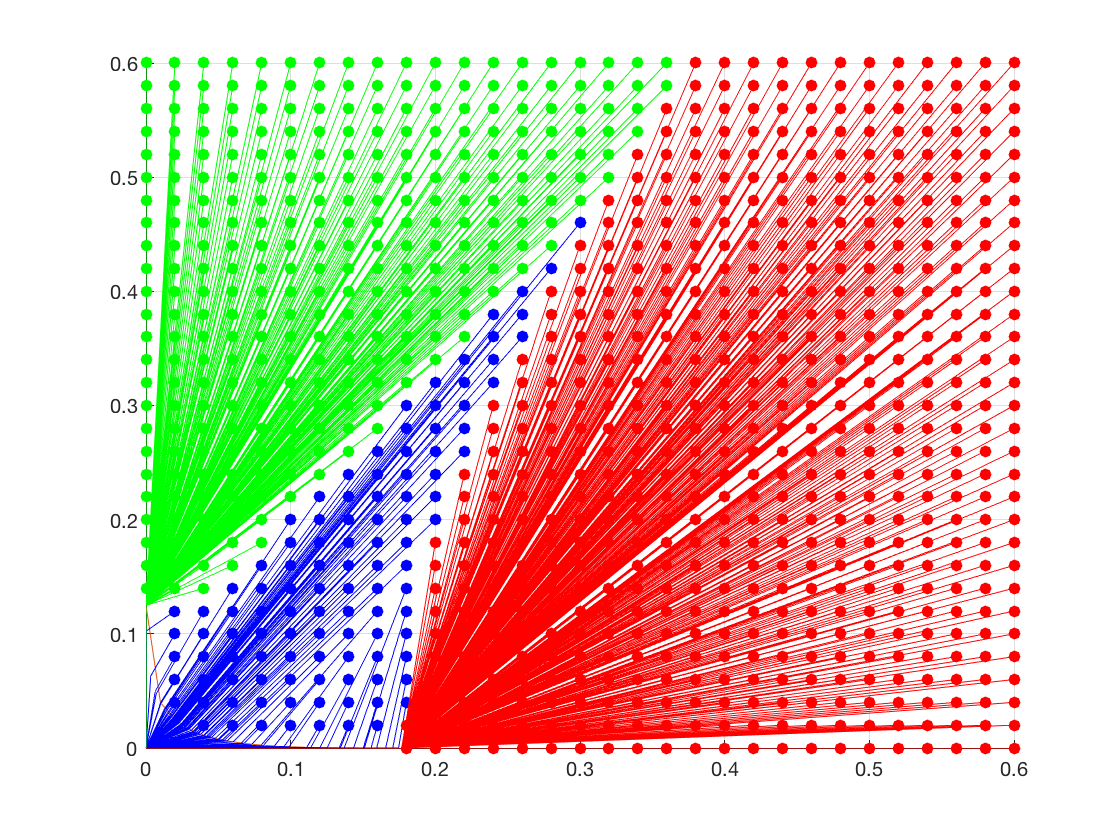}
 \caption{$c=0.5, b=1.2, k=3,  r=4$.}
    \end{subfigure}
     ~
    \begin{subfigure}[t]{0.3\textwidth}
        \centering
        \includegraphics[height=1.2in]{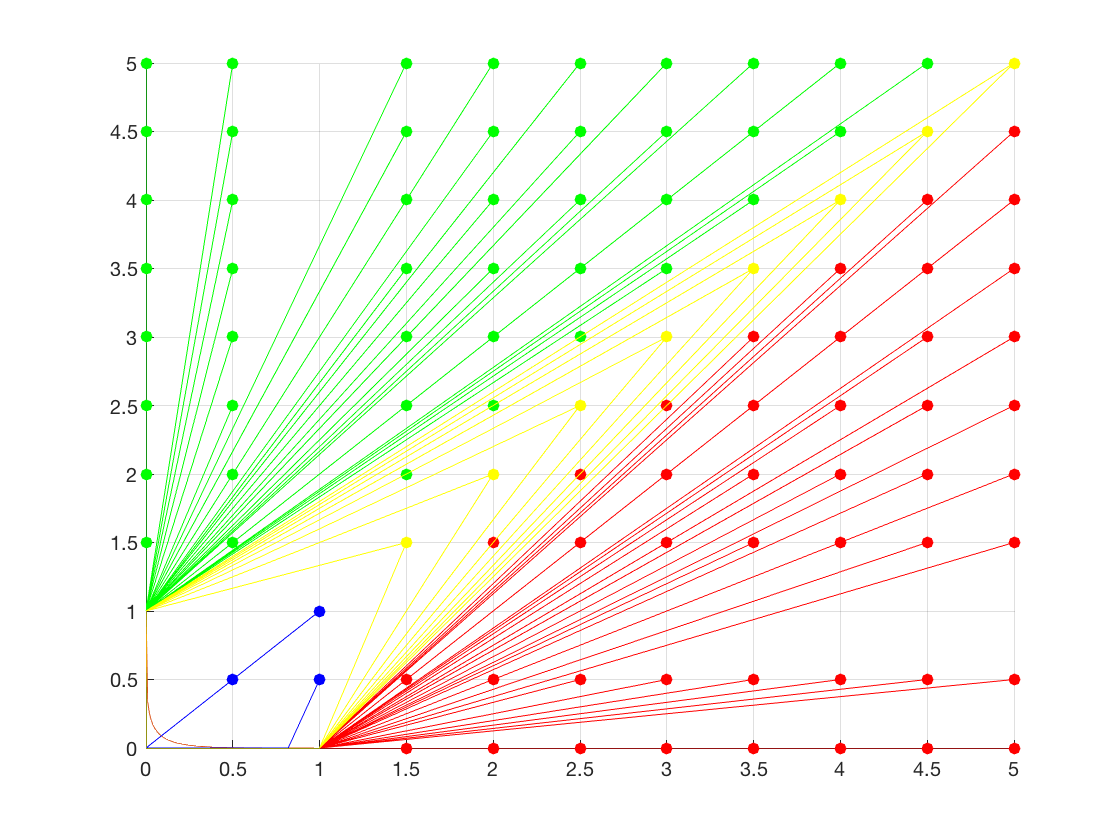}
        \caption{$c=1, b=1, k=3.5, r=3$.}
    \end{subfigure}%
    ~ 
    \begin{subfigure}[t]{0.3\textwidth}
        \centering
        \includegraphics[height=1.2in]{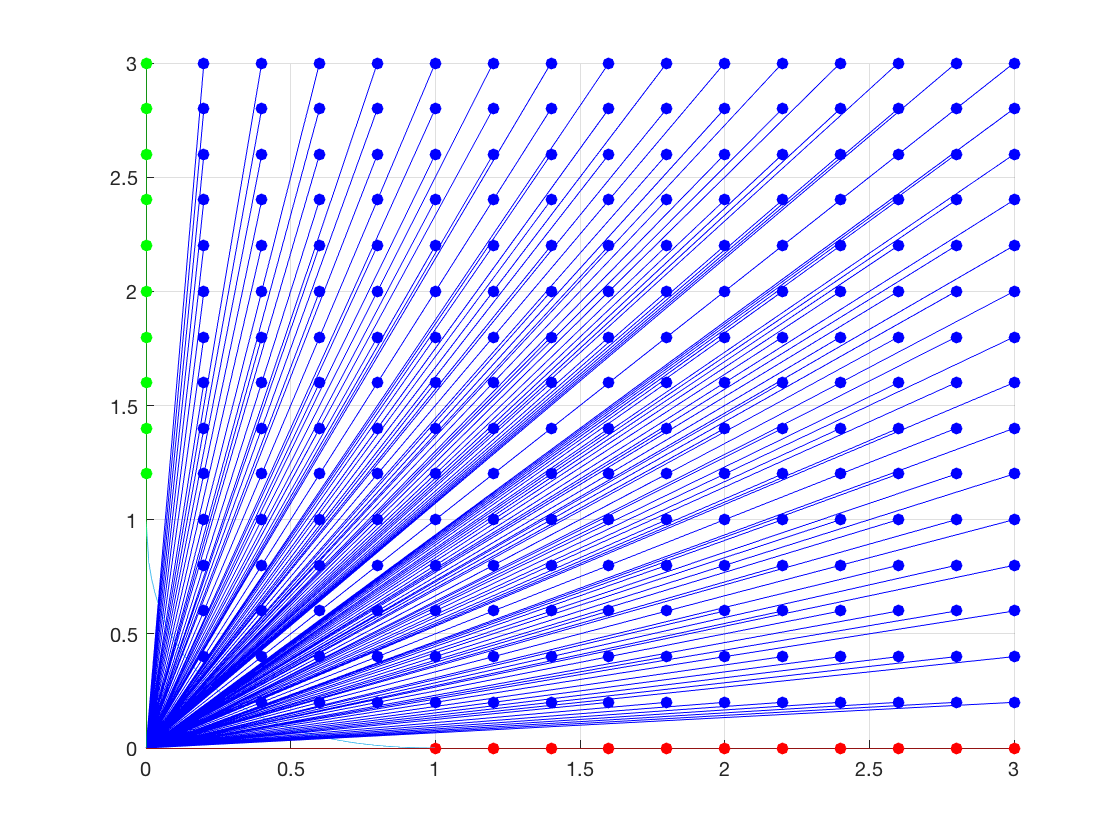}
        \caption{$c=1, b=1, k=2,  r=3$.}
    \end{subfigure}
    \caption{(Colour online) Blue paths are maximisers of type C, i.e. they cross to the $r$-region from the interior of $f$. The set of all $(x,y)$ reached by such paths may be bounded (e.g. see subfigures (D), (E)). Green and red paths are type B maximisers that follow either the $y$- or the $x$- axis respectively. Simulations suggest that when the regions $R_{0, f(0)}$ and $ R_{a_0, 0}$ are not degenerate they can intersect, and bound the Type C region. Finally, the target points of yellow paths are those for which the maximiser is not unique. }
    \label{fig:sim}
\end{figure}

Combining the explicit results obtained in the two examples, we can state the following theorem of counterexamples, describing situations that do not occur in the homogeneous setting.

\begin{theorem} \label{thm:2}Depending on the speed function $c(x,y)$,
	\begin{enumerate}
	\item $\Gamma_c(x,y)$ is not necessarily concave,  
		and its level curves are not necessarily convex. ($\Gamma_{c_\ell}$ in Theorem \ref{thm:2shifted}). 
	\item $\Gamma_c(x,y)$ may exhibit flat edges.	 ($\Gamma_{c_\ell}$ in Theorem \ref{thm:2shifted}).
	\item $\Gamma_c(x,y)$ is not necessarily differentiable on the interior of $\R^2_+$.  ($\Gamma_{c_\ell}$ in Theorem \ref{thm:2shifted}).
	\item The maximisers of \eqref{macroLPT} for some $(x,y)$ are not necessarily unique. (See points on $L(x,y)$ in Theorem \ref{thm:2shifted}, Remark \ref{rem:non-uniqueness}, and Fig. \ref{fig:sim})
	\item It is possible to have terminal points $(x,y)$ for which the maximiser of \eqref{macroLPT}	has an initial segment on one of the coordinate axes. (Theorem \ref{thm:regionLM}, Proposition \ref{prop:EX2}).
	\end{enumerate}
\end{theorem}
We leave the calculus details necessary for the proof of Theorem \ref{thm:2} to the reader.

\section{The shifted two-phase model} \label{sec:2shifted}
From Jensen's inequality and Theorem \ref{thm:1} the variational formula for the limiting last passage time can be simplified to 
\be\label{eq:LPL}
\Gamma_{c_\ell}(x,y) = 
	\begin{cases}
		\displaystyle \sup_{b_1>a_1 \ge \lambda} \Big\{ \gamma (a_1,a_1-\lambda) +\frac{1}{r}\gamma(b_1-a_1, b_1-a_1)+ \gamma (x-b_1,y-b_1+\lambda)\Big\}\ \bigvee \gamma(x,y)\\
		\hspace{9.2cm} \quad \,\, \text{if } y>x-\lambda,\\
		\displaystyle \sup_{a_2 \ge \lambda} \Big\{ \gamma (a_2,a_2-\lambda) + \frac{1}{r}\gamma(x-a_2, y-a_2)\Big\} \bigvee \gamma(x,y), \quad\quad \text{ if }  y=x-\lambda,\\	
		\displaystyle \sup_{a_3 \ge \lambda} \Big\{ \gamma (a_3,a_3-\lambda) + \frac{1}{r}\gamma(x-a_3, y-a_3+\lambda)\Big\},\quad\qquad\!\!\qquad\text{ if }  y<x-\lambda. 
	\end{cases}
\ee  
The top  and middle expressions corresponds to the passage time up to $(x,y)$ above or on the discontinuity line. If $x \ge \lambda$ then the optimal paths can either be a straight line up to $(x,y)$ corresponding to microscopic maximal path in environment Exp$(1)$, or a piecewise linear path which takes advantage of the smaller rate on the discontinuity line. Microscopically, the maximal path enters the region with environment Exp$(r)$ but does not fluctuate from the discontinuity line macroscopically. It could also be that by default the maximal path is the straight line segment when $x < \lambda$ at which point the supremum takes the value $-\infty$ and only $\gamma(x,y)$ remains.

If $(x,y)$ is below the discontinuity then it has to be that the macroscopic maximal path is piecewise linear and it crosses the line $t = s -\lambda$ at some optimal point.

In the computations that follow set 
\[
K(r)=\sqrt{1+\frac{r^2}{4(1-r)}}.
\]
We treat the three cases separately:
\begin{enumerate}[(1)]
\item {\bf{ Case 1: $y>x-\lambda$:} } Assume $x \ge \lambda$ otherwise, as we discussed the maximal path is the straight line and the shape function is $\gamma(x,y)$. We begin by explicitly computing the supremum, which after substitution of the formula for $\gamma$ and some manipulation, it becomes 
\begin{align*}
I_{c_\ell,}(x,y)
&= \sup_{b_1\geq a_1}\Big\{\Big(2-\frac{4}{r}\Big)(a_1-b_1)+x+y+2(\sqrt{a_1(a_1-\lambda)}+\sqrt{(x-b_1)(y-b_1+\lambda)})\Big\},
\end{align*}
where the parameters $a_1,b_1, \lambda$ and the point $(x,y)$ have to satisfy the constraints 
\[
x \ge b_1 \ge a_1\geq \lambda, \text{ and }  y\geq b_1-\lambda.\\
\]
The unknowns are $a_1, b_1$ and they are the $x$ - coordinates of the points on the line $t = s -\lambda$ that determine the second segment of tha potential piecewise linear path.  Compute the first partial derivatives for $a_1$ and $b_1$ and set them equal to 0 to obtain
\begin{align*}
&\frac{\partial I_{c_\ell}(x,y)}{\partial a_1}=2-\frac{4}{r}+\frac{2a_1-\lambda}{\sqrt{a_1(a_1-\lambda)}}=0 \\
&\frac{\partial I_{c_\ell}((x,y)}{\partial b_1}=\frac{4}{r}-2+\frac{2b_1-x-y-\lambda}{\sqrt{(x-b_1)(y-b_1+\lambda)}}=0. 
\end{align*}
From the first equation, imposing the condition $ x \ge a_1>0$ to obtain the optimal entry point 
\be\label{eq:astar}
(a_1^*,a_1^*-\lambda)=\Big(\frac{\lambda}{2}(K(r) +1),\frac{\lambda}{2}(K(r)-1)\Big).
\ee
From the second equation and the condition and $a_1 \le b_1 \le x$, we 
\begin{align}
&(b_1^*,b_1^*-\lambda)=\Big(\frac{(x+y+\lambda) + (x - y-\lambda)K(r)}{2},\frac{(x+y-\lambda) + (x - y- \lambda)K(r)}{2}\Big)\label{eq:bstar}
\end{align}
under the constraint 
\be \label{eq:linecon}
y\le \frac{K(r)+1}{K(r)-1}x-\frac{2K(r)}{K(r)-1}\lambda.
\ee
The constraint is equivalent to $a_1^* \le b_1^*$. When it is not satisfied, the optimal path is the straight line.  It is always true that $b_1^* < x$.
Check that $(a_1^*,b_1^*)$ gives a local maximum by computing the Hessian matrix $H(a_1,b_1)$ for which
\[
\det\{H(a_1^*,b_1^*)\}=\frac{\lambda^2(x-y-\lambda)^2}{4[a_1^*(a_1^*-\lambda)(x-b_1^*)(y-b_1^*+\lambda)]^{3/2}}, \text{ and }   \frac{\partial^2  \Gamma_{c_\ell}(a_1^*,b_1^*)}{\partial a_1^2}=\frac{-\lambda^2}{2[a_1^*(a_1^*-\lambda)]^{3/2}}.
\]
It is immediate to check that it is also a global maximum for $I_{c_{\ell}}(x,y)$. 
We substitute the values of $a_1^*$ and $b_1^*$ of respectively \eqref{eq:astar} and \eqref{eq:bstar} into \eqref{eq:LPL} to obtain the value on the trapezoidal path $I_{c_\ell}(x,y)$
\begin{align*}
I_{c_\ell}(x,y) &=x\Big(1+\Big(\frac{2}{r}-1\Big)(1+K(r))-\sqrt{K(r)^2-1}\Big)+\\
&\hspace{1.5cm}+y\Big(1+\Big(\frac{2}{r}-1\Big)(1-K(r))+\sqrt{K(r)^2-1}\Big)\\
&\hspace{4cm}+2\lambda\Big(\Big(1-\frac{2}{r}\Big)K(r)+\sqrt{K(r)^2-1}\Big)\\
&= (1+A(r)) x + \Big( 1  +\frac{1}{A(r)}\Big)y -D(r),
\end{align*}
where we set 
\be \label{eq:AD}
A(r) = \frac{(1+ \sqrt{1 - r})^2}{r}, \quad D(r) = 4\lambda\frac{\sqrt{1-r}}{r}. 
\ee
In order to find the region for which $I_{c_\ell}(x,y)$ is actually $\Gamma_{c_\ell}(x,y)$, we directly compare with $\gamma(x,y)$. The two functions give the same value on the curve 
\be
 A(r) x + \frac{1}{A(r)}y -D(r) = 2 \sqrt{xy}.
 \ee
For $(x,y)$ in the region $x- \lambda \le y \le \frac{K(r)+1}{K(r)-1}x-\frac{2K(r)}{K(r)-1}\lambda$, the left-hand side in the display above is always positive, so we can square both sides and identify the curve as  
\begin{align*} 
0 &= \Big(A(r)x-\frac{1}{A(r)}y\Big)^2-2 D(r) \Big(A(r)x+\frac{1}{A(r)}y\Big)+D(r)^2= L(x,y),
\end{align*}
where $L(x,y)$ is defined by the expression in the display above. Equation $L(x,y) = 0$ defines a parabola. 
It has an axis of symmetry that is parallel to - and above - the line \eqref{eq:linecon} and it is tangent to the discontinuity line $y= x - \lambda$ precisely at point $(a_1^*, a_1^*-\lambda)$ given by \eqref{eq:astar}.  Line \eqref{eq:linecon} also crosses both the parabola and the discontinuity line precisely at the same point \eqref{eq:astar}. Therefore, 
\be
I_{c_{\ell}}(x,y) = G_{c_{\ell}}(x,y) \text{ if and only if } (x,y) \in \mathcal R_{\lambda, r} = \{ (x,y):  a^*_1 \le x,\,\,  x - \lambda \le y, \,\,L(x,y) > 0\}.
\ee 
For $(x,y) \in  \mathcal R_{\lambda, r}$ the maximiser is the trapezoidal path with second segment on the discontinuity line of  $c_{\ell}$. For all other $(x, y)$ with $y > x - \lambda$ the maximizing path is the straight line and $\Gamma_{c_{\ell}}(x,y) = \gamma(x,y)$. Points on the curve $L(x, y) = 0$ have two maximizing paths. 

One last remark is that if $(x,y)$ and $(z,w)$ both belong in $\mathcal R_{\lambda, r}$ then the slope of the third segments of the corresponding maximising paths are actually the same and equal to $\frac{K(r)+1}{K(r)-1}$. Therefore they are parallel to the axis of symmetry of the parabola (so they also intersect the critical parabola) and have finite macroscopic length.

\medskip
\item {\bf Case 2:  $y=x-\lambda$}. The same steps as before (or continuity of $G_{c_\ell,}(x,y)$ as $y \searrow x -\lambda$ ) give
\begin{align*}
\Gamma_{c_\ell,}(x,y)&=(\sqrt{a_1^*}+\sqrt{a_1^*-\lambda})^2+\frac{1}{r}(\sqrt{x-a_1^*}+\sqrt{x-a_1^*})^2\\
&= \frac{4}{r}x+\lambda\Big(K(r)+\sqrt{K^2(r)-1}-\frac{2}{r}(1+K(r))\Big).
\end{align*} 
When $x \ge a_1^*$, the maximiser has two linear segments; the first one goes from $0$ to $(a_1^*, a_1^*-\lambda)$ and the second one follows the discontinuity line up to $(x, x-\lambda)$.
 
\medskip
\item \textbf{Case 3: $y<x-\lambda$}. An explicit analytical solution to the variational problem is not easily tractable. The maximisers are piecewise linear, with slopes $m_1$, $m_2$ with $m_2 > m_1$. The optimal crossing point $(a^*_3, a_3^*-\lambda)$ on the discontinuity line always has $a_3^* < a_1^*$.
\end{enumerate}

\begin{figure}[h!]
\centering
\begin{tikzpicture}[>= latex, scale=0.8]

    \draw [red!30, domain=0.5:2,fill,rotate=-20] plot (\x, {\x*\x})--(2,4)--(3.125,2.875)--(0.5,0.25);
 \draw [blue!30, domain=0.5:2,fill,rotate=-20] plot (\x, {\x*\x})--(2,4)--(0.5,5.5)--(0.5,0.25);
\draw[green!30,fill] (-0.7,-0.5)--(3.93,-0.5)--(3.93,1.62)--(-0.7,-0.5);
\draw[blue!30,fill] (2.35,5)--(-2,5)--(-2,-0.5)--(-0.7,-0.5)--(0.6,0.08);
%horizontal axis
\draw[->] (-2,-0.5) -- (4.5,-0.5) node[anchor=north] {\small$x$};
% vertical axis
\draw[->] (-2,-0.5) node[below]{\small0}-- (-2,6) node[anchor=east] {\small$y$};
 \draw [name path=parabola, domain=0.5:2,rotate=-20] plot (\x, {\x*\x});
 \draw [name path=parabola, dashed,domain=-2:0.5,rotate=-20] plot (\x, {\x*\x});
\draw[ rotate=-20] (0.5,-0.35)--(0.5,5.5);
\draw[ rotate=-20] (-0.45,-0.7)--(3.125,2.875);
\draw (5,1.62)node{\small$y=x-\lambda$};
\draw ((2.35,5.4)node{\small$y=A(r)(A(r)x-2\lambda\frac{(2-r)}{r})$};
\draw[ nicos-red, line width=2pt, rotate=-20] (0.5,0.25)--(2.25,2)--(2.25,3.5);
\draw[ nicos-red, line width=2pt] (-2,-0.5)--(0.6,0.08);
\draw[ blue, line width=2pt] (-2,-0.5)--(1,3);
\draw[ my-green, line width=2pt] (-2,-0.5)--(-0.2,-0.3)--(3.3,0.6);
\end{tikzpicture}
\hspace{0.5cm}
\includegraphics[height=5cm]{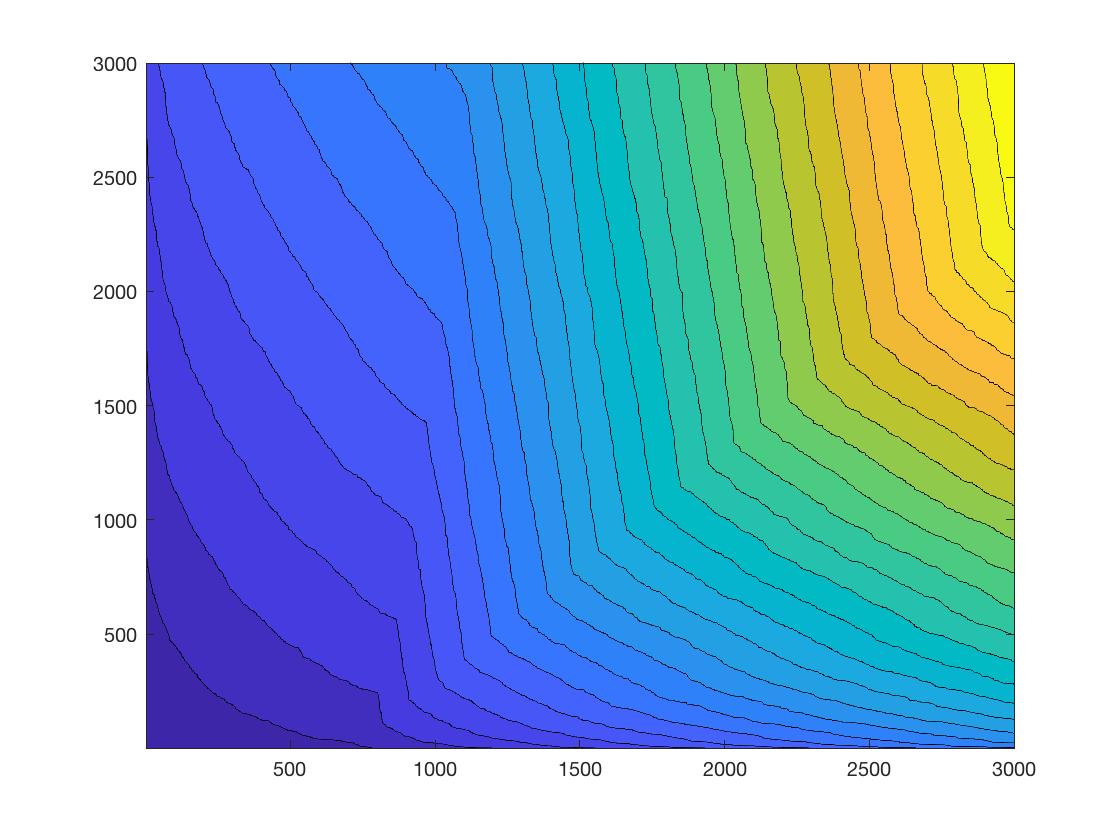}
\caption{(Left) Maximal macroscopic paths for the shifted two-phase corner growth model. In the \textcolor{blue}{blue} region we have a straight line path, in the \textcolor{red}{red} region we have a three piecewise linear path and in the \textcolor{green}{green} region we have a two piecewise linear path.\\
(Right) Numerical simulation of the shape function $\Gamma_{c_\ell}(x,y)$. Notice the non-convexity of the level curves, and the points of non-differentiability of the level curves, and by extension of $\Gamma_{c_\ell}$.
}
\label{fig:PathSM}
\end{figure}
\begin{remark} When the environment is homogeneous and $c(x,y) = c$, the shape function is strictly concave and in $C^2(\R^2_+)$.
As one can see in Figure  \ref{fig:PathSM}, the simulations suggest that the shape function for the shifted inhomogeneous model is no longer strictly concave or $C^1$ in the interior of $\R^2_+$. Indeed this is a straight-forward calculation because we have precise formulas for the shape function for $(x,y) \in \mathcal R_{\lambda, r} $ and for $(x,y)$ for which $y$ is above the critical parabola. We leave this calculation to the reader. The concavity-breaking does not occur in the two-phase model without shifting of \cite{Geo-Kum-Sep-10}. The flat edge is common in both inhomogeneous models.
\end{remark}

\section{The corner-discontinuous last passage percolation}
\label{sec:cornerinho}
It will be convenient to adopt a more general setting for the discontinuity curve $f$ then the one described in Section \ref{sec:model}. To this end, we begin from
Consider a $C^2$ function $g: \R^2_+ \to \R_+$ with the property that its level curve $g(x,y) = k$ when viewed as a function of $ y = f (x)$ is strictly decreasing and twice differentiable function so that the first and second derivative never become zero, i.e. 
\[ \frac{df}{dx} < 0, \quad  \frac{d^2f}{dx^2} \neq 0. \] 
For what follows we restrict to the case  where $f$ is convex and its second derivative strictly positive. 

Since the gradient of $g$ is always perpendicular to its level curve, for any $(a,b) \in \R^2_{>0}$ with $g(a,b) =h$ we have that 
\begin{equation}\label{eq:a2}
\partial_xg(a,b) \cdot \partial_yg (a,b) >0. 
\end{equation}
Let $a_0$ and $b_0$ be defined by $g(a_0, 0) = g(0, b_0)=k$. They can also take the value infinity if $g$ does not intersect the coordinate axes. 

We define the macroscopic speed function $c_{g,k}(x,y)$ on $\R^2_+$ to be
\[ 
c_{g,k}(x,y) = 	\begin{cases}
			1, & \text{if } g(x,y) < k,\\
			r, & \text{if } g(x,y) \ge k.
		\end{cases} 
\]
From Theorem \ref{thm:1} and the fact that macroscopic optimisers are piecewise linear in constant regions, the limiting last passage time is given by 
\be\label{eq:varprop}
\Gamma_{c_{g, k}}(x,y) = 
	\begin{cases}
		\gamma(x,y), &\text{ if } g(x,y) \le k \\
		\displaystyle \sup_{a \le x \wedge a_0,\, b \le y\wedge b_0,\, g(a,b) =k} \Big\{ \gamma (a,b) + \frac{1}{r}\gamma(x-a, y-b)\Big\}, & \text{ if } g(x,y) > k. 
	\end{cases}
\ee  
Except for some specific cases, the solution to the variational problem in \eqref{eq:varprop} cannot be explicit but can be approximated numerically. However, this model allows for partial analysis, and despite its simplicity it demonstrates behaviour that can be rigorously shown to differ from passage time in a homogeneous environment. 

We write again Definition \ref{def:cross} using the notation introduced so far in this section. 

\begin{definition}[Crossing points] We say that a point $(a,b)$ is a ($g$ -) crossing point for point $(x,y)$ if it belongs in the set 
\[
\mathcal S_{x,y} = \{ (a,b): g(a, b) = k \text{ which solve \eqref{eq:varprop} for the given } (x,y)\}.
\]
In words, $(a,b)$ solves the optimization problem \eqref{eq:varprop}. 
The set of all crossing points is defined by 
\[
\mathcal S = \{ (a,b): g(a, b) = k \text{ which solve \eqref{eq:varprop} for some } (x,y)\}.
\]
\end{definition}

If $|\mathcal  S_{x,y}| = 1$ then there is unique piecewise linear macroscopic maximal path from the origin to $(x,y)$ which is a maximiser of the variational formula \eqref{macroLPT}, and this passes through $(a,b) \in \mathcal S_{x,y}$.

In the homogeneous environment ($r =1$), maximisers of \eqref{macroLPT} are unique and are straight lines, i.e. $|\mathcal  S_{x,y}|=1$. Here, depending on the function $g$, this is no longer true, as discussed in the following remark.

\begin{remark}\label{rem:non-uniqueness} Depending on the function $g$, it is possible to have a point $(x,y)$ that does not lead to a unique maximiser of the problem \eqref{eq:varprop}. Suppose you fix a point $(t, t)$ in the $r$-region, and further assume that $f$ is symmetric about the main diagonal. By carefully modulating the values of $f$ around the main diagonal, and by appropriately lowering the value of $r$, one can show that the main diagonal cannot be an optimiser for $\Gamma$. Then the optimiser ${\bf x}$ is a concatenation of two linear segments that crosses $f$ at some point. Because $f$ is symmetric, the piecewise linear curve that is symmetric to ${\bf x}$ about the diagonal is also an optimiser. We leave the details to the reader.  \qed
\end{remark}
 
\begin{lemma}\label{lem:AS}
The set of \emph{crossing points} $\mathcal S  $
is dense on the curve $g(a,b) = k$. 
\end{lemma}
\begin{proof}
To see this, fix an arbitrary segment on the level curve 
\[
\mathcal I = \{ (a,b): a_1 < a < a_2, \,\, b_1 < b < b_2, \,\, g(a,b) = k\}
\]
and consider $(x,y)$ so that $a_1/2 < x < a_2/2, \,\, b_1/2 < y < b_2/2, \,\, g(x,y) > k$ which is possible since the level curve is convex. The maximal path to $(x,y)$ has to cross the curve at some point $(a_{x,y}, b_{x,y})$ with $a_1/2 < a_{x,y} < a_2/2, \,\, b_1/2 < b_{x,y} < b_2/2$ since it will be piecewise linear with strictly positive slope for each segment. This suffices for the proof.
\end{proof}
Fix a \emph{crossing point} $(a,b)$. Then, for some $(x,y)$, this point solves the Lagrange multiplier problem  
\begin{align}
h(a, b, \lambda) &=\sup_{a,b,\lambda}\Big\{ \gamma (a,b) + \frac{1}{r}\gamma(x-a, y-b) + \lambda (g(a,b) - k)\Big\}, \label{eq:lagrange} \\
&\phantom{xxxxxxxxxxxxxxxxxx} \quad 0 \le a \le x \wedge a_0, \,\, 0 \le b \le y \wedge b_0. \notag
\end{align}
Function $h$ has two derivatives in the interior of its domain, so we can optimize over $(a, b, \lambda)$ as usual. If the local maximum is in the interior we will find it using the Lagrange multiplier method. Otherwise, we will check even the boundaries of the region.
The derivatives give 
\begin{subnumcases}{}
&$\frac{\partial h}{\partial a}=\frac{\sqrt{a}+\sqrt{b}}{\sqrt{a}}-\frac{1}{r}\frac{\sqrt{x-a}+\sqrt{y-b}}{\sqrt{x-a}}+\lambda\partial_ag(a,b)=0$, \label{eq:a3a}
  \\
&$\frac{\partial h}{\partial b}=\frac{\sqrt{a}+\sqrt{b}}{\sqrt{b}}-\frac{1}{r}\frac{\sqrt{x-a}+\sqrt{y-b}}{\sqrt{y-b}}+\lambda\partial_bg(a,b)=0$,\label{eq:a3b}
\\
&$\frac{\partial h}{\partial \lambda}=g(a,b)-k=0$. \label{eq:a3c}
\end{subnumcases}
Solve the first two for $\lambda$ and set the two expressions equal to obtain 
\be\label{eq:mess}
r\Big(1+\frac{\sqrt{b}}{\sqrt a}\Big)\partial_ag\Big(\frac{\sqrt{a}}{\sqrt{b}}-\frac{\partial_bg}{\partial_ag}\Big)=\bigg(1+\frac{\sqrt{x-a}}{\sqrt{y-b}}\bigg)\bigg(\partial_ag-\partial_bg\frac{\sqrt{y-b}}{\sqrt{x-a}}\bigg).
\ee
For the $(x,y)$ for which the crossing point is the $(a, b)$ that satisfies equation \eqref{eq:mess}, the maximal path is piecewise linear with slopes 
\[
m_1 = \frac{b}{a} \quad \text{and} \quad m_2 = \frac{y - b}{x - a}.
\]
Then equation \eqref{eq:mess} can be written as 
\be \label{eq:mess2}
\nabla g (a,b) \cdot \bigg(\frac{r(1+\sqrt{m_1})}{\sqrt{m_1}} - \frac{1+\sqrt{m_2}}{\sqrt{m_2}}, -r(1+\sqrt{m_1}) + (1 +\sqrt{m_2})\bigg) =0.
\ee
Equation \eqref{eq:mess2} has a very convenient form. It shows that if for a fixed $(x,y)$ the crossing point $(a,b)$ solves the Lagrange multiplier problem \eqref{eq:lagrange}, then the same point $(a,b)$ solves \eqref{eq:lagrange} for any $(x', y') = (a, b) + \lambda(x - a, y-b)$ on the line from $(a,b)$ with slope $m_2$. 
Using the form $g(x,y) = y - f(x)$, we have that $\nabla g(a,b) = (-f'(a), 1)$. Relation \eqref{eq:mess2} after some algebraic manipulations then becomes 
\be\label{explain}
\frac{r-1}{r} + \sqrt{m_1} - \frac{\sqrt{m_2}}{r} = -\frac{f'(a)}{r}\Big( r-1 + \frac{r}{\sqrt{m_1}} - \frac{1}{\sqrt{m_2}} \Big).
\ee
We will use this equation later, as any crossing point away from the boundary satisfies relation \eqref{explain}. 

The next lemma shows that if  $(a, b)$ solves \eqref{eq:mess2}  (or $a$ solves \eqref{explain}) does not imply that we found a global maximiser.

\begin{lemma}[Maximal paths cannot cross each other]\label{lem:UMP} 
Suppose that for a point $(x,y)$ there exist two crossing points $(a_1^*,b_1^*)$ and $(a_2^*,b_2^*)$ $(a_1^* > a_2^*) $ that satisfy \eqref{eq:mess}, \eqref{eq:mess2} subject to the constraint \eqref{eq:a3c} and in particular maximise \ref{eq:lagrange}. Then for $(x', y') = (a_1^*, b_1^*) + \kappa(x - a_1^*, y-b_1^*)$  we have that 
\begin{enumerate}
\item If $\kappa > 1$, crossing point $(a_1^*, b_1^*)$ is a critical point for the Lagrange multiplier problem when the terminal point is $(x', y')$. 
\item If $\kappa > 1$, crossing point $(a_1^*, b_1^*)$ is not a maximiser for the Lagrange multiplier problem when the terminal point is $(x', y')$.
\item If $\kappa < 1$, crossing point $(a_1^*, b_1^*)$ is the unique maximiser for the Lagrange multiplier problem when the terminal point is $(x', y')$.
\end{enumerate}
\end{lemma}
\begin{proof}

See Figure \ref{fig:2MP} for the geometric construction.

For (1) the statement follows from the fact that slope of the segment  $(a_1^*, b_1^*) \to (x', y')$ is the same as that for $(a_1^*, b_1^*) \to (x, y)$. Equation \eqref{eq:mess2} is automatically satisfied so $(a_1^*,b_1^*)$ is a critical point.

For (2) we reason as follows. The path $(0,0) \to(a_2^*, b_2^*) \to (x,y) \to (x', y')$ cannot be optimal for $(x',y')$, because it is polygonal in the homogeneous region of rate $r$ and the straight line $(a_2^*, b_2^*)$ is strictly better. However it has the same weight as the path  $(0,0) \to(a_1^*, b_1^*) \to (x', y')$ and therefore this path cannot be optimal for $(x',y')$. 

Part (3) follows with similar arguments.
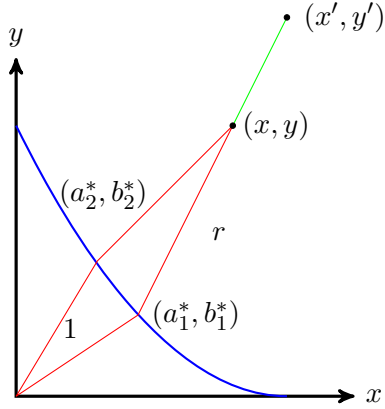
\begin{figure}[h!]
\center{
\begin{tikzpicture}[
        scale=1.8,
        IS/.style={blue, thick},
        LM/.style={red, thick},
        axis/.style={very thick, ->, >=stealth', line join=miter},
        important line/.style={thick}, dashed line/.style={dashed, thin},
        every node/.style={color=black},
        dot/.style={circle,fill=black,minimum size=4pt,inner sep=0pt,
            outer sep=-1pt},
    ]
    % axis
    \draw[axis,<->] (2.5,0) node(xline)[right] {$x$} -|
                    (0,2.5) node(yline)[above] {$y$};
  
    \draw[IS] (0,2) coordinate (IS_1) parabola[bend at end]
         (2,0) coordinate (IS_2);
 \draw (1.5,1.2) node{$r$};
 \draw (0.4,0.5) node{$1$};
\draw[red] (0,0)--(0.9,0.6)node[xshift=2em]{$(a_1^*,b_1^*)$}--(1.6,2)node[xshift=1.5em]{$(x,y)$};
\draw[red] (0,0)--(0.6,1)node[xshift=0.2em,yshift=2.3em]{$(a_2^*,b_2^*)$}--(1.6,2);
\draw[green] (1.6,2)--(2,2.8)node[xshift=2em]{$(x',y')$};
\draw [fill,black] (1.6,2) circle [radius=0.02] ;
\draw [fill,black] (2,2.8) circle [radius=0.02] ;
\end{tikzpicture}
}
\caption{The construction described in the proof of Lemma \ref{lem:UMP}.}
\label{fig:2MP}
\end{figure}
\end{proof}

Next, we want to verify that the maximal path will never follow a vertical or horizontal line in the $r$ region, i.e. the slope of the second segment of a potential maximiser cannot have slope equal to zero or infinity. 

\begin{lemma} \label{lem:dense}Suppose that $(a,b) \in \mathcal S_{x,y}$. Then $a < x$ and $b < y$. In particular, any $(x,y)$ for which the maximiser of $\Gamma_{c_g,k}(x,y)$ does not cross $(a_0, 0)$ or $(0, b_0)$ has to cross at a point $(a,b)$ that satisfies \eqref{eq:mess}, \eqref{eq:mess2} and the second segment has a non-zero, finite slope. 
\end{lemma}

\begin{proof}

We only show that a second segment of infinite slope is not optimal. The strictly positive slope claim follows similarly. We compare the last passage time of a path which crosses the discontinuity in the point whose $x$ coordinate is the same of the point that it has to reach, in other words $a=x, b= f(x)$, and another path with $a=x-\varepsilon$. Under these assumptions, we have that 
\[
 f(x-\e)=b+\delta(\varepsilon) \text{  with } \lim_{\varepsilon\to0}\frac{\varepsilon}{\delta(\varepsilon)}= c \in (0, \infty).\]
 This is because $b + \delta(\e) = f(x-\e) = f(x) - f'(x)\e + o(\e) = b - \e f'(x) + o(\e)$ by a Taylor expansion around $x$ and the fact that $-f'(x) \in (0, \infty)$. Then, a direct comparison between the weight of the two paths, $\pi_1$ which crosses at $(x, b)$ and $\pi_2$ crossing at $(x - \e, f(x -\e))$ gives
\begin{align*}
I_{c_{g,k}}( \pi_2) - I_{c_{g,k}}(\pi_1) &= 
(\sqrt{x-\varepsilon}+\sqrt{b+\delta(\varepsilon)})^2+\frac{1}{r}(\sqrt{\varepsilon}+\sqrt{y-b-\delta(\varepsilon)})^2 - (\sqrt{x}+\sqrt{b})^2-\frac{1}{r}(y-b)\\
&= \Big(1-\frac{1}{r}\Big)(\delta(\varepsilon)-\varepsilon)+2\sqrt{(x-\varepsilon)(b+\delta(\varepsilon))}+\frac{2}{r}\sqrt{\varepsilon(y-b-\delta(\varepsilon))} - 2\sqrt{xb}\\
&= \Big(1-\frac{1}{r}\Big)(\delta(\varepsilon)-\varepsilon) - \frac{\e[b - x f'(x)]}{\sqrt{xb}} +\frac{2}{r}\sqrt{\varepsilon(y-b-\delta(\varepsilon))} + o(\e).
\end{align*}
Divide through by $\e$ and let it tend to 0 to see that the last expression is eventually positive. As such, $I_{c_{g,k}}( \pi_2)$ is a lower bound for the shape function at  $(x,y)$ and therefore the maximiser cannot be $\pi_1$. 
\end{proof}

\begin{lemma}\label{lem:46} Let $(x,y)$ and $(z,w) \in \R^2_+$ so that $(x,y) \neq \lambda (z,w)$ for any $\lambda \in \R$. Then 
\[ \mathcal S_{x,y} \cap \mathcal S_{z,w} \in \{ \varnothing, (a_0, 0), (0, b_0)\}.\]
In other words, the only possible crossing points from which more than one maximiser passes, are the axes points  $(a_0, 0), (0, b_0)$.
\end{lemma}

\begin{proof} Assume by way of contradiction that two terminal points in general position, $(x,y)$ and $(z,w)$ have the same crossing point $(a,b)$ for which $0 < a < a_0$ and $0 < b < b_0 $. Then the gradient of $g$ at $(a, b)$ is well defined. By the previous lemma, equation \eqref{eq:mess2} holds for $m_1 = b/a$ and for both values of $m_2$, 
\[
m_2 = m_{x,y} = \frac{y - b}{ x - a}, \quad \text{and} \quad  m_2 = m_{z,w} = \frac{w - b}{ z - a}.
\]
For $(i,j) \in \{(x,y), (z,w)\}$ define 
\[ {\bf v}_{i,j} = \bigg(\frac{r(1+\sqrt{m_1})}{\sqrt{m_1}} - \frac{1+\sqrt{m_{i,j}}}{\sqrt{m_{i,j}}}, -r(1+\sqrt{m_1}) + (1 +\sqrt{m_{i,j}})\bigg) = (v_{i,j}^{(1)}, v_{i,j}^{(2)}).\]
Vector ${\bf v}_{i,j}$ would be tangent to the level curve $g(x,y) = k$ at $(a,b)$ and at such, ${\bf v}_{i,j} \neq 0$. The monotonicity of the level curve and the fact that $(a,b)$ does not lie on one of the axes give that $v_{i,j}^{(1)} \cdot v_{i,j}^{(2)} \neq 0$.
By planarity and \eqref{eq:mess2}, this and the last equation imply that there exists a $\kappa \in \R \setminus \{0\}$ so that 
${\bf v}_{z,w} = \kappa {\bf v}_{x,y}$. The assumption that $(x,y)$ and $(z,w)$ are not collinear gives that $\kappa \neq \pm 1$. Assume without loss of generality that  $m_{z,w} > m_{x,y}$. Then coordinate-wise, 
\[
v_{x,y}^{(1)} < v_{z,w}^{(1)}, \qquad v_{x,y}^{(2)} < v_{z,w}^{(2)}.
\]

On the other hand, it has to be by equations \eqref{eq:a2} and \eqref{eq:mess2} that the $v_{x,y}^{(1)}$ and $v_{x,y}^{(2)}$ have opposite signs, otherwise \eqref{eq:mess2} would never be satisfied. Assuming $0 < v_{x,y}^{(1)}$, it has to be that $\kappa > 1$, but that would imply that $v_{x,y}^{(2)} > v_{z,w}^{(2)}$ which leads to a contradiction. Similarly, we reach a contradiction when $v_{x,y}^{(1)} < 0$. 
\end{proof}

From Lemma \ref{lem:46} we know that from each crossing point except $(0, f(0))$ and $(a_0, 0)$ there is only one optimal slope that can be obtained.  Remark \ref{rem:non-uniqueness} suggests that it is possible that a point could be reached by two maximal paths that both cross at the interior of $f$. Finally we discuss what happens when two maximal paths exists for a point $(x,y)$, one from the axis and the other from a crossing point or both from the axes.

\begin{proposition}\label{lem:NAP} The following properties hold:
\begin{enumerate}
\item If a maximal path which crosses $(0, f(0))$ or $(a_0, 0)$ 
and a maximal path through any crossing point 
$(a, f(a))$ intersect, they intersect at their terminal point and that point has to belong on $\partial R_{0, f(0)}$.

\item If $(x,y) \in {\rm int}(R_{0, f(0)})$ and it also belongs on the extension of a maximiser ${\bf x}$ that crosses at $(a', f(a'))$, $a' \neq 0$, it has to be 
 \[
 I(\pi_{0, (0,f(0))}) + I(\pi_{(0,f(0)), (x,y)}) > I(\pi_{0, (a',f(a'))}) + I(\pi_{(a',f(a')), (x,y)}), 
 \]
 where $\pi_{\bf u, v}$ is a linear segment between $\bf u$ and $\bf v$.
 In particular, any $(x,y) \in {\rm int}(R_{0, f(0)})$ has a unique maximiser that has to go through $(0, f(0))$.   
 
\item If $R_{0, f(0)} \cap R_{a_0, 0} \neq \varnothing$ and $r > 1$, then the intersection is a segment of a (possibly degenerate) hyperbola.
\end{enumerate}
\end{proposition} 

\begin{proof}[ Proof of Proposition \ref{lem:NAP}]
We prove all the three properties one by one starting from the first.
\begin{enumerate}[(1)]

\item First, we show that also in this situation maximisers cannot cross.
The contrary would be impossible. In fact, if it was possible to extend either maximiser, we would be able to construct a polygonal path which is not linear in a homogeneous environment, and this is not optimal with the same arguments as in Lemma \ref{lem:UMP}. 
 
$R_{0, f(0)}$ by definition is a closed, star-shaped domain. Moreover, since maximal paths cannot cross, $R_{0, f(0)}$ is simply connected. Suppose by way of contradiction that such a terminal point $(x_T, y_T) \in {\rm int}(R_{0, f(0)})$. Then the type C maximiser ${\bf x}_{0, (x_T, y_T)}$ intersects $\partial R_{0, f(0)}$ at some point $(x_R, y_R)$. Since $R_{0, f(0)}$ is closed, $(x_R, y_R)$ has a maximiser ${\bf x}_{0, (x_R, y_R)}$ that goes through $(0, f(0))$. 
By Lemma \ref{lem:UMP}, $(x_R, y_R)$ is also maximised by the portion of ${\bf x}_{0, (x_T, y_T)}$ that terminates at $(x_R, y_R)$, and by the discussion above, $(x_R, y_R)$ has to be a terminal point. This means that $(x_T, y_T)$ cannot be optimised by that type $C$ maximiser, which gives the desired contradiction.

\item Same arguments as above imply the statement.

\item This is a computation of the set of all points $(x,y)\in \R^2_+$ which take the same amount of time going through the $x$ and $y$ axis. 
\begin{align*}
&a_0+\frac{1}{r}(\sqrt{x-a_0}+\sqrt{y})^2=f(0)+\frac{1}{r}(\sqrt{x}+\sqrt{y-f(0)})^2\\
&(a_0-f(0))\frac{r-1}{2}=\sqrt{x(y-f(0))}-\sqrt{y(x-a_0)}.
\end{align*}
Since $r>1$, we have that $(r-1)/2 > 0$. Then, for the equality to hold, we must have  $a_0\ge f(0)$ and $y\ge xf(0)/a_0$ or $a_0<f(0)$ and $y<xf(0)/a_0$. When either of these hold, we can square both sides and after some rearrangements we have
\[
2\sqrt{xy(y-f(0))(x-a_0)}=2xy-xf(0)-ya_0-(a_0-f(0))^2\frac{(r-1)^2}{4}.
\]
This holds only if $y>\frac{xf(0)+(a_0-f(0))^2\frac{(r-1)^2}{4}}{2x-a_0} $ and it implies that both sides above are non-negative. Square both sides another time
\begin{align*}
0 =f(0)^2x^2&+a_0^2y^2-2xy\Big(a_0f(0)+(a_0-f(0))^2\frac{(r-1)^2}{2}\Big)\\
&+(a_0-f(0))^2\frac{(r-1)^2}{2}(f(0)x+a_0y)+(a_0-f(0))^4\frac{(r-1)^4}{16}, 
\end{align*}
which represent the equation of a hyperbola since $(a_0f(0)+(a_0-f(0))^2\frac{(r-1)^2}{2}\Big)^2-a_0^2f(0)^2>0$. 
Note that if $a_0 = f(0)$, the relation that gives the boundary is $x = y$. \qedhere
\end{enumerate}
\end{proof}

We have now verified that the set of crossing points is dense on the level curve (Lemma \ref{lem:AS}) and each one corresponds to a non-degenerate (Lemma \ref{lem:dense}) unique value $m_2$ (Lemma \ref{lem:46}) which in turn corresponds to the slope of the second linear segment of the maximiser. Starting from equation \eqref{eq:mess}, we can identify $m_2$.

Set
\[
D = D(a,b) = r\Big(1+\frac{\sqrt{b}}{\sqrt a}\Big)\Big(\frac{\sqrt{a}}{\sqrt{b}}-\frac{\partial_bg}{\partial_ag}\Big)= r\Big( 1 + \sqrt{m_1} \Big)\Big(\sqrt{\frac{1}{m_1}}-\frac{\partial_bg}{\partial_ag}\Big).
\]
The left-hand side in \eqref{eq:mess} becomes $\partial_ag(a,b) D$. 
Keep in mind that $m_2 > 0$ and solve \eqref{eq:mess} for $m_2$: 
\begin{align}
m_2
&= \frac{4}{ \Big( \frac{\partial_bg}{\partial_ag} - 1 + D + \sqrt{\Big(\frac{\partial_bg}{\partial_ag} - 1 + D\Big)^2 + 4\frac{\partial_bg}{\partial_ag}}\Big)^2}. \label{eq:m_2}
\end{align}
Particularly, equation \eqref{eq:m_2} uniquely identifies the slope of the second segment of the optimal path for a given crossing point $(a,b)$. 
Rewrite equation \eqref{eq:m_2} using the fact that when $b = f(a)$, $\frac{\partial_bg}{\partial_ag}(a, f(a)) = -1/ f'(a)$ to obtain equations  \eqref{eq:m_2f} and \eqref{eq:dioff}. 
\subsection{Maximisers that follow the axes}

We investigate whether the optimization problem \eqref{eq:varprop} in the region $g(x,y) >k$ admits maximisers $(a_0, 0),\,\, (0, b_0)$, i.e.\ maximisers for which the first segment of the macroscopic maximal path follows the axes.

For $(x,y) \in [0,a_0)\times[0, f(0)) = B$ the maximal macroscopic path is obtained by the solution of \eqref{eq:varprop}, and it is impossible for a maximiser to follow one of the axes.
For this behaviour to materialise,  we consider an $(x,y)$ outside of $[0,a_0)\times[0, f(0))$. 

We are finally able to study what happens to $m_2$ defined in \eqref{eq:m_2} if $a$ tends to the boundary values. The idea is that if $m_2$ for crossing points near the $y$-axis (resp. $x$-axis) does not approach $+\infty$ (resp. $0$) then it has to be that type B maximisers exist. 

The behaviour of $m_2$ for $a$ near $0$ (resp. $a_0$) is the content of Proposition  \ref{prop:EX2}, which we prove next.  

\begin{proof}[Proof of Proposition \ref{prop:EX2}] We use equation \eqref{eq:m_2f} for the slope $m_2(a)$ and \eqref{eq:dioff} for the expression $D =D_a$. 
We only show the case for which $a \to 0$ and leave $a\to a_0$ to the reader.
Keep in mind that as $a \to 0$, $f(a)/a \to \infty$. 

First we estimate the limiting behaviour of $D$ using \eqref{eq:dioff}
\begin{align}
D_0 = \varliminf_{a \to 0} D &= \varliminf_{a \to 0} r \Big( 1+  \frac{1}{f'(a)}  + \frac{1}{f'(a)}  \sqrt{\frac{f(a)}{a}} + \sqrt{\frac{a}{f(a)}} \Big) \notag \\
&= r +  \frac{r \sqrt{f(0)}}{\displaystyle \varlimsup_{a\to 0} f'(a) a^{1/2}}= 
\begin{cases}
r, & \alpha > \frac{1}{2},\\
r\Big(1 - \frac{\sqrt{f(0)}}{c^{(-)}_{1/2}}\Big), & \alpha =\frac{1}{2},\\
-\infty, & \alpha < \frac{1}{2}. \label{eq:dliminf}
\end{cases}
\end{align}
\begin{enumerate}[(1)]
\item \textbf{Case 1: $a \to 0, \, f'(a) \, \to -\infty $:} 
Focus on the denominator in \eqref{eq:m_2f} 
\begin{align}
\varlimsup_{a \to 0} m_2(a) &= 4  \varlimsup_{a \to 0}  \bigg(-\frac{1}{f'(a)}-1+D+\sqrt{\bigg(-\frac{1}{f'(a)}-1+D\bigg)^2-4\frac{1}{f'(a)}}\bigg)^{-2}\nonumber\\
&= 4  \varlimsup_{a \to 0}  \bigg(-\frac{1}{f'(a)}-1+D+\bigg|-\frac{1}{f'(a)}-1+D\bigg| + O(\frac{1}{f'(a)})\bigg)^{-2}\nonumber\\
&= 4  \varlimsup_{a \to 0}  \bigg(\Big(-\frac{1}{f'(a)}-1+D\Big)\Big( 1+ \text{sign}\Big(-\frac{1}{f'(a)}-1+D\Big)\Big) + O(\frac{1}{f'(a)})\bigg)^{-2}.\label{eq:35}
\end{align}
Focus for the moment on the sign function in the last display. We have 
\begin{align*}
-\frac{1}{f'(a)}-1+D = (r-1) + \frac{r- 1}{f'(a)}+r\sqrt{\frac{a}{f(a)}}+r\frac{1}{f'(a)}\sqrt{\frac{f(a)}{a}}.
\end{align*}
As $a \to 0$, the second and third term tend to 0 while the last term is negative and as $a \to 0$ the $\liminf$ of the last term is actually $D_0-r$.
Therefore, for $a$ sufficiently small
\be \label{eq:sign}
\text{sign}\Big(-\frac{1}{f'(a)}-1+D\Big)=
\begin{cases}
\text{sign}(r-1), & \alpha > \frac{1}{2},\\
-1, & \alpha < \frac{1}{2},  
\end{cases}
\ee
We are now in a position to finish the calculation from equation \eqref{eq:35}:\begin{enumerate}
\item $r > 1, \alpha > 1/2$: From equation \eqref{eq:sign} substitute it in equation \eqref{eq:35} to obtain 
\[
\varlimsup_{a\to0}m_2(a) = \frac{1}{(r-1)^2}.
\]
\item $r < 1, \alpha > 1/2$, or  $r \neq 1, \alpha < 1/2$: From equations \eqref{eq:sign}, \eqref{eq:35} we now have
\[
\varlimsup_{a\to0}m_2(a) = +\infty.
\]

\item When $\alpha = 1/2$, there are several cases to consider: 
\begin{enumerate}[(i)]
\item $r < 1$, then $\text{sign}\Big(-\frac{1}{f'(a)}-1+D\Big) = -1$ which implies $\varlimsup_{a\to 0 }m_2(a) = + \infty$. 
\item $r > 1$ and  $c_{1/2}^{(+)} < \frac{r\sqrt{f(0)}}{r-1}$, then $\text{sign}\Big(-\frac{1}{f'(a)}-1+D\Big) = -1$.
In this case, $\varlimsup_{a\to 0 }m_2(a) = + \infty$. 
\item $r > 1$ and  $c_{1/2}^{(-)} > \frac{r\sqrt{f(0)}}{r-1}$, then $\text{sign}\Big(-\frac{1}{f'(a)}-1+D\Big) = +1$. This is the most interesting case, as it leads to yet a different possible limit. For the condition to hold it has to be that 
\[
c_{1/2}^{(-)}> \sqrt{ f(0) } \text{ and that } r > \frac{c_{1/2}^{(-)}}{c_{1/2}^{(-)}-\sqrt{f(0)}}> 1.  
\]
When both these conditions are met, we have that
\[
\varlimsup_{a\to 0} m_2(a) = \frac{1}{\Big( r -1 - \frac{r\sqrt{f(0)}}{c_{1/2}^{(-)}} \Big)^2}.
\]
\item $r > 1$ and  $c_{1/2}^{(-)} < \frac{r\sqrt{f(0)}}{r-1} \le c_{1/2}^{(+)} $, then we can find a subsequence $a_k$ such that the 
 $\text{sign}\Big(-\frac{1}{f'(a_k)}-1+D\Big) = -1$ and so that $-\frac{1}{f'(a_k)}-1+D \to r-1 - \frac{r\sqrt{f(0)}}{c_{1/2}^{(-)}}$. Again, 
 $\varlimsup_{a\to 0 }m_2(a) = + \infty$. 
 \item $r > 1$ and $c_{1/2}^{(-)} = \frac{r\sqrt{f(0)}}{r-1}$, we cannot determine the sign function, however, we can find  a subsequence $a_k$ so that  $\varliminf_{a_k\to 0}(-\frac{1}{f'(a_k)}-1+D) = 0$ so also here $\varlimsup_{a\to 0 }m_2(a) = + \infty$. 
\end{enumerate}
\end{enumerate}
\medskip
\item \textbf{Case 2: $a \to 0, \, f'(a) \, \to -c $:} In this case, $D \to -\infty$ as $a \to 0$ so the result follows by a direct limiting argument on $\eqref{eq:m_2f}$. \qedhere
\end{enumerate} 

A close inspection of the previous proof suggests the following crucial lemma.
\begin{lemma} \label{lem:anothergap} Suppose that 
	$ \varlimsup_{a\to 0} m_2(a) = + \infty,$ and that if $\alpha = 1/2$ then $r \notin \Big[ \frac{c_{1/2}^{(+)}}{c_{1/2}^{(+)}-\sqrt{f(0)}}, \frac{c_{1/2}^{(-)}}{c_{1/2}^{(-)}-\sqrt{f(0)}} \Big]$.
 	Then there exists a sequence $\{ a_k \}_{k \in \N}$ with distinct elements so that 
	\begin{enumerate}
	\item $\lim_{k\to \infty} a_k = 0,$
	\item Points $(a_k, f(a_k))$ are all crossing points,
	\item$\lim_{k\to \infty} m_2(a_k) = +\infty$.
	\end{enumerate}
\end{lemma}

\begin{proof}[Proof of Lemma \ref{lem:anothergap}] The lemma is immediately true if $r = 1$ and the environment is homogeneous.

Now assume $r\neq 1$. From Proposition \ref{prop:EX2}, we know that $\varlimsup_{a\to 0} m_2(a) = + \infty$ when 
	\begin{enumerate}
		\item $\alpha < 1/2$,
		\item $ \alpha > 1/2$ and $r < 1$,
		\item $ \alpha = 1/2$ and $r \in \Big(1, \frac{c_{1/2}^{(-)}}{c_{1/2}^{(-)}-\sqrt{f(0)}}\Big]$ where the interval may be potentially empty, in which case we are not concerned with this case. 
	\end{enumerate}
	These correspond to cases 1b, 1 c(i), 1c(ii), 1c(iv), 1c(v) and 2, in the proof of Proposition \ref{prop:EX2}. 
	
	The assumption of the Lemma guarantees we are not in cases 1c(iv), 1c(v); For these cases $c_{1/2}^{(-)} \le \frac{r\sqrt{f(0)}}{r-1} \le c_{1/2}^{(+)}$ which is equivalent to  
	\[
	r \in \Big[ \frac{c_{1/2}^{(+)}}{c_{1/2}^{(+)}-\sqrt{f(0)}}, \frac{c_{1/2}^{(-)}}{c_{1/2}^{(-)}-\sqrt{f(0)}} \Big].
	\]
	
	In cases 1b, 1c(i), 1c(ii) and 2,  the fact that $\varlimsup_{a\to 0} m_2(a) = + \infty$ is independent of which sequence of $a_k$ we select, as long as it tends to 0. Therefore we can select $a_k$ to be sequence that corresponds to the first coordinate of crossing points and which tends to 0, since by Lemma \ref{lem:AS} we know they are dense on $f$.   
\end{proof}
\end{proof}

\begin{proof}[Proof of Theorem \ref{thm:regionLM}] We only prove the theorem for $a \to 0$, as the case $a \to a_0$ is analogous.

The direction $(2) \Longrightarrow (1)$ is immediate; the condition implies that all points $(x,y) \in \text{int}(\R^2_+)$ are optimised by a type C maximiser, and by letting $x \to 0$ while keeping $y > f(0)$ fixed, the crossing points $(a_{x,y}, f(a_{x,y}))$ tend to $(0, f(0))$. This forces $m_2(a_{x,y})$ to $+\infty$. 

Now for $(1) \Longrightarrow (2)$. Assume that $\varlimsup_{a \to 0} m_2(a) = + \infty$ and assume by way of contradiction that ${\rm int}(R_{0, f(0)}) \neq \varnothing$. 

Then we can find a sequence of points 
$(x_k, y_k) \in \R^2_+ \setminus R_{0, f(0)}$ with $(x_k, y_k) \to (0, f(0))$ so that
%$|\mathcal M_{(x_k, y_k)}| \ge 2$. We restrict the subsequence of these points further so that 
\begin{enumerate} 
\item For each $k$, the crossing points $\{(a_k, f(a_k))\}_k$ of a maximiser that does not follow the axis are different; this is possible because the crossing points are dense on the curve.
\item The limit $\lim_{k \to \infty}m_2(a_k) = + \infty$.
\end{enumerate}
This can be done by Lemma \ref{lem:anothergap}.

Now, by Proposition \ref{lem:NAP}-(2), we have that for any point $(x,y) \in {\rm int}(R_{0, f(0)})$ on the line segment 
$\ell_k: (a_k, f(a_k)) - (x_k, y_k)-(x,y)$ the limiting passage time satisfies 
 \[
 I(\pi_{0, (0,f(0))}) + I(\pi_{(0,f(0)), (x,y)}) > I(\pi_{0, (a_k,f(a_k))}) + I(\pi_{(a_k,f(a_k)), (x,y)}). 
 \]
For notational convenience set $\e = a_k$ and notice that the relation above stays true when we let $(x,y)$ tend to infinity, \emph{along the line which contains the segment} $\ell_k$. We substitute the explicit values for $I(\pi)$ in the display above to obtain
\begin{align}
&f(0)+\frac{1}{r}(\sqrt{x}+\sqrt{y-f(0)})^2 > (\sqrt{\varepsilon}+\sqrt{f(\varepsilon)})^2+\frac{1}{r}(\sqrt{x-\varepsilon}+\sqrt{y-f(\varepsilon)})^2\label{claim}.
\end{align}
Call $m_1(\e) = \frac{f(\e)}{\e}$, $m_2(\varepsilon)=\frac{y-f(\varepsilon)}{x-\varepsilon}$ and $m=\frac{y-f(0)}{x}$ and note that $m_2(\varepsilon)>m$. Both slopes are always finite for every $(x,y)\in(0,a_0)\times\mathbb{R}_+$.
Inequality \eqref{claim} is then re-written as 
\begin{equation}\label{comp}
\frac{1}{r}\Big[x\Big(1+\sqrt{m}\Big)^2- x\Big(1+\sqrt{m_2(\varepsilon)}\Big)^2\Big]> \varepsilon+f(\varepsilon)-f(0)+2\sqrt{\varepsilon f(\varepsilon)} - \frac{\e }{r}\big(1+\sqrt{m_2(\varepsilon)}\big)^2.
\end{equation}
Since the point $(x,y)$ belongs to the line $y=m_2(\varepsilon)(x-\varepsilon)+f(\varepsilon)$, taking $x\to\infty$ gives $m\to m_2(\varepsilon)$.
We first manipulate the left-hand side of \eqref{comp}.
\begin{align*}
x\Big[\Big(1+&\sqrt{m}\Big)^2-\Big(1+\sqrt{m_2(\varepsilon)}\Big)^2\Big]=x\big[2(\sqrt{m}-\sqrt{m_2(\varepsilon)})+m-m_2(\varepsilon)\big]\\
&=\frac{x(f(\varepsilon)-f(0)-\varepsilon m_2(\varepsilon))+\varepsilon (\varepsilon m_2(\varepsilon)-f(\varepsilon)+f(0))}{x-\varepsilon}\Big[1+\frac{2}{\sqrt{m}+\sqrt{m_2(\varepsilon)}}\Big].
\end{align*}
Now take the limit $x \to \infty$ in \eqref{comp}. After that, and some algebraic operations, we get that the limiting version of \eqref{comp} is 

\begin{align}
\frac{1}{r}\Big( \frac{1}{\sqrt{m_2(\e)}} + 1 - r \Big) \frac{f(\e) - f(0)}{\e} &\ge 1 - \frac{1}{r} + 2 \sqrt{m_1(\e)} - \frac{\sqrt{m_2(\e)}}{r} \notag \\
&= \sqrt{m_1(\e)} + \Big(\frac{r-1}{r} + \sqrt{m_1(\e)} - \frac{\sqrt{m_2(\e)}}{r} \Big). \label{eq:prelm}
\end{align}
This is the point where we are using the fact that $(\e, f(\e))$ is a crossing point: Utilize the relation of equation \eqref{explain} to change the last parenthesis in \eqref{eq:prelm} and obtain the equivalent inequality
\[
\frac{1}{r}\Big( \frac{1}{\sqrt{m_2(\e)}} + 1 - r \Big) \frac{f(\e) - f(0)}{\e} \ge \sqrt{m_1(\e)} - \frac{f'(\e)}{r}\Big( r-1 + \frac{r}{\sqrt{m_1(\e)}} - \frac{1}{\sqrt{m_2(\e)}} \Big),
\]
or equivalently 
\be \label{eq:pre-analysis}
\frac{1}{r}\Big( \frac{1}{\sqrt{m_2(\e)}} + 1 - r \Big)\Big( \frac{f(\e) - f(0)}{\e} - f'(\e)\Big) \ge \sqrt{m_1(\e)} - \frac{f'(\e)}{\sqrt{m_1(\e)}}.
\ee 

Now, if equation \eqref{eq:pre-analysis} is violated, we automatically reach a contradiction to the assumption that ${\rm int}(R_{0, f(0)})  \neq \varnothing$. We will show precisely this
by splitting the analysis into cases: 

\begin{enumerate}[(1)]
\item $\lim_{\e\to 0}f'(\e) = c_0$: Then as $\e \to 0$, the left-hand side of \eqref{eq:pre-analysis} converges to $0$ while the right-hand side tends to $\infty$. This gives the desired contradiction.
\item $r < 1$: In this case, select an $\e$ small enough so that $ \frac{1}{\sqrt{m_2(\e)}} +1 - r > 0$. The convexity and monotonicity of $f$ imply that $ \frac{f(\e) - f(0)}{\e} - f'(\e) < 0$ so the left-hand side of \eqref{eq:pre-analysis} is negative while the right-hand is strictly positive. This gives again a contradiction. 
\item $r > 1, \alpha < 1/2$: Since $\alpha < 1/2$, we have that for $\delta$ small, $\alpha + \delta < 1/2$. Then, using definition \eqref{eq:order}, for any $\eta$ small, we can find $\e_0$ so that for all $\e < \e_0$
\[ -f'(\e) < \frac{\eta}{ \e^{\alpha + \delta}}.\]
Integrating the inequality from $0$ to $\e$ we get
\[
f(0) - f(\e) < \frac{\eta}{1 - \alpha - \delta} \e^{1 - \alpha - \delta} < c \sqrt{\e}.
\]
The last inequality is true for any constant $c$, as long as $\e$ is small enough. We pick $c < \sqrt{\frac{f(0)}{2}}$ and reduce $\e$ further so that $f(\e) > \frac{f(0)}{2}$. We then have for all $\e$ small that
\[
\frac{f(0) - f(\e)}{\e} < \sqrt{\frac{f(\e)}{\e}}=\sqrt{m_1(\e)}.
\]
Reduce $\e$ even more, so that $1/\sqrt{m_2(\e)} < \frac{r-1}{2}$. Then we bound
\begin{align*}
\frac{1}{r}\Big( \frac{1}{\sqrt{m_2(\e)}} + 1 - r \Big)\Big( \frac{f(\e) - f(0)}{\e} - f'(\e)\Big)
&= \frac{1}{r}\Big( -\frac{1}{\sqrt{m_2(\e)}} - 1 + r \Big)\Big( \frac{f(0) - f(\e)}{\e} + f'(\e)\Big)\\
&< \frac{1}{r}\Big( r-1 -\frac{1}{\sqrt{m_2(\e)}} \Big)\Big( \frac{f(0) - f(\e)}{\e} - \frac{f'(\e)}{\sqrt{m_1(\e)}}\Big)\\
&< \frac{r-1}{r}\Big( \sqrt{m_1(\e)} - \frac{f'(\e)}{\sqrt{m_1(\e)}}\Big)\\
&<  \sqrt{m_1(\e)} - \frac{f'(\e)}{\sqrt{m_1(\e)}},
\end{align*}
which is a direct violation of \eqref{eq:pre-analysis}.
\end{enumerate}

The remaining proof is for when $\alpha = 1/2$.  In this case we have that $\varlimsup m_2(a_k) \to \infty$ for any sequence $a_k \to 0$ and $r \notin \Big[ \frac{c_{1/2}^{(+)}}{c_{1/2}^{(+)}-\sqrt{f(0)}}, \frac{c_{1/2}^{(-)}}{c_{1/2}^{(-)}-\sqrt{f(0)}} \Big]$. 

\begin{enumerate}[(4)]
\item We further impose on the subsequence of $a_k$ that

	 $a_k^{1/2}|f'(a_k)| \to c_{1/2} \le c_{1/2}^{(+)}  < \frac{r}{r-1}\sqrt{f(0)}$ by the assumption. Here $c_{1/2}$ can be any limit point. 

	For any $\delta > 0$ we can find a 
	$K = K(\delta)$ so that for all $k > K$ we have
	\[
	\frac{r-1}{r} (c_{1/2} + 3\delta) < \sqrt{f(0)}  - \delta < \sqrt{f(a_k)}, \quad  |a_k^{1/2}f'(a_k) +c_{1/2}|  < \delta.
	\]
	 The first inequality above is true for $\delta$ sufficiently small.
	 Then we estimate, as in case (3), that 
	\[
	-f'(a_k) < (c_{1/2}+\delta) a_k^{-1/2},  \quad \text{for all $k > K$ by construction}
	\]
	which implies that 
	\[
	\frac{f(0) - f(a_k)}{a_k} < 2(c_{1/2}+\delta)a_k^{-1/2}.
	\]
	Then use the inequalities above to bound 
	\begin{align*}
\frac{1}{r}\Big( -\frac{1}{\sqrt{m_2(a_k)}} - 1 + r \Big)&\Big( \frac{f(0) - f(a_k)}{a_k} + f'(a_k)\Big)\\
&< \frac{1}{r}\Big( r-1 -\frac{1}{\sqrt{m_2(a_k)}} \Big)\Big( 2(c_{1/2}+\delta)a_k^{-1/2} + f'(a_k)\Big)\\
&< \frac{1}{r}\Big( r-1 -\frac{1}{\sqrt{m_2(a_k)}} \Big)\frac{\Big(c_{1/2} + 3 \delta\Big)}{a_k^{1/2}}\\
&< \frac{1}{r}\Big( r-1 -\frac{1}{\sqrt{m_2(a_k)}} \Big)\frac{r}{r-1}\frac{ \sqrt{f(a_k)}}{a_k^{1/2}}\\
&<  \sqrt{m_1(a_k)} - \frac{f'(a_k)}{\sqrt{m_1(a_k)}},
\end{align*}
which also contradicts \eqref{eq:pre-analysis}. The last inequality follows immediately from the fact that $f ' < 0$. \qedhere
\end{enumerate}
	\end{proof}

\begin{proof}[Proof of Theorem \ref{thm:gap}] The proof is identical to that of case (4) in the 
proof of Theorem \ref{thm:regionLM}. The reason we cannot apply the argument directly is the fact that we do not know a priori that the $\varlimsup_{a_k\to0} m_2(a_k) = 0$ on a sequence of \textit{crossing points}, since Lemma \ref{lem:anothergap} does not apply here. This condition is now taken care by the assumption of Theorem \ref{thm:gap}.

To finish the proof, impose on this sequence $\{ a_k \}_{k \in \N}$ of crossing points the extra condition that   $a_k^{1/2}|f'(a_k)| \to c_{1/2} < \frac{r}{r-1}\sqrt{f(0)}$ by the assumption. Again, $c_{1/2}$ can be any limit point.  Now the calculation for (4) in the proof of Theorem \ref{thm:regionLM} can be repeated and it finishes the proof. 
\end{proof}

\subsection{Phase transition at $c_{1/2}^{(-)} = \frac{r}{r-1}\sqrt{f(0)}$}

\begin{proposition}[Phase transition at $c_{1/2}^{(-)} = \frac{r}{r-1}\sqrt{f(0)}$] \label{prop:pt2}
Suppose that $c_{1/2}^{(-)} = \frac{r}{r-1}\sqrt{f(0)}$ and assume that for some $\gamma > 0$ and some $c\in \R$, 
\be \label{eq:-fprime}
- f'(a) = c_{1/2}^{(-)}  a^{-1/2} + c  a^{\gamma - \nicefrac12}. 
\ee
Then, when $\gamma < 1/4$ the equivalence of Theorem \ref{thm:gap} is false when $c < 0$ and true when $c > 0$. When $\gamma > 1/4$, type B maximisers exist.
\end{proposition}

We first need a geometric lemma:
\begin{lemma} \label{lem:finale} Assume that $R_{0, f(0)} = \{0\}\times[f(0), \infty)$ and $R_{a_0, 0} = [a_0, \infty)\times\{ 0 \}$  (i.e. they are both degenerate). Then, there exists a sequence of points $(x_k, y_k)$ with $x_k \to \infty$ as $k\to \infty$, so that their corresponding crossing points $(\beta_k, f(\beta_k)) \to (0, f(0))$.  
\end{lemma}

\begin{proof}[Proof of Lemma \ref{lem:finale}]
Suppose by way of contradiction that there exists a constant $A>0$ so that for all $(x,y) \in \R^2_+$ with $x > A$, the crossing points 
$(a_{x,y}, f(a_{x,y}))$ satisfy $a_{x, y} > \alpha_A > 0$. 

Fix an $\alpha > 0$ small and define 
\be\label{eq:x+}
 x_+ =x_{+} (\alpha)= \sup\{ x : \exists \, y \text{ so that the crossing point $(a_{x,y}, f(a_{x,y}))$ satisfies } a_{x,y}  \le \alpha\}.
\ee
The assumption guarantees that  $x_{+} (\alpha)$ is bounded for $\alpha$ small enough, and the set for which we take the supremum is not empty, since crossing points are dense on the graph of $f$ by Lemma \ref{lem:AS}.

For any $\delta > 0$ define the terminal point $(x_{\delta}, y_{\delta}) = (x_+ - \delta, y_\delta)$ to be such that its crossing point satisfies 
$a_{x_\delta, y_{\delta}} \le \alpha$. Then it has to be that for all points $(x_+ - \delta, y)$ with $y > y_{\delta}$ their corresponding crossing points 
has to satisfy  $a_{x_\delta, y} \le \alpha$. If this is not true, then the maximal path for  $(x_+ - \delta, y)$ would cross the one for  $(x_{\delta}, y_{\delta})$ and this is impossible by Lemma \ref{lem:UMP}. 

Now there are three cases to consider:
\begin{enumerate}[(1)]
\item  $x_+ >  a_0$: In this case, consider now a point $(x_+ + \e, y_0)$, for some small $\e>0$. Because of its $x$-coordinate, this point must have a crossing point with first coordinate larger than $\alpha$. The maximal possible slope for its second segment is $m_{\max} =\frac{ y_0 }{x_+ + \e - a_0}$. Now notice that for $y_0$ large enough, the line   $y = m_{\max} (x - a_0)$ must intersect the optimal path from 0 to $(x_{\delta}, y_{\delta})$ by planarity. In particular, the maximal paths to $(x_{\delta}, y_{\delta})$ and $(x_+ + \e, y_0)$ must intersect in the $r$-region, and this violates Lemma \ref{lem:UMP}. 

\item $x_{+} = a_0$: The same arguments as in case (1) give that the only possible crossing point for $(x_++\e, y_0)$ when $y_0$ is large enough is $(a_0, 0)$ otherwise maximal paths would intersect. This contradicts the assumption that  $R_{a_0, 0} = [a_0, \infty)\times\{ 0 \}$. 
  
\item  $x_+ < a_0$: This is the most challenging case, and we need to split it into yet two more cases.
\begin{enumerate}[(a)]
	\item $x_+(\alpha)$ is a maximum. Assume that $(x_+(\alpha), y_+(\alpha))$ is point with the crossing point of its maximiser less than $\alpha$. Now, for any $\delta, \e>0$, we can find $y_1 >  y_+(\alpha)$ so that the point $(x_++\e, y_1)$ has crossing point 
	$a_{x_+ +\e} \ge x_+ -\delta$. This is because maximal macroscopic paths cannot cross, and any point $(x_++\e, y_1)$ has to have a maximiser with crossing point with $a_{x_+ +\e} > \alpha$. 
	Suppose by way of contradiction that the crossing point $a_{x_++\e, y_1} \le x_+$.
	Keeping $\e>0$ but raising the value of $y_1$, we can find a crossing point larger than 
	$a_{x_++\e, y_1}$. But that would mean that maximisers cross, which cannot happen. 	Therefore, the crossing point $a_{x_++\e, y_1} > x_+$. 
	This has to be true for all values of $y_1$, and it is true for all $\e >0$.

	Now we want to understand the behaviour of the maximal paths when $\e \to 0$ as $y_1$ remains fixed. For each point $(x_++\e, y_0)$ let $(a_\e, f(a_\e))$ the corresponding crossing point.  For all $\e$, $a_\e > x_+$ and since maximal paths cannot cross each other, $\lim_{\e\to 0} a_\e = x_+$: Then, as $\e \to 0$ and by continuity of $\Gamma$ (Theorem \ref{cor:conti}),  $\Gamma(x_+, y_0)$ must also be optimised by the path $0 \to (x_+, f(x_+)) \to  (x_+, y_0)$. By Lemma \ref{lem:dense} this is impossible.
	
	\item $x_+(\alpha)$ is a supremum but not a maximum. Then consider terminal points of the form $(x_+, y)$, and their crossing points 
	$(a_{x_+,y}, f(a_{x_+,y}))$. Notice that for all $y$ large enough we must have  
	\[ a_{x_+,y} \in (x_+ -\delta, x_+) .\]
	Set that $y$ as $y_1$. Now, for all $y > y_1$,  we have that $a_{x_+,y} \in (x_+ -\delta, a_{x_+,y_1})$.
	This is because the maximal paths cannot cross, by Lemma \ref{lem:UMP}. 
	Now consider a terminal point $(x_2, y_2)$ so that $ a_{x_+,y_1} < x_2 < x_+ $, $y_2 > y_1$ and $a_{x_2, y_2} \le \alpha$. 
	Finally, find a $y_3>y_2$ so that $(x_+, y_3)$ has a crossing point with $a_{x_+, y_3} \ge x_2$. But this implies that 
	\[
	 a_{x_+,y_1} < a_{x_+, y_3}, \text{ while } y_3 > y_1,
	\] 
	and in particular it means maximal paths cross. This cannot happen, so we reached a contradiction. \qedhere
	\end{enumerate}

\end{enumerate}  
\begin{figure}
\center{
\begin{tikzpicture}[
        xscale=2.5,
        IS/.style={blue, thick},
        LM/.style={red, thick},
        axis/.style={ ->, >=stealth', line join=miter},
        important line/.style={thick}, dashed line/.style={dashed, thin},
        every node/.style={color=black},
        dot/.style={circle,fill=black,minimum size=4pt,inner sep=0pt,
            outer sep=-1pt},
    ]
    % axis
    \draw[axis,<->] (3.5,0) node(xline)[xshift=0.5em,yshift=-0.5em] {\small$x$} -|
                    (0,5.5) node(yline)[xshift=-0.75em,yshift=0.25em] {\small$y$};
  
    \draw[IS] (0,3) coordinate (IS_1) parabola[bend at end]
         (3,0) coordinate (IS_2);
 \draw (1.2,2) node{$r$};
 \draw (0.4,0.5) node{$1$};
 \draw (3,0) node[xshift=0.5em,yshift=-0.75em]{\small$a_0$};
 \draw (0,3) node[xshift=-1.25em,yshift=0.25em] {\small$f(0)$};
\draw[dashed] (2.5,0)node[xshift=0.5em,yshift=-0.85em] {\tiny$x_+$}--(2.5,5.5);
\draw[dashed] (0.9,0)node[xshift=0.2em,yshift=-0.85em] {\tiny$x_+-\delta$}--(0.9,5.5);
\draw[dashed] (1.3,0)node[xshift=0.75em,yshift=-0.85em] {\tiny$a_{x_+,y_1}$}--(1.3,5.5);
\draw[dashed] (0.42,0)node[xshift=0.1em,yshift=-0.85em] {\tiny$a_{x_2,y_2}$}--(0.42,5.5);
\draw[dashed] (2,0)node[xshift=0.5em,yshift=-0.85em] {\tiny$x_2$}--(2,5.5);
\draw[nicos-red, dashed] (1.3,0.95)--(2.5,4.464);
\draw[nicos-red, line width=2 pt] (0,0)--(1.3,0.95)--(2.5,2.6)node[xshift=1.5em]{\tiny$(x_+,y_1)$};
\draw[my-green, line width=2 pt] (0,0)--(0.42,2.2)--(2,3)node[xshift=1.5em]{\tiny$(x_2,y_2)$};
\draw[my-blue, line width=2 pt] (0,0)--(1.8,0.47)--(2.5,5)node[xshift=1.5em]{\tiny$(x_+,y_3)$};
\draw[line width=1 pt,yscale=2.5] (2,0.75)circle(1.5mm);
\end{tikzpicture}
}
\caption{Construction in the proof of Lemma \ref{lem:finale}, part 3(b).}
\label{fig:2MP}
\end{figure}
\end{proof}

\begin{proof}[Proof of Proposition \ref{prop:pt2}.]
When $f'$ satisfies \eqref{eq:-fprime}, we have that $c_{1/2}^{(-)} = c_{1/2}^{(+)}$. This implies that for any sequence $a_k \to 0$ we will simultaneously have 
\[
a_k \to 0, \quad a_k^{1/2}|f'(a_k)| \to c_{1/2}^{(-)} \text{ and } m_2(a_k) \to \infty.
\]
In particular this will be true on a sequence $a_k$ coming from crossing points.

Fix any such sequence.
Integrate both sides of \eqref{eq:-fprime} and divide by $a_k$ to obtain
\be\label{eq:last}
\frac{f(0) - f(a_k)}{a_k} = 2c_{1/2}^{(-)}  a_k^{ - 1/2} + \frac{c}{\gamma + \nicefrac12}a_k^{\gamma - \nicefrac12}.
\ee
Moreover, we have that 
\[
(\sqrt{f(0)} - \sqrt{f(a_k)})(\sqrt{f(0)} + \sqrt{f(a_k)}) = 2c_{1/2}^{(-)}  a_k^{1/2} + \frac{c}{\gamma + \nicefrac12}a_k^{\gamma + \nicefrac12}.
\]
Equation \eqref{eq:-fprime} also implies that $- a_k^{1/2}f'(a_k) = c_{\alpha}^{(-)} + c a_k^{\gamma}$. Now we are in position to estimate
	\begin{align*}
\frac{1}{r}\Big( -\frac{1}{\sqrt{m_2(a_k)}} - 1 + r \Big)&\Big( \frac{f(0) - f(a_k)}{a_k} + f'(a_k)\Big)\\
&= \frac{1}{r}\Big( r-1 -\frac{1}{\sqrt{m_2(a_k)}} \Big)\Big( 2c_{1/2}^{(-)}  a_k^{ - 1/2} + \frac{c}{\gamma + \nicefrac12}a_k^{\gamma - \nicefrac12} + f'(a_k)\Big)\\
&= \frac{1}{r}\Big( r-1 -\frac{1}{\sqrt{m_2(a_k)}} \Big)\frac{\Big( 2c_{1/2}^{(-)}+ \frac{c}{\gamma + \nicefrac12}a_k^{\gamma} + a_k^{1/2} f'(a_k)\Big)}{a_k^{1/2}}\\
&= \frac{1}{r}\Big( r-1 -\frac{1}{\sqrt{m_2(a_k)}} \Big)\frac{\Big( 2c_{1/2}^{(-)}+ \frac{c}{\gamma + \nicefrac12}a_k^{\gamma} - c_{1/2}^{(-)}  -  c  a_k^{\gamma }\Big)}{a_k^{1/2}}\\
&= \frac{r-1}{r}\frac{\Big( c_{1/2}^{(-)}+ c \frac{1/2-\gamma}{\nicefrac12 + \gamma } a_k^{\gamma}\Big)}{a_k^{1/2}} - \frac{1}{r}\frac{1}{\sqrt{m_2(a_k)}} \frac{\Big( c_{1/2}^{(-)}+ c \frac{1/2-\gamma}{\nicefrac12 + \gamma } a_k^{\gamma}\Big)}{a_k^{1/2}}.
\end{align*}
In the last line there are two competing terms; one is asymptotically positive and the other asymptotically negative so we must treat them separately: First the higher order positive term
\begin{align*}
 \frac{r-1}{r}\frac{\Big( c_{\alpha}^{(-)}+ c \frac{1/2-\gamma}{\nicefrac12 + \gamma } a_k^{\gamma}\Big)}{a_k^{1/2}}
&= \frac{\sqrt{f(0)}}{a_k^{1/2}} +\frac{r-1}{r}c\frac{\nicefrac12-\gamma}{\nicefrac12 + \gamma}\frac{a_k^{\gamma}}{a_k^{1/2}}\\
&= \frac{\sqrt{f(a_k)}}{a_k^{1/2}} + \frac{\sqrt{f(0)} -\sqrt{f(a_k)}}{a_k^{1/2}}  + c \frac{r-1}{r}\frac{\nicefrac12-\gamma}{\nicefrac12 + \gamma}\frac{a_k^{\gamma}}{a_k^{1/2}}\\
&= \frac{\sqrt{f(a_k)}}{a_k^{1/2}} + \frac{2c_{1/2}^{(-)} + \frac{c}{\gamma + \nicefrac12}a_k^{\gamma }}{\sqrt{f(0)} + \sqrt{f(a_k)}}  +c \frac{r-1}{r}\frac{\nicefrac12-\gamma}{\nicefrac12 + \gamma}\frac{a_k^{\gamma}}{a_k^{1/2}}.
\end{align*}
Then we work with the negative term. First we perform an asymptotic expansion on $1/\sqrt{m_2(a)}$ as $a$ tends to $0$:
\be\label{m2ext}
\frac{1}{\sqrt{m_2(a)}} = \begin{cases} 
\vspace{0.1in}
\frac{1}{|c |(r-1)}a^{1/2-\gamma} +  O(a^{1/2}), &\quad \gamma \in (0, 1/2),  \\
\frac{a^{1/4}}{\sqrt{c_{1/2}^{(-)}}}+O(a^{1/2}), &\quad  \gamma\in[1/2,\infty). 
\end{cases}
\ee
The details for \eqref{m2ext} can be found in the Appendix. Using this expansion we obtain 
\begin{align*}
\frac{1}{r} \frac{1}{\sqrt{m_2(a_k)}}\frac{\Big( c_{1/2}^{(-)}+ c \frac{1/2-\gamma}{\nicefrac12 + \gamma } a_k^{\gamma}\Big)}{a_k^{1/2}}
&= \frac{1}{r} \frac{\Big( c_{1/2}^{(-)}+ c \frac{1/2-\gamma}{\nicefrac12 + \gamma } a_k^{\gamma}\Big)}{a_k^{1/2}}\times
\begin{cases} 
\vspace{0.1in}
\frac{1}{|c |(r-1)}a_k^{1/2-\gamma} +  O(a_k^{1/2}), & \gamma \in (0, 1/2),  \\
\frac{a_k^{1/4}}{\sqrt{c_{1/2}^{(-)}}}+O(a_k^{1/2}), &  \gamma\in[1/2,\infty).  
\end{cases}\\
&=\begin{cases}
\vspace{0.1in}
\frac{c_{1/2}^{(-)}}{|c |r (r-1)}a_k^{-\gamma} +  O(1), & \gamma \in (0, 1/2),  \\
\frac{\sqrt{c_{1/2}^{(-)}}}{r a_k^{1/4}}+O(1), & \gamma\in[1/2,\infty).
\end{cases}
\end{align*}
Combining the two expansions we have 
\begin{align}
\frac{1}{r}\Big( -\frac{1}{\sqrt{m_2(a_k)}} - 1 + r \Big)&\Big( \frac{f(0) - f(a_k)}{a_k} + f'(a_k)\Big)\notag\\
&= \frac{\sqrt{f(a_k)}}{a_k^{1/2}}  +c \frac{r-1}{r}\frac{\nicefrac12-\gamma}{\nicefrac12 + \gamma}{a_k^{\gamma -1/2}}- \begin{cases}
\vspace{0.1in}
\frac{c_{1/2}^{(-)}}{|c |r (r-1)}a_k^{-\gamma} +  O(1), & \gamma \in (0, 1/2),  \\
\frac{\sqrt{c_{1/2}^{(-)}}}{r a_k^{1/4}}+O(1), & \gamma\in[1/2,\infty).
\end{cases} \label{eq:ana}
\end{align}
Now the phase transition reveals itself. First  when $\gamma > 1/4$, the leading order terms in \eqref{eq:ana} are those in the brace; they are negative and tend to $-\infty$, so as before, \eqref{eq:pre-analysis} is violated. 

Now assume $1/4 \ge \gamma$. This means $1/2 - \gamma \ge \gamma$. Then, If $c > 0$, the middle term in \eqref{eq:ana} tends to $+ \infty$, and immediately gives a contradiction to \eqref{eq:pre-analysis}.

If $c < 0$ with a sufficiently large modulus (if $\gamma < 1/4$ any $c <0$ will do), we have for all $a_k$ sufficiently small that \eqref{eq:ana} can be bounded by
\begin{align}
\frac{1}{r}\Big( -\frac{1}{\sqrt{m_2(a_k)}} - 1 + r \Big)&\Big( \frac{f(0) - f(a_k)}{a_k} + f'(a_k)\Big) > \frac{\sqrt{f(a_k)}}{a_k^{1/2}}  +\frac{c}{2} \frac{r-1}{r}\frac{\nicefrac12-\gamma}{\nicefrac12 + \gamma}{a_k^{\gamma -1/2}}\label{eq:anasta}\\
&\phantom{xxxxxxxxx}>  \frac{\sqrt{f(a_k)}}{a_k^{1/2}}  - \frac{f'(a_k)}{\sqrt{m_1(a_k)}} + \frac{c}{4} \frac{r-1}{r}\frac{\nicefrac12-\gamma}{\nicefrac12 + \gamma}{a_k^{\gamma -1/2}}. \label{eq:anastasi}
\end{align}
Compare \eqref{eq:anastasi} with equation \eqref{eq:pre-analysis}. The only difference is the last term on the right-hand side, which for $c < 0$ and $\gamma < 1/4$ it is a positive term that goes to $+\infty$ as $a_k \to 0$. 

Assume by way of contradiction that in this case $R_{0, f(0)}$ is degenerate. Then we can find a sequence of terminal points $(x_k, y_k)$ with $x_k \to \infty$ (as $k \to \infty$)
with corresponding crossing points $(\beta_k, f(\beta_k)) \to (0, f(0))$ by Lemma \ref{lem:finale}. Then it must be that $m_2(\beta_k) \to \infty$ and we may assume without loss of generality that $m_2(\beta_k)$ is strictly increasing.

Assume $x_k$ is large enough so that $ \frac{x_{k}}{x_k - \beta_k} -1< A \beta_k $ for some constant $A$. Moreover we have the relations 
\[
m_1(\beta_k) = \frac{f(\beta_k)}{\beta_k}, \quad m_2(\beta_k) = \frac{y_k - f(\beta_k)}{x_k - \beta_k}, \quad m(\beta_k) = \frac{y_k - f(0)}{x_k} \text{ and } y_k = m_2(\beta_k)(x_k - \beta_k) + f(\beta_k).
\]
Since we are assuming that the region $R_{0, f(0)}$ is degenerate, the weight collected on a piecewise linear path that goes though $(0, f(0))$ and then to $(x_k, y_k)$ must be less than the weight collected on the path from the crossing point. As such, the same calculation that led to \eqref{comp}, now gives the inequality
\begin{align}\label{eq:crazy}
&\frac{1}{r}\frac{x_k(f(\beta_k)-f(0)-\beta_km_2(\beta_k))+\beta_k (\beta_k m_2(\beta_k)-f(\beta_k)+f(0))}{x_k-\beta_k}\Big[1+\frac{2}{\sqrt{m(\beta_k)}+\sqrt{m_2(\beta_k)}}\Big] \\
&\hspace{5.5cm}< \beta_k+f(\beta_k)-f(0)+2\sqrt{\beta_k f(\beta_k)} - \frac{\beta_k }{r}\big(1+\sqrt{m_2(\beta_k)}\big)^2.\notag
\end{align}
In the left hand side use the bounds  $1 < \frac{x_{k}}{x_k - \beta_k} < 1 + A\beta_k $ and $m_2(\beta_k) > m(\beta_k)$ to bound from below
 \begin{align*}
&\frac{1}{r}(f(\beta_k)-f(0)-\beta_km_2(\beta_k))(1+A\beta_k)\Big[1+\frac{2}{\sqrt{m(\beta_k)}+\sqrt{m_2(\beta_k)}}\Big]\\ 
&\hspace{3cm}+\frac{1}{r}\frac{ \beta_k (\beta_k m_2(\beta_k)-f(\beta_k)+f(0))}{x_k-\beta_k}\Big[1+\frac{1}{\sqrt{m_2(\beta_k)}}\Big] \\
&\hspace{5.5cm}< \beta_k+f(\beta_k)-f(0)+2\sqrt{\beta_k f(\beta_k)} - \frac{\beta_k }{r}\big(1+\sqrt{m_2(\beta_k)}\big)^2.\notag
\end{align*}
Using equation \eqref{m2ext},  we have that $\beta_k m_2(\beta_k) \to 0 $, so we simplify the inequality above one more time as
\begin{align}\label{eq:crazy4}
&\frac{1}{r}(f(\beta_k)-f(0)-\beta_km_2(\beta_k))\Big[1+\frac{2}{\sqrt{m(\beta_k)}+\sqrt{m_2(\beta_k)}}\Big] + O(\beta_k)\\ 
&\hspace{5.5cm}< \beta_k+f(\beta_k)-f(0)+2\sqrt{\beta_k f(\beta_k)} - \frac{\beta_k }{r}\big(1+\sqrt{m_2(\beta_k)}\big)^2.\notag
\end{align}

We finally use the estimate 
\[ |\sqrt{m(\beta_k)} - \sqrt{m_2(\beta_k)} | \le C_x \sqrt{m_2(\beta_k)}(f(0) - f(\beta_k)) \le C'_x \beta_k^{1/2} .\]
The last inequality comes from \eqref{eq:last}. We use this for one last simplification in \eqref{eq:crazy4} to
\begin{align*}
&\frac{1}{r}(f(\beta_k)-f(0)-\beta_km_2(\beta_k))\Big[1+\frac{1}{\sqrt{m_2(\beta_k)}}\Big] + O(\beta_k)\\ 
&\hspace{5.5cm}< \beta_k+f(\beta_k)-f(0)+2\sqrt{\beta_k f(\beta_k)} - \frac{\beta_k }{r}\big(1+\sqrt{m_2(\beta_k)}\big)^2.\notag
\end{align*}
With the same algebraic manipulations that led to \eqref{eq:pre-analysis}, we obtain 
\be \label{eq:pre-analysis2}
\frac{1}{r}\Big( \frac{1}{\sqrt{m_2(\beta_k)}} + 1 - r \Big)\Big( \frac{f(\beta_k) - f(0)}{\beta_k} - f'(\beta_k)\Big) \le \sqrt{m_1(\beta_k)} - \frac{f'(\beta_k)}{\sqrt{m_1(\beta_k)}} + O(1).
\ee 
This gives the desired contradiction, since equality \eqref{eq:pre-analysis2} is precisely opposite of inequality \eqref{eq:anastasi}.
\end{proof}

\begin{example}[An exactly solvable corner-step model: $(g(a,b) =  \sqrt{a} + \sqrt{b}, k =1$] 
\label{ex:solvable}
\end{example}
We have that $\partial_bg/\partial_ag = 1/\sqrt{m_1}$ and therefore $D = 0$. Then 
\[
m_2 = \frac{4}{ \Big( \frac{\partial_bg}{\partial_ag} - 1 + \sqrt{\Big(\frac{\partial_bg}{\partial_ag} + 1\Big)^2}\Big)^2}=\left(\frac{\partial_ag}{\partial_bg}\right)^2=m_1.
\]
Therefore, the optimal paths are straight lines and the last passage time can be explicitly computed for any $(x,y)$. If $(x,y)$ are such so that $\sqrt{x} + \sqrt{y} > 1$ the common optimal slope will be $m  = y/x\in \R_+$. The crossing point is given by 
\be
(a^*,b^*)=\Big(\frac{x}{(\sqrt{x}+\sqrt{y})^2},\frac{y}{(\sqrt{x}+\sqrt{y})^2}\Big),
\ee 
and the last passage time shape function can be computed to be 
\begin{align*}
\Gamma_{c_g,1}(x, y) &=\begin{cases} 
\Big( 1 - \frac{1}{r}\Big) +\frac{1}{r}(\sqrt{x} + \sqrt{y})^2, &\text{ if } \sqrt{x} + \sqrt{y} > 1\\
(\sqrt{x} + \sqrt{y})^2, &\text{ if } \sqrt{x} + \sqrt{y} \le 1.
\end{cases}
\end{align*}
 One can verify directly that going through the axes is not optimal and all maximisers have to cross the curve.
 
In fact, this is the unique case of a speed function with this form, for which the optimal paths are straight lines. Assume that always $m_2 = m_1=m = b/a$.  From equation \eqref{eq:mess2} we have 
\be \label{eq:diffiq}
0 =\nabla g (a,b) \cdot \bigg(\frac{1}{\sqrt{m}} +1, - (\sqrt{m} +1)\bigg) = \nabla g (a,b) \cdot \bigg(\frac{\sqrt{a}}{\sqrt{b}} +1, - \frac{\sqrt{b}}{\sqrt{a}} -1\bigg).
\ee 

Solve the differential equation \eqref{eq:diffiq} for $a$ and $b$ to conclude that there exists $c_1 , c_2 \in \R$ 
such that 
\[ 
g(a, b) = c_2 (\sqrt{a} + \sqrt{b})^2 + c_1.
\]
Then the level curve is enforced by \eqref{eq:a3c} and is given by $\sqrt{a} + \sqrt{b} = \alpha$ for some $\alpha = \alpha(k, c_1, c_2)$ in $\R_+$.\qed

\section{Continuity properties of $\Gamma(x,y)$ }
\label{sec:continuity}

Now, we want to study what happen to the difference of the macroscopic last passage time of two points that are very close to each other.

\begin{lemma}\label{lem:2}
Fix $a,b,z,w>0$ and a speed function $c$. Then there exists a constant $C=C(a, b,z,w,c(\cdot,\cdot))<\infty$ such that for any $\delta > 0$ we can find sufficiently small $\delta_1, \delta_2> 0$ so that
the following two regularity conditions hold:
For $0 \le a \le z$,
\be\label{eq:9}
\Gamma((a,0),(z+\delta_1,\delta_2))-\Gamma((a,0),(z,0))\leq C\sqrt{\delta}.
\ee
For $0\leq b\leq w,$
\be\label{eq:11}
\Gamma((0,b),(\delta_1,w+\delta_2))-\Gamma((0,b), w)\leq C\sqrt{\delta}.
\ee
\end{lemma}
\begin{proof}

The arguments will be symmetric, so we will prove only \eqref{eq:11}. Pick a $\delta$ positive.

First select  $\delta_1\in[0,1)$, $\delta_2\in[0,1)$ small enough such that 
\begin{enumerate}
\item Any discontinuity curve $h_i$ in $[0, \delta_1]\times [0, w + \delta_2]$ is monotone and their domain is the interval $[0, \delta_1]$.
\item The intersection points of the discontinuity curves in $[0, \delta_1]\times [0, w + \delta_2]$ (if any) all lie on the $y$-axis.
\end{enumerate}
The first one is possible since the $h_i$ are finitely many in any compact set, and piecewise monotone functions. The second one because there only finitely many intersections points. Let $H$ be the number of discontinuity curves in this rectangle, and enumerate them from the lowest to the highest, including the north and south straight boundaries.
Decrease $\delta_1$ further so that 
\[
 \max_{1 \le i \le H} \{ \om_{h_i}(\delta_1)\} < \delta 
\]
and select an $\eta = \eta(\delta_1)>0$ which satisfies the condition 
\[
\eta \le \min_{1 \le i \le H} \{ \om_{h_i}(\delta_1)\}.
\]

Keep in mind that $\eta \to 0$ as $\delta_1 \to 0$. Decrease $\delta_1$ further so that $H \eta << w$. Since $c(x,y)$ is piecewise constant, we have that in-between these discontinuity curves the rates are fixed, and on the discontinuity curve the value is the smallest of the rates in the two adjacent regions by condition (1) in Assumption \ref{ass:c2}.

From the hypotheses so far, we have that the rectangles 
$Q_i = [0, \delta_1]\times[ h_i(0)\wedge h_i(\delta_1), h_i(0)\vee h_i(\delta_1)],$ have completely disjoint interiors for all $1 \le i \le H$ and $c(x,y)$ takes two values. In the rectangles $R_{i} = [0, \delta_1] \times [h_{i}(0)\vee h_{i}(\delta_1), h_{i+1}(0)\wedge h_{i+1}(\delta_1)]$, the speed function is constant. We allow the rectangles $R_i, Q_i$ to be degenerate horizontal lines.

For any $\mathbf{x}=(x^1(s),x^2(s))\in \mathcal{H}(\delta_1,w+\delta_2)$ set
\be\label{eq:13}
I(\mathbf{x})=\int_0^1\frac{\gamma(\mathbf{x}'(s))}{c(x^1(s),x^2(s))}ds.
\end{equation}
Let $\e>0$ and assume that $\mathbf \phi=(\phi^1,\phi^2)\in\mathcal{H}(\delta_1,w+\delta_2)$ is a path such that $\Gamma(\delta_1,w+\delta_2)-I(\mathbf{\phi})<\e$.  It is possible to decompose $\mathbf \phi$ into disjoint segments $\phi_j$ so that $\phi=\sum_{j=1}^{2H} \phi_j$ and that 
\begin{enumerate}
\item For $j$ even, $\phi_j \subseteq R_{j/2}$, and therefore it is a linear segment with derivative $\phi_j'$ in $\R_+^{2}$ 
\item For $j$ odd,  $\phi_j \subseteq Q_{(j+1)/2}$.
\end{enumerate}
 The sum $\sum_{j=1}^{2H}\phi_j$ means path concatenation.
 
For $j$ odd, the total contribution of $\phi_j$ to $I(\phi)$ can be bounded by $\frac{1}{r_{\ell}}\gamma(\delta_1, \eta(\delta_1))$ where $\displaystyle r_{\ell} = \min_{(x, y) \in [0, \delta_1]\times [0, w + \delta_2]} c(x,y)$. Over all, the total contribution of the odd-indexed segments is bounded above by $4 H r_{\ell}^{-1}(\eta(\delta_1) \vee \delta_1)$. 

For $j$ even, the path segment is linear and the maximum contribution of any such segment is given by 
\begin{align*}
I(\phi_{j}) &= \frac{1}{r_{R_j}} \gamma(\delta_1, \text{height}(R_j)) =  \frac{1}{r_{R_j}}( \delta_1 +  \text{height}(R_j) + 2\sqrt{\delta_1 \text{height}(R_j)})\\
&\le  \frac{1}{r_{R_j}} \text{height}(R_j) + 2 C_j \sqrt{\delta_1}.
\end{align*}
Overall, on the even-indexed segments, the total contribution to $I(\phi)$ is bounded above by $\sum_{k=1}^H ( \frac{1}{r_{R_{2k}}} \text{height}(R_{2k}) + 2 C_{2k} )\sqrt{\delta_1} \le \sum_{k=1}^H  \frac{1}{r_{R_{2k}}} \text{height}(R_{2k}) + C \sqrt{\delta_1}$.

Then, 
\begin{align*}
\Gamma(\delta_1,w+\delta_2)- \e \le I(\phi) &\le \sum_{k=1}^H  \frac{1}{r_{R_{2k}}} \text{height}(R_{2k}) + C \sqrt{\delta_1}+ 4 H r_{\ell}^{-1}(\eta(\delta_1) \vee \delta_1)\\
&\le \Gamma(0, w+\delta_2)+ C \sqrt{\delta_1}+ 4 H r_{\ell}^{-1}(\eta(\delta_1) \vee \delta_1)\\
&\le \Gamma(0, w)+ \frac{1}{r_{\ell}}\delta_2+ C \sqrt{\delta_1}+ 4 H r_{\ell}^{-1}(\eta(\delta_1) \vee \delta_1)\\
&\le \Gamma(0, w)+ C \delta_2 \vee \sqrt{\delta_1}\vee \eta(\delta_1). 
\end{align*}
Let $\e \to 0$.
\end{proof}

\begin{corollary}\label{cor:1}
Fix $(x,y)\in \R^2_+$ and a speed function $c$. Then there exists $C=C(x,y,c(\cdot,\cdot))<\infty$ such that for any $\delta$  positive, there exist $\delta_1$,  $\delta_2$ sufficiently small
\begin{equation}\label{eq:16}
\Gamma(x+\delta_1,y+\delta_2)-\Gamma(x,y)<C\delta.
\end{equation}

\end{corollary}
\begin{proof}Let $B_{(x,y)}$ be a rectangle, where the north-east corner point is $(x,y)$ and south-west corner is $(0,0)$. \\
Let $\e>0$ and $\phi^\e$ a path such that $\Gamma(x+\delta_1,y+\delta_2)-I(\phi^\e)<\e$. Moreover, let $\bf u $ be the point where $\phi^\e$ first intersects the north or the east boundary of $B_{(x,y)}$. Without loss of generality assume is the east boundary and so ${\bf u}=(x,b)$ for some $b\in [0,y]$. Then,
\begin{align*}
\Gamma(x+\delta_1,y+\delta_2)-\e&\leq I(\phi^\e)\\
&\leq \Gamma(x,b)+\Gamma((x,b),(x+\delta_1,y+\delta_2))\\
&=\Gamma(x,b)+\Gamma((x,b),(x+\delta_1,y+\delta_2)) \pm \Gamma((x,b),(x,y))\\
&\leq \Gamma(x,y)+\Gamma((x,b),(x+\delta_1,y+\delta_2))-\Gamma((x,b),(x,y)).
\end{align*}
A rearrangement of terms gives  
\begin{align*}
\Gamma(x+\delta_1,y+\delta_2)-\Gamma(x,y)&\leq \Gamma((x,b),(x+\delta_1,y+\delta_2))-\Gamma((x,b),(x,y))+\e\\
&\leq C\delta+\e
\end{align*}
where we used \eqref{eq:11}, albeit with a starting point of $(x,b)$. Let $\e \to 0$ to prove the corollary.
\end{proof} 

We are now ready to prove  Theorem \ref{cor:conti}.

\begin{proof}[Proof of Theorem \ref{cor:conti}]

Fix an $\e>0$ and let $\zeta_1, \zeta_2$ small enough so that  by Corollary \ref{cor:1} we have 
\[
\Gamma((a,b), (x+\zeta_3, y+\zeta_4)) - \Gamma((a,b), (x, y)) < \e/4.
\]
Then, keep $\zeta_3, \zeta_4$ fixed and find a $\zeta_1, \zeta_2$ small enough so that again by Corollary \ref{cor:1}, 
\[
\Gamma((a-\zeta_1,b-\zeta_2), (x+\zeta_3, y+\zeta_4)) - \Gamma((a,b), (x+\zeta_3, y+\zeta_4)) < \e/4.  
\]
Together the inequalities above give
\be\label{eq:stronga1} 
\Gamma((a - \zeta_1, b - \zeta_2), (x + \zeta_3, y+ \zeta_4)) - \Gamma((a,b),(x,y)) < \e/2.
\ee
Similarly, one can approximate from the inside, and find $\zeta_5$, $\zeta_6$, $\zeta_7$, $\zeta_8$ so that  
\be\label{eq:stronga2} 
\Gamma((a,b),(x,y)) - \Gamma((a + \zeta_5, b + \zeta_6), (x - \zeta_7, y- \zeta_8))  < \e/2.
\ee
Let $\delta_0 = \min_{1\le i \le 8}\{ \zeta_i\}$.  Since $\Gamma(u,v)$ decreases in the first argument and increases in the second argument the inequalities \eqref{eq:stronga1} and \eqref{eq:stronga2}, together with our choice of $\delta_0$ give 
\[
\Gamma((a-\delta_0,b-\delta_0), (x+\delta_0, y+\delta_0)) - \Gamma((a+\delta_0,b+\delta_0), (x-\delta_0, y-\delta_0)) < \e.
\] 
and that for any $\tilde a \in[ a-\delta_0, a+ \delta_0]$, $\tilde b \in[ b-\delta_0, b+ \delta_0]$, $\tilde x \in[ x-\delta_0, x+ \delta_0]$, $\tilde y \in[ y-\delta_0, y+ \delta_0]$, we have 
\[
\Gamma((a+\delta_0,b+\delta_0), (x-\delta_0, y-\delta_0)) \le \Gamma((\tilde a, \tilde b),(\tilde x, \tilde y)) \le \Gamma((a-\delta_0,b-\delta_0), (x+\delta_0, y+\delta_0)).
\]
The last two inequalities combined give the result.
\end{proof}

The reason for this technical approximation is the statements in the next lemma, motivated by the following argument. In the simplest case we would like to approximate the limits of last passage times using the limiting $\Gamma_c$ in rectangles where $c(x,y)$ has one discontinuity line. Unfortunately, unless the discontinuity of the speed is a line of slope 1, we cannot say at this point that the limit is $\Gamma_c(x,y)$. However, if the speed function is continuous, the fact that the limit of passage times is $\Gamma_c$ in that environment is given by Theorem 3.1.\ in \cite{Geo-Kum-Sep-10}. So we may approximate $\Gamma_c$ with the value $\Gamma_{\tilde c}$ where $\tilde c$ will be a continuous speed function that approximates $c(s,t)$. 

\begin{lemma}[Continuity of $\Gamma$ in the speed function]
\label{lem:01:20}
Let $c(s,t)$ take only two values $r_1, r_2$ in two regions of $[a,x]\times[b,y]$ separated by a  weakly monotone curve $h$, which satisfies Assumption \ref{ass:c}. Then, for every $\e > 0$ there exists a 
$\eta_{h,\e}> 0$ so that for all $\eta < \eta_{h, \e}$ there exists a continuous speed function $c^{\text{cont}}_{\eta}(s,t) \le c(s,t)$ so that 
\[
\Gamma_{c^{\text{cont}}_{\eta}}((a,b)(x,y))-\Gamma_c((a,b),(x,y)) \le \e.
\]  
\end{lemma}

\begin{proof}[Proof of Lemma \ref{lem:01:20}]
Fix $(x,y)$ and without loss assume that the starting point is $(a, b) = (\alpha,0)$ for some $\alpha > 0$.
We present the case when the curve $h$ starts somewhere on $[\alpha,x]$ and exits somewhere on the east boundary
 $\{ x \}\times [0,y]$ and the rates above the curve is 
$r_1 < r_2$. Symmetric arguments as the one below will work in all other cases, and are left to the reader. 
\begin{figure}
\centering
\begin{tikzpicture}[>= latex, scale=0.7]

%horizontal axis
\draw[->] (0,0)-- (7,0) node[anchor=north] {\small$x$};
% vertical axis
\draw[->] (0,0) node[below]{\small0}-- (0,10) node[anchor=east] {\small$y$};
\draw (6,0)--(6,9)node[above]{$(x,y)$}--(0,9);
 \fill[color=nicos-red!30](2,0)[bend right=20]to(6,8)--(6,9)--(0,9)--(0,0)--(2,0);
 \fill[color=blue!30](3.5,0)[bend right=20]to(6,4)--(6,0)--(3.5,0);
 \fill[color=purple!60](3.5,0)[bend right=20]to(6,4)--(6,8)[bend left=20]to(2,0)--(3.5,0);
\draw (6,0)--(6,9)node[above]{$(x,y)$}--(0,9);
\draw[line width=1pt] (2,0)node[xshift=-0.5em,yshift=-1em]{\small$(\alpha,0)$} [bend right=20]to(6,8);
\draw[line width=1pt] (3.5,0)node[xshift=0.5em,yshift=-1em]{\small$(\alpha+\eta,0)$} [bend right=20]to(6,4);
\draw[dashed,line width=1pt] (1,0)--(1,9);
\draw[dashed,line width=1pt] (5,0)--(5,9);
\draw (4.2,4)node{\small$h$};
\draw (5.65,1.8)node{\small$ h_\eta$};
\draw [fill] (5,9) circle [radius=0.09]node[xshift=-2em,yshift=1em]{\small$(x-\eta,y)$};
\draw [fill] (1,0) circle [radius=0.09]node[xshift=0.4em,yshift=1em]{\small$(\alpha-\eta,0)$};
\end{tikzpicture}
\caption{Graphical representation for the proof of Lemma \ref{lem:01:20}.}
\label{fig:ex}
\end{figure}
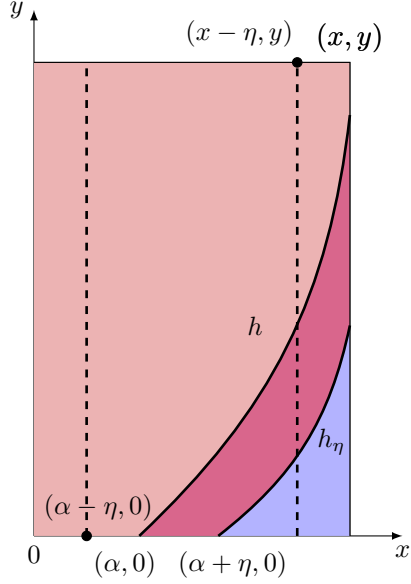

For a fixed $\e>0$ we can find an $\eta_{\e,h}>0$ so that for all $\eta < \eta_{\e,h}>0 $ we have $| \Gamma_c((\alpha-\eta,0), (x - \eta, y)) - \Gamma_c((\alpha,0),(x,y)) | < \e$. This is possible by Theorem \ref{cor:conti}.
Fix any such  $\eta$ and define the curve $h_{\eta}$ by the relation $h_{\eta}(t) = h(t+\eta)$, i.e.~ this correspond to shift of $h$ by $\eta$ to the right. 
Then, we define a speed function $c_{\eta}(\cdot, \cdot)$ on $[\alpha, x]\times[0, y]$
\[
c_{\eta}(z, w)= 
\begin{cases}
r_1, &\quad \text{ if  $(z, w)$ is above or on the graph of $h_{\eta}$, }\\
r_2, & \quad \text{otherwise}.
\end{cases}
\]
We make two observations:
\begin{enumerate}
\item $c(z,w) \ge c_{\eta}(z, w)$ for all $(z, w) \in [\alpha, x]\times[0, y]$, giving $\Gamma_{c_{\eta}}((\alpha,0),(x,y)) \ge \Gamma_{c}((\alpha,0),(x,y))$.
\item By construction
\be\label{eq:nonso}
 \Gamma_{c}((\alpha-\eta, 0), (x- \eta, y))  = \Gamma_{c_\eta}((\alpha, 0),(x,y)).
\ee
\end{enumerate}
From these observations we define a new, continuous  function $c^{\text{cont}}_\eta(\cdot, \cdot)$ on $[\alpha, x]\times[0, y]$ so that 
\[
 c_{\eta}(z, w) \le c^{\text{cont}}_\eta(z, w) \le c(z,w), \quad \text{ for all  $(z, w) \in [\alpha, x]\times[0, y]$}.
\]
This and \eqref{eq:nonso} imply
\be
 \Gamma_{c^{\text{cont}}_\eta}((\alpha,0),(x, y)) \le \Gamma_{c_{\eta}}((\alpha,0),(x, y))= \Gamma_{c}((\alpha-\eta, 0), (x- \eta, y))  \leq  \Gamma_{c}((\alpha,0),(x, y)) + \e,
\ee
which in turn yields the Lemma.
\end{proof}

\section{Proof of Theorem \ref{thm:1}}
\label{sec:5}

To prove Theorem \ref{thm:1} we need some Lemmas which help us to define some properties of the last passage time in a 2D inhomogeneous environment.

We begin by identifying the last passage time limits in simple cases of speed function, that will be used as building blocks for approximations to the general case.
We first find the law of large numbers without fixing the maximal path but forcing it to stay in a homogeneous corridor.
Let the speed function be 
\be\label{eq:strip}
c(x,y)=
\begin{cases}
&r_2\quad y>x+\lambda,\\
&r_1\quad x-\lambda\leq y\leq x+\lambda,\\
&r_3\quad y<x-\lambda.\\
\end{cases}
\ee
with $\lambda\in\R_+$.

\begin{lemma}[Passage times in homogeneous corridors] \label{lem:4} 
Assume $c(x,y)$ in \eqref{eq:strip} for all $(x,y) \in (0,b) \times (0,e)$. Let $(z,w) \in (0,b] \times (0,e]$ with $w\in(z-\lambda,z+\lambda)$ and let  $\tilde{G}_{(\fl{nz},\fl{nw})}$ be the last passage time from $(0,0)$ to $(\lfloor nz \rfloor,\lfloor nw \rfloor)$ subject to the constraint that

\begin{center}

 admissible paths stay in the $r_1$-rate region inside the strip $\fl{nb}-\lambda\leq \fl{ne}\leq \fl{nb}+\lambda$,

except possibly for a bounded number of initial and final steps.

\end{center}

Then 
\begin{equation}\label{eq:18a}
\lim_{n\to\infty}n^{-1}\tilde{G}_{(\fl{nz},\fl{nw})}=r_1^{-1}\gamma(z,w),\quad \P-\text{a.s.}
\end{equation}
\end{lemma}
\begin{proof}
 To obtain the upper bound $\lim_{n\to \infty}n^{-1}\tilde{G}_{(\fl{nz},\fl{nw})}\leq r_1^{-1}\gamma(z,w)$ ignore the path restrictions and assume that the environment in the whole region is homogeneous of constant rates $r_1$.

For the lower bound we use a coarse graining argument, taking into account the path restrictions.
Fix an $\e\in(0,1)$ and consider the points 
\[
\mathscr P_{z,w, \e} = \{ (k\fl{\e n z}, k\fl{\e n w}): k=1, 2, \ldots, \fl{\e^{-1}} \}\cup\{ \fl{nz}, \fl{nw}) \}.
\] 
To bound $\tilde G_{(\fl{nz},\fl{nw})}$ from below, force the path to go through the partition points of $\mathscr P_{z,w, \e} $. 
By possibly reducing $\e$ further, for each  $1 \le k \le \fl{\e^{-1}}$, each rectangle with lower-left and upper-right corners two consecutive points of $\mathscr P_{z,w, \e} $ is completely inside the region of rate $r_1$. For these rectangles we allow the path segments to explore space.

For $2\leq k< \fl{ \e^{-1}}$ let $G_{R^n_k}$ be the last passage time from $((k-1)\fl{\e nz} ,(k-1)\fl{\e nw})$ to $(k\fl{nz\e}, k\fl{nw\e})$. $R_k^n$ refers to the rectangle that contains all the admissible paths between the two points. 

Let $0\leq\delta=\delta(\e)<\e r^{-1}\gamma(z,w)$ and assume without loss that $\delta/\e \to 0$ as $\e \to 0$. A large deviation estimate (Theorem 4.1 in \cite{Sep-98-mprf-2}) gives a constant $C=C(r,z,w,\e,\delta)$ such that for $k$ fixed
\begin{equation}\label{eq:18}
\mathbb{P}\{G_{R^n_k}\leq n(\e r^{-1}\gamma(z,w)-\delta)\}\leq e^{-Cn^2}.
\end{equation}
The sequence of passage times $\{ G_{R^n_k}\}_k$ are i.i.d.~ and as such, a Cram\`er large deviation estimate and a Borel-Cantelli argument give for large $n$,
\[\tilde G_{(\fl{nz},\fl{nw})}\geq \sum_{k=1}^{\fl{\e^{-1}}-1}G_{R_k^n}\geq n(\fl{\e^{-1}}-1)(\e r^{-1}\gamma(z,w)-\delta), \quad \P \text{-a.s.}\]
Divide the inequality through by $n$ and take the $\liminf$ as $n\to\infty$. After that, send $\e\to0$ to finish the proof.
\end{proof}

From the coarse graining argument in the previous proof, we see that when we restrict to maximal paths in a narrow (but macroscopic) homogeneous corridor we still obtain the same limiting passage time as if the environment was homogeneous throughout. This is a consequence of the mesoscopic fluctuations of the maximal paths and the strict concavity of $\gamma$.  As the width $\e$ of the corridor tends to $0$, the limiting shape of the corridor is a straight line, which is the shape of the macroscopic maximal path in a homogeneous region.  

\begin{lemma}[Passage times in $C^1$ homogeneous corridors]\label{lem:c1cor} Let $\bf x(s)$ be a $C^1$ increasing path from $(a,b)$ to $(c,d)$, and let $\mathcal N({\bf x}, \e)$ be a neighborhood  subject to the constraint that $c( {\bf x(s)}) = r$ (constant) on $\mathcal N({\bf x}, \e)$. 
Let $G^{(n)}_{n \mathcal N({\bf x}, \e)}$ be the passage time from $\fl{n(a,b)}$ to $\fl{n(c,d)}$, subject to the constraint that maximal paths never exit  
$ n \mathcal N({\bf x}, \e)$. Then 
\[
\varliminf_{n\to \infty} n^{-1} G^{(n)}_{n \mathcal N({\bf x}, \e)} \ge  \frac{1}{r}\int_{0}^{1} \gamma({\bf x}'(s)) \,ds.
\]
\end{lemma}

\begin{proof} 
Consider a partition of the interval $[0,1]$ 
$ \mathcal P = \{ 0 =s_0 < s_1 < \ldots < s_N = 1\} $
fine enough so that the rectangles $R({\bf x}(s_i), {\bf x}(s_{i+1}))$ are completely inside the neighborhood $\mathcal N({\bf x}, \e)$. Then, 
\begin{align*}
\varliminf_{n\to \infty} n^{-1} G^{(n)}_{n \mathcal N({\bf x}, \e)} &\ge \varliminf_{n\to \infty} n^{-1}\sum_{i=0}^{N-1} G^{(n)}_{\fl{ n{\bf x}(s_i)}, \fl{n{\bf x}(s_{i+1})}} \ge\sum_{i=0}^{N-1}  \varliminf_{n\to \infty} n^{-1} G^{(n)}_{\fl{n{\bf x}(s_i)}, \fl{n{\bf x}(s_{i+1})}} \\
&\ge \frac{1}{r}\sum_{i=0}^{N-1} \gamma( {\bf x}(s_{i+1})-{\bf x}(s_i))= \frac{1}{r}\sum_{i=0}^{N-1} \gamma\Big( \frac{{\bf x}(s_{i+1})-{\bf x}(s_i)}{s_{i+1}-s_1}\Big)(s_{i+1} - s_i)\\
&=\frac{1}{r}\sum_{i=0}^{N-1} \gamma\big({\bf x}'(\xi_i)\big)(s_{i+1} - s_i),\text{ for some $\xi_i \in [s_i, s_{i+1}]$, by the mean value theorem.} \\
\end{align*} 
As the mesh of the partition tends to $0$, the last line converges to $\frac{1}{r}\int_{0}^{1} \gamma({\bf x}'(s)) \,ds$, as it is a Riemann sum. This gives the result.
\end{proof}

\begin{lemma}[Passage times in two-phase rectangles] \label{lem:2p} Consider a $C^1$ function $h:[0,a]\to[0,b]$ and a 
macroscopic rectangle $[0,a]\times[0,b]$ and in which the speed function is 
\[
c(x,y) = r_1 \mathbbm1_{\{y > h(x)\}} + r_2 \mathbbm 1_{\{y < h(x)\}}  + r_1 \wedge r_2 \mathbbm 1_{\{ y= h(x)\}}.
\]
We further assume that
\begin{enumerate}
\item $h([0,a]) = [0,b]$, $h$ is monotone and $h(x) \notin \{ 0, b \}$, for any $x \in (0,a)$.
\item  There exists $\eta>0$ so that $\min_{x \in (0,a)} | h '(x) | >\eta > 0.$
\item If $h$ is increasing, then we further assume that for the same $\eta>0$ as in (2), we have  $\sup_{x \in (0,a)}\Big |h'(x) - \frac{b}{a} \Big| < \eta. $ 
In particular, the first derivative is bounded and there exists a constant $L$ so that the curve is Lipschitz-$L$.
\end{enumerate}
Assume for convenience that $r_1 < r_2$. Then, there exists a uniform constant $C_{h}$ so that last passage time limits satisfy
\begin{enumerate}
\item For $h$ increasing ,
\be \label{eq:diagbound}
 \frac{1}{r_1}\gamma(a,b) - \frac{2}{r_1}C_h{ \rm length}(h) \eta  \le \varliminf_{n} n^{-1} G_{\fl{na}, \fl{nb}}^{(n)} \le \varlimsup_{n} n^{-1} G_{\fl{na}, \fl{nb}}^{(n)} \le  \frac{1}{r_1}\gamma(a,b).
\ee
Moreover, 
\be \label{eq:diagcontbound}
\frac{1}{r_1}\gamma(a,b) - \frac{2}{r_1}C_h{ \rm length}(h) \eta \le  \Gamma(a,b) < \frac{1}{r_1}\gamma(a,b), 
\ee
which in turn implies 
\be
\varlimsup_{n\to \infty} | n^{-1} G^{(n)}_{\fl{na}, \fl{nb}} -  \Gamma(a,b) | \le  \frac{2}{r_1}C_h{ \rm length}(h) \eta.
\ee
\item When $h$ is decreasing
\be
\lim_{n \to \infty}  n^{-1} G^{(n)}_{\fl{na}, \fl{nb}} = \Gamma(a,b).
\ee
\end{enumerate}
\end{lemma}

\begin{proof}
We first treat the case of increasing $h$. Without loss, assume $h(0)=0$ and $h(a)=b$.
Since $r_1 < r_2$ we obtain the upper bound in \eqref{eq:diagbound} if we lower $r_2$ to $r_1$ and assume a homogeneous environment with constant speed function $c_{\rm low}(x,y)=r_1$. This also gives the upper bound in \eqref{eq:diagcontbound} since $c_{\rm low}(x,y) \le c(x,y)$.

Now for the lower bound. 
Let $\e>0, \delta > 0$ sufficiently small. First consider a graph $h_{\e}(x) = (h(x) + \e)\wedge b$ which lies solely in the $r_1$ region of $c(x,y)$. 

By hypothesis $(1)$, assume $\e$ is small enough so that the first time $h_\e$ touches the top boundary $[0,a]\times \{b\}$, is precisely at some point $x_\e > a - \delta$. 
Consider a parametrisation for $h$,  $(h^{(1)}(s), h^{(2)}(s)): [0,1] \to \R^2$. Then point $x_\e$ corresponds to some $1 -s_\e \in [0,1]$. 

Then define the  curve $\bf x$ that goes from $(0,0)$ to $(0, h_\e(0))$ by time $s_\e$, then follows $h_{\e}$ until it takes the value $b$ by time $1$ and then stays on the north boundary at value $b$ for time $s_\e$.

Since $h$ is rectifiable, so is $h_\e$, and we assume without loss that $h_\e$ has the Lipschitz parametrization 
\[ \Big(h^{(1)}\Big((s-s_\e) \frac{1- s_\e}{1-2s_\e}\Big), h^{(2)}\Big((s-s_\e) \frac{1- s_\e}{1-2s_\e}\Big)+\e\Big), \quad s \in [s_\e, 1-s_\e].\] 
 Then we estimate
\begin{align}
\Gamma(a,b) &\ge  \int_{s_\e}^{1-s_\e}\frac{\gamma( {\bf x}'(s))}{r_1}\,ds 
=  \frac{1-s_\e}{1-2s_\e}  \int_0^{1-s_\e} \frac{\gamma( h^{(1)'}(s),  h^{(2)'}(s) )}{r_1}\,ds \notag \\
&=  \frac{1-s_\e}{1-2s_\e}  \int_0^{1-s_\e} h^{(1)'}(s)\frac{\gamma( 1 ,  \frac{h^{(2)'}(s) }{h^{(1)'}(s)})}{r_1}\,ds =  \frac{1-s_\e}{1-2s_\e}  \int_0^{1-s_\e} h^{(1)'}(s)\frac{\gamma( 1 , h'(h^{(1)}(s))}{r_1}\,ds \notag \\
& =  \frac{1-s_\e}{1-2s_\e}  \int_0^{h^{(1)}(1-s_\e)} \frac{\gamma( 1 , h'(u))}{r_1}\,du \notag  \ge  \frac{1-s_\e}{1-2s_\e}  \int_0^{h^{(1)}(1-s_\e)} \frac{\gamma( 1 , \frac{b}{a}-\eta)}{r_1}\,du \notag \\
&=  \frac{1-s_\e}{1-2s_\e} h^{(1)}(1-s_\e) \frac{\gamma(1 , \frac{b}{a}-\eta)}{r_1} \notag \\
&\ge a \frac{\gamma(1 , \frac{b}{a}-\eta)}{r_1} - \delta  \frac{1-s_\e}{1-2s_\e}  \frac{\gamma(1 , \frac{b}{a}-\eta)}{r_1} - \frac{s_\e}{1-2s_\e} \frac{\gamma(1 , \frac{b}{a}-\eta)}{r_1}. \label{eq:mod}
\end{align}
Letting $\e \to 0$ makes the last term vanish, 
and by then letting $\delta \to 0$ we obtain 
\be \label{eq:a16}
\Gamma(a,b) \ge \frac{\gamma(a , b -a\eta)}{r_1} = \frac{1}{r_1}\Big( a + b - a\eta + 2\sqrt{a}\sqrt{b}\sqrt{1 - \frac{a\eta}{b}}\Big).
\ee
By the mean value theorem $\eta < \min | h'(s) | < ba^{-1}$ and by item (2) in the hypothesis, one can check that  
\[
\sqrt{1 - \frac{a\eta}{b}} \ge 1 -  \frac{a\eta}{b}. 
\]
We now estimate the $\gamma$-term in the left hand side of \eqref{eq:a16}.
\begin{align}
\gamma(a , b -a\eta) 
&= a + b - a\eta + 2\sqrt{a}\sqrt{b}\Big(1 - \frac{a\eta}{b}\Big) = a + b - a\eta + 2\sqrt{a}\sqrt{b} - 2\frac{a^{3/2}\eta}{b^{1/2}}\\ 
&\ge \gamma(a,b) - 2\eta\Big( a + \frac{a^{3/2}}{b^{1/2}} \Big).\label{eq:lbbb}
\end{align}
Now the lower bound in \eqref{eq:diagcontbound}. Let 
\[ 
C^2_h > \frac{1 + 2\sqrt L}{L^3} \vee \Big(1 +\frac{1+2\sqrt{L}}{\min_{x \in (0,a)} h'(x)}\Big).
\]
Keep in mind that by the mean value theorem, $b/a \ge \min_{x \in (0,a)} h'(x)$ and by the choice of $C_h$ we have 
\[
\frac{b}{a} \ge \min_{x \in (0,a)} h'(x) \ge \frac{1+2\sqrt{L}}{C_h^2 - 1}. 
\] 
Then we can bound 
\begin{align*}  
0 &\le a^2((C_h^2 - 1)b -(1 + 2 \sqrt L)a) 
< (C_h^2 - 1)a^2b - (1 + 2 \sqrt L)a^3 + C^2_h b^3 \\
&= (C_h^2 - 1)a^2b - a^3   - 2 \sqrt L a^3 + C^2_h b^3< (C_h^2 - 1)a^2b - a^3   - 2 a^{5/2}b^{1/2} + C^2_h b^3.
\end{align*}
In the last inequality above we used (3), since it implies $ h(a) - h(0) = b \le L a$.
An equivalent way to write the last inequality is 
\be\label{eq:a17}
\Big(a + \frac{a^{3/2}}{b^{1/2}}\Big)^2 < C_h^2(a^2+b^2).
\ee
From \eqref{eq:a17}, we conclude that $a + \frac{a^{3/2}}{b^{1/2}} < C_h\sqrt{a^2 + b^2} \le C_h {\rm length}(h)$. Substitute this  in \eqref{eq:lbbb}
to finally prove the lower bound in \eqref{eq:diagcontbound}.

For the lower bound in \eqref{eq:diagbound} consider again the function $h_{\e}$ and $s_\e$ from before and consider a partition of $[0, 1 -s_\e]$, $ \mathscr P_{s_\e, \delta} = \{ x_k =  k \delta (1- s_\e) \}_{ 0\le k \le \fl{\delta^{-1}}},$ 
of mesh $\delta > 0$. We assume the partition is fine enough
so that the rectangles $R_k = [x_k, x_{k+1}]\times[h_\e(x_k), h_\e(x_{k+1})]$ completely lie in the homogeneous region of rate $r_1$ and so that Riemann sum 
\be\label{eq:R-s}
\sum_{k=0}^{\fl{\delta}^{-1} -1} r_1^{-1} \gamma( h^{(1)'}(x_{k+1}),  h^{(2)'}(x_{k+1}))(x_{k+1} - x_k)
\ge \int_{0}^{1 - s_\e}\frac{\gamma( h^{(1)'}(s),  h^{(2)'}(s) )}{r_1}\,ds - \theta_1
\ee
for some fixed tolerance $\theta_1 >0$. Moreover, assume the partition is fine enough so that for $\eta_1$ sufficiently small, with $0 < \eta_1 < \alpha$
\[
\Big|\frac{h^{(i)}(x_{k+1}) - h^{(i)}(x_k)}{x_{k+1} - x_k} -  h^{(i)'}(x_{k+1})\Big| < \eta_1, \quad \text{for } i=1, 2.
\]
Finally, fix a small $\theta_2>0$ and let $n$ large enough so that Theorem 4.1 in \cite{Sep-98-mprf-2} gives
\[
\P\{ G_{nR_k} < n r_1^{-1} \gamma(h^{(1)}(x_{k+1}) - h^{(1)}(x_k), h^{(2)}_\e(x_{k+1})-h^{(2)}_\e(x_{k})) - n\theta_2 \} \le e^{- c n}. 
\]
By the Borel-Cantelli lemma we can then let $n$ be large enough so that $\P$-a.s.\ for all $k$ 
 \[
 G_{nR_k} > n r_1^{-1} \gamma(h^{(1)}(x_{k+1}) - h^{(1)}(x_k), h^{(2)}_\e(x_{k+1})-h^{(2)}_\e(x_{k})) - n\theta_2. 
 \] 
 Above we denoted by  $G_{nR_k}$ the maximum weight that can be collected from oriented paths in the set $nR_k$.
 
By superadditivity, the passage times satisfy 
\begin{align*}
G_{\fl{na}, \fl{nb}}^{(n)} &\ge \sum_{k=0}^{\fl{\delta}^{-1} -1} G_{nR_k}  \ge n \sum_{k=0}^{\fl{\delta}^{-1} -1} r_1^{-1} \gamma(h^{(1)}(x_{k+1}) - h^{(1)}(x_k), h^{(2)}_\e(x_{k+1})-h^{(2)}_\e(x_{k})) - n\theta_2 \delta^{-1}\\
& = n \sum_{k=0}^{\fl{\delta}^{-1} -1} r_1^{-1} \gamma\Big(\frac{h^{(1)}(x_{k+1}) - h^{(1)}(x_k)}{x_{k+1} - x_k},\frac{ h^{(2)}_\e(x_{k+1})-h^{(2)}_\e(x_{k})}{x_{k+1} - x_k}\Big)(x_{k+1} - x_k) - n\theta_2 \delta^{-1}\\
&\ge n \sum_{k=0}^{\fl{\delta}^{-1} -1} r_1^{-1} \gamma( h^{(1)'}(x_{k+1})-\eta_1,  h^{(2)'}(x_{k+1})-\eta_1)(x_{k+1} - x_k) - n\theta_2 \delta^{-1}\\
&\ge n \sum_{k=0}^{\fl{\delta}^{-1} -1} r_1^{-1} \gamma( h^{(1)'}(x_{k+1}),  h^{(2)'}(x_{k+1}))(x_{k+1} - x_k) - \frac{n}{r_1}\om_{\gamma}(\eta_1)- n\theta_2 \delta^{-1},\\
&\ge n \int_{0}^{1-s_\e}\frac{\gamma(h'(s))}{r_1}\,ds - \frac{n}{r_1}\om_{\gamma}(\eta_1)- n\theta_1 - n\theta_2 \delta^{-1}, \text{ by \eqref{eq:R-s}}.
\end{align*}
Divide through by $n$ and take the $\varliminf$ on both sides. First let $\theta_1, \theta_2 \to 0$. After that take $\eta_1 \to 0$. The final estimate comes from a repetition of computation \eqref{eq:mod} and bounds \eqref{eq:lbbb}, \eqref{eq:a17}.

When $h$ is decreasing, the approximation argument is simpler. We briefly highlight it but leave the details to the reader. First of all,
any monotone curve from $[0,a]$ to $[0,b]$ will have to cross $h$ at a unique point $(\zeta, h(\zeta))$. Then from Jensen's inequality, the piecewise linear curve from $0$ to  $(\zeta, h(\zeta))$ and then to $(a,b)$ achieves a higher value for the functional \eqref{eq:I}. So, candidate macroscopic optimisers can be restricted to piecewise linear curves, and this gives the lower bound 
\[
\Gamma(a,b) \le \varliminf_{n \to \infty}  n^{-1} G_{\fl{na}, \fl{nb}}^{(n)} 
\] 
by a coarse graining argument as for the case when $h$ was increasing. For the upper bound, partition the curve $h$ finely enough with a mesh $\delta>0$. Any microscopic optimal path will have to cross the microscopic curve $[nh]$ at some point $(\fl{n\zeta}, \fl{n(h(\zeta))})$, lying between two of the partition points. For $n$ large enough, the passage time on this path will $\P$-a.s , be no more than $nr_1^{-1}\gamma(\zeta, h(\zeta)) + nr_2^{-1}\gamma(a-\zeta, b -h(\zeta)) +n\e + Cn\sqrt{\delta}$ for any fixed $\e$. Divide by $n$, take the quantifiers to $0$ and then take supremum over all crossing points to obtain the upper bound.
\end{proof}

\begin{example}\label{ex:x2}
Consider a square with south-west corner  $(0,0)$ and north-east corner $(1,1)$. This square is subdivided in two constant-rate regions by a parabola $h(x)=x^2$ where above the rate is $1$ and below is $r\in(0,1)$. 
Then the set of the all potential optimisers is a concatenation of straight lines in the $1$ region and convex segments along the discontinuity $h(x)$.
\end{example}
\begin{figure}[!h]
\centering
\begin{tikzpicture}[>= latex, scale=0.8]

%horizontal axis
\draw[->] (0,0)-- (10,0) node[anchor=north] {\small$x$};
% vertical axis
\draw[->] (0,0) node[below]{\small0}-- (0,7) node[anchor=east] {\small$y$};
\draw (7.5,7)node [right,black] {$(1,1)$};
\draw[line width=2pt, blue]  (0,0)to [bend right=15] (8.5,6.5);
\draw[line width=2pt, my-green]  (2,0)to [bend right=15] (9.5,6.5);
\draw[line width=2pt, my-red]  (0,0)-- (8.5,6.5);
\draw[line width=2pt, purple]  (4.16,0.95)node[xshift=0.3em,yshift=-0.7em,black]{\tiny$(x_1,y_1)$}-- (7.1,3.25) node[xshift=1.5em,black]{\tiny$(x_2,y_2)$};
\draw (4,2.5) node{\small$h(x)$};
\draw (8.3,4) node{\small$\tilde h(x)$};
\draw (3.9,3.6) node{\small$1$};
\draw (4.5,1.8) node{\small$r$};
\draw (6.9,1.8) node{\small$\ell$};
\draw (6.2,2.2) node{\small$\tilde\ell$};
\draw (2.9,3.3) node{\small$\ell$};
\draw (2.3,0.7) node{\small$\delta$};
\draw[<->,line width=1pt](2,0)--(2,0.9);
\draw[line width=2pt, my-red] (0,0)to[bend right=10] (3.8,0.95);
\draw[line width=2pt, my-red,decorate,decoration={snake,amplitude=1mm,segment length=5mm,post length=5mm,pre length=5mm}] (3.8,0.95)to[bend right=40] (7.2,3.6);
\draw[line width=2pt, my-red] (7.2,3.6)to[bend right=10] (8.5,6.5);
\draw [fill] (5.63,2.1) circle [radius=0.09]node[xshift=-1em,yshift=0.5em]{\tiny$(x_t,y_t)$};
\end{tikzpicture}
\caption{Graphical representation for the Example \ref{ex:x2}.}
\label{fig:ex}
\end{figure}
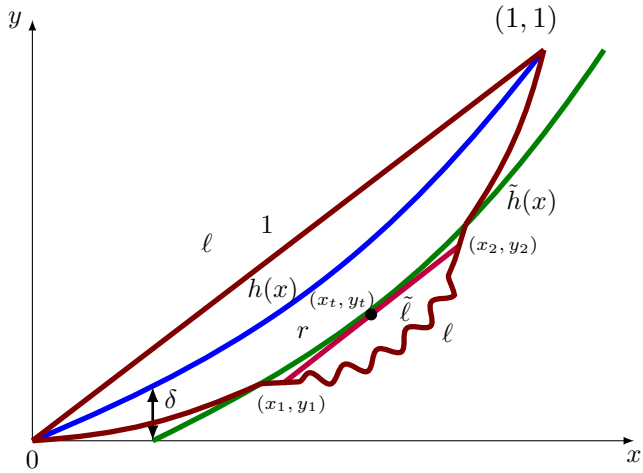

From Jensen's inequality and the convexity of $h(x)$ it is immediate to see that any segment of an optimiser in the rate $1$ region will have to be a straight line from the entry point to the exit point of the optimiser in the region. Therefore it remains to prove the shape of the maximal path in the $r$ region.

We first claim that for any potential optimiser $\ell \in \mathcal H(1,1)$, 
{
there exists a neighborhood $\mathcal N_\ell$ on $[0,1]$ such that for every $x\in \mathcal N_\ell$ a potential optimiser in $\mathcal H(1,1)$ takes the value $h(x)$ for $x\in\mathcal N_\ell$.}

To see this we use a proof by contradiction: First, we show that for $r$ small enough, any potential optimiser has to enter the $r$-region. If that was not the case, Jensen's inequality would give that the straight line from $(0,0)$ to $(1,1)$ is actually an optimiser and the last passage time constant would be  
\[ I_{\ell}(1,1)=\int_0^1(\sqrt{1}+\sqrt{1})^2dt=4.\]
However, the $C^1$ curve $h(x)$ is also an admissible curve, and it achieves potential 
\[
I_{h(x)}(1,1)=\frac{1}{r}\int_0^1(1+\sqrt{2t})^2dt=\frac{2}{r}(1+\frac{2}{3}\sqrt{2}),
\]
by the lower semicontinuity assumption on $c(x,y)$.
Therefore, for $r<\frac{1}{2}+\frac{\sqrt{2}}{3}$, we have  $I_{\ell}(1,1)<I_{h(x)}(1,1)$, so the optimiser $\ell$ has to enter the slow region.

Now suppose that $r < \frac{1}{2}+\frac{\sqrt{2}}{3}$ in order to complete the example. We can find  points $(a, h(a))$ and $(b, h(b))$ so that $\ell$ enters in the $r$ region through the point $(a,h(a))$ with $a\in[0,1)$ and stays in there without touching $h(x)$ except until $(b, h(b))$. We allow that potentially $(1,1)= (b, h(b))$. Since $\ell$ is continuous, it is possible to find a $\delta>0$ so that for  $t$ in some open interval $\mathcal N_\ell$ we have
\be\label{conddel}
|h(t)-\ell(t)|>\delta.
\ee 
To see that \eqref{conddel} is not respected by a potential optimiser, consider a $\delta$ shift $\tilde h = (h - \delta/2)^+$. Since $\ell$ is continuous it will cross $\tilde h$ at least in two points $(a_1, \tilde h(a_1))$ and  $(b_1, \tilde h(b_1))$ and without loss assume $[a_1, b_1] \subseteq \mathcal N_\ell$. Pick any $t \in (a_1, b_1)$ and consider the tangent line at $(t, \tilde h (t))$ on $\tilde h$. By construction, this should cross $\ell$  in $(x_1,y_1)$ and $(x_2,y_2)$  (see Figure \ref{fig:ex}). By Jensen's inequality we know that the path $\tilde \ell$ which goes through $\ell$ up to point $(x_1,y_1)$, straight to $(x_2,y_2)$ and then follows $\ell$. Then, $I( \tilde \ell) > I(\ell)$ and therefore, $\ell$ cannot be an optimiser. This gives the desired contradiction.

The contradiction was reached by assuming that a potential optimiser enters the slow region, but without following the discontinuity curve $h$. This completes the example.\qed

\begin{remark} \label{rem:shape} In the above example, we only used the explicit form of the discontinuity $h$ just to argue that a potential optimiser will eventually enter the slow region. If this information is known, the latter part of the proof is completely general and it uses local convexity properties of the discontinuity. In particular it just uses the fact that the discontinuity curve and the potential optimiser are continuous, piecewise $C^1$ and there exists a point $(t, h(t))$ for which the tangent line does not enter the fast region. \qed
\end{remark}

\begin{remark} The previous example suggests that potential optimisers cannot be more regular than the discontinuity curves. \qed
\end{remark}

\begin{lemma}[Exponential concentration of passage times with continuous speed]
\label{lem:newLDP}
Let $c(s,t)$ be a continuous speed function in $[0,x]\times[0,y]$. Then, for any $\theta > 0$, there exists constants $A$ and $\kappa_{\theta, c}$ 
\be \label{eq:newldp}
\P\{ G_{\fl{nx}, \fl{ny}}^{(n)} \ge n \Gamma_c(x,y) + n \theta \} \le A e^{-\kappa_{\theta, c} n}.
\ee
\end{lemma}

\begin{proof}[Proof of Lemma \ref{lem:newLDP}]
Fix a tolerance $\e$ small. Its size will be determined in the proof. For a $K \in \N$, consider the two partitions  
\[
\mathscr P^{(K)}_{x}=\{ \alpha_{\ell} = \ell x K^{-1} \}_{0 \le \ell \le K}, \text{ and } \mathscr P^{(K)}_{y}=\{ \beta_{\ell} = \ell y K^{-1} \}_{0 \le \ell \le K}
\]
of $[0,x]$ and $[0,y]$ respectively. Let $R_{i,j}$ denote the open rectangle with south-west corner $(\alpha_i, \beta_j )$. Let 
\[
r_{i,j} = \inf_{(s,t) \in R_{i,j}} c(s,t).
\]
Define a speed function 
\[
c_{\text{low}}(s,t) = \sum_{(i, j)} r_{i,j} \mathbbm1_{\{ (s,t )\in R_{i,j}\}} + \sum_{(i,j)} r_{i-1,j} \wedge r_{i,j} \mathbbm1_{\{s = \alpha_i, \beta_j < t < \beta_{j+1 }\}} + \sum_{(i,j)}  r_{i,j-1}\wedge r_{i,j} \mathbbm1_{\{\alpha_i <s < \alpha_{i+1}, t = \beta_j\}}.
\]
The value of $c(\alpha_i, \beta_j)$ is the minimum of the values in a neighborhood around it.

We are assuming the initial condition that $r_{i, -1} = r_{-1,j} =\infty$. In words, $c_{\text{low}}(s,t)$ is a step function with the minimum value of the neighbouring rates on the boundaries of $R_{i,j}$. Note that 
$c_{\text{low}}(s,t) \le c(s,t)$. Let $\overline R_{i,j}$ denote the rectangle together with any of its boundaries for which it contributed the rate, using some rules to break ties, if the boundary value agrees for two rectangles.

At this point we assume that $K=K(\e)$ is large enough so that $\| c - c_{\text{low}}\|_{\infty} < \e$. This implies that 
\[
\Gamma_{c_{\text{low}}}(x,y) - \Gamma_{c}(x,y) \le \e \gamma(x,y) r_{\min}^{-2},
\]
where $r_{\min}$ is the smallest value of $c(x,y)$.
This is because for any path $\bf x$, 
\[\begin{aligned}
\int_0^1&\bigg\{\frac{\gamma(\mathbf{x}'(s))}{c_{\text{low}}(x_1(s),x_2(s))}-\frac{\gamma(\mathbf{x}'(s))}{c(x_1(s),x_2(s))}\bigg\}ds\\
&=\int_0^1\frac{\gamma(\mathbf{x}'(s))(c(x_1(s),x_2(s))-c_{\text{low}}(x_1(s),x_2(s)))}{c(x_1(s),x_2(s))c_{\text{low}}(x_1(s),x_2(s))}ds\leq\epsilon\int_0^1\frac{\gamma(\mathbf{x}'(s))}{c^2(x_1(s),x_2(s))}ds\\
&\leq\epsilon r_{\min}^{-2}\gamma(x,y),
\end{aligned}
\]
and the bound extends to the supremum over paths $\bf x$.

Pick a $L > 0$ so that $L^{-1} << K^{-1}$ and further partition each axis segment 
\[
\mathscr H^{(L)}_{i} = \{\alpha_i + \ell (\alpha_{i+1} - \alpha_i) L^{-1} \}_{0 \le \ell \le L}, \text{ and } \mathscr V^{(L)}_{j}=\{ \beta_i + \ell (\beta_{i+1} - \beta_i) L^{-1} \}_{0 \le \ell \le L}.
\] 
Define 
\[
\mathscr D_{i,j} =\{{\bf d}^{\ell}_{i,j} = (\alpha_i + \ell (\alpha_{i+1} - \alpha_i) L^{-1}, \beta_j )\}, \qquad \mathscr E_{i,j} = \{ {\bf e}^{\ell}_{i,j} = (\alpha_i, \beta_i + \ell (\beta_{i+1} - \beta_i) L^{-1})\}.
\]  
These completely partition all boundaries of the rectangles. 

We are now ready to prove the concentration estimate. Let $G^{\text{low}}_{\fl{nx}, \fl{ny}}$ denote the last passage time in environment determined by $c_{\text{low}}$. Let $\pi_{\max}$ be the maximal path, and let $\pi_k$ be the segment of the path in the $k$-th rectangle it visits  $n \overline R_{i_k, j_k}$.   

Now, for each $k$, $\pi_k$ will enter and exit $n \overline R_{i_k, j_k}$ between two consecutive points of $n\mathscr D_{i_k,j_k}, n\mathscr E_{i_k,j_k}$. We denote by  $n{\bf z_1}_{i_k,j_k}, n{\bf z_2}_{i_k,j_k} $ the consecutive points for the entrance and by $n{\bf z_1}_{i_{k+1},j_{k+1}}, n{\bf z_2}_{i_{k+1},j_{k+1}} $ for the exit. 

Let ${\bf x}$ be a continuous, piecewise linear path from $(0,0)$ to $(x,y)$ so that it crosses through the boundary segments  $[ {\bf z_1}_{i_k,j_k}, n{\bf z_2}_{i_k,j_k}]$ at some point ${\bf x }_k$. Then for $L$ small enough, we have that for some predetermined $\delta$
that 
\[ \Big| \frac{ \gamma( {\bf z_2}_{i_{k+1},j_{k+2}} - {\bf z_1}_{i_k,j_k})  }{r_{i_k, j_k}} - \frac{\gamma( {\bf x}_{k+1} - {\bf x }_k)}{r_{i_k, j_k}} \Big| < \delta. \]
 
We estimate
\begin{align*}
\P\{ &G^{(n)}_{\fl{nx}, \fl{ny}} \ge n \Gamma_c(x,y) + n \theta \} \le \P\{ G^{\text{low}}_{\fl{nx}, \fl{ny}} \ge n \Gamma_c(x,y) + n \theta \}\\
&\le \P\Big\{ \sum_k G^{\text{low}}_{\pi_k}  \ge n \Gamma_{c_{\text{low}}}(x,y) + n (\theta - \e \gamma(x,y) r_{\min}^{-2}) \Big\}\\
&\le \P\bigg\{ \sum_k G^{\text{low}}_{\fl{n{\bf z_1}_{i_k,j_k}}, \fl{n{\bf z_2}_{i_{k+1},j_{k+1}}}}   \ge n \Gamma_{c_{\text{low}}}(x,y) + n (\theta - \e \gamma(x,y) r_{\min}^{-2}) \bigg\}\\ 
&\le \P\bigg\{ \sum_k G^{\text{low}}_{\fl{n{\bf z_1}_{i_k,j_k}}, \fl{n{\bf z_2}_{i_{k+1},j_{k+1}}}}   \ge n \sum_{k} \frac{\gamma( {\bf x}_{k+1} - {\bf x }_k)}{r_{i_k, j_k}}+ n (\theta - \e \gamma(x,y) r_{\min}^{-2}) \bigg\}\\ 
&\le \P\bigg\{ \sum_k G^{\text{low}}_{\fl{n{\bf z_1}_{i_k,j_k}}, \fl{n{\bf z_2}_{i_{k+1},j_{k+1}}}}   \ge n \sum_{k} \frac{\gamma( {\bf z_2}_{i_{k+1},j_{k+2}} - {\bf z_1}_{i_k,j_k}) }{r_{i_k, j_k}}+ n (\theta - \e \gamma(x,y) r_{\min}^{-2} - K^2\delta) \bigg\}\\ 
&\le \sum_{k} \P\bigg\{G^{\text{low}}_{\fl{n{\bf z_1}_{i_k,j_k}}, \fl{n{\bf z_2}_{i_{k+1},j_{k+1}}}}   \ge n  \frac{\gamma( {\bf z_2}_{i_{k+1},j_{k+2}} - {\bf z_1}_{i_k,j_k}) }{r_{i_k, j_k}}+ n K^{-2}(\theta - \e \gamma(x,y) r_{\min}^{-2} - K^2\delta) \bigg\}\\ 
&\le A e^{-\kappa_{\theta, \e}n}, \quad \text{by Theorem 4.2 in \cite{Sep-98-mprf-2}}.
\end{align*}
The last inequality is only true if $\theta - \e \gamma(x,y) r_{\min}^{-2} - K^2\delta > 0$ which can be achieved when $\e$ is small enough so that $\e \gamma(x,y) r_{\min}^{-2} < \theta/3$ and then we reduce $\delta$ so that $ K^2\delta = K^{2}(\e)\delta < \theta/3$. Theorem 4.2 in \cite{Sep-98-mprf-2} is a large deviation principle which gives an exponential concentration inequality for passage times in a homogeneous environment.
\end{proof}

The final approximation before the proof of the main theorem is the limiting time constant in any piecewise constant environment. 

\begin{proposition}\label{prop:1}
Let $c(s,t)$ be a piecewise constant speed function satisfying assumption \ref{ass:c2}, with a set of discontinuity curves $\{ h_i \}_i$ satisfying Assumption \ref{ass:c}. 
Let ${\bf u} = (x, y) \in \R_2^+$. Then the following law of large numbers holds
\be
\lim_{n\to \infty} \frac{1}{n} G^{(n)}_{ \fl{n\bf u}} = \Gamma_c(\bf u), \quad \P-\text{a.s.}
\ee
\end{proposition}

\begin{proof}[Proof of Proposition \ref{prop:1}]

Fix ${\bf u} = (x,y) \in \R^2_+$ and consider any admissible path
$\mathbf{x} \in\mathcal{H}(x,y)$, viewed as a curve $s \in [0,1] \mapsto \mathbf{x}(s)=(x_1(s), x_2(s))$. Recall the definition of $I( \mathbf x)$ from \eqref{eq:I} and 
 remember that $\Gamma = \sup_{\mathbf{x}\in\mathcal{H}(x,y)}I(\mathbf{x})$.

Before  proceeding with the technicalities, we highlight the intuition and main approximation idea. 
The most used technique in literature to prove this kind of limit is to find an upper and lower bound for the microscopic last passage time and then show that they tend to the same macroscopic last passage time in the limit $n\to \infty$.
For the lower bound we use the superadditivity property of the microscopic last passage time, and any path acts as a lower bound. For the upper bound we have to construct a particular path which will  represent an upper bound for the microscopic last passage time, while approximating the macroscopic limit after scaling its weight by $n$. For this, we first partition the rectangle $R_{0, (x,y)} =[0,x]\times[0,y]$ in a very specific way so the following conditions are all satisfied. 

\begin{enumerate}
\item Isolate the finitely many points of intersection of the discontinuity curves in squares of size $\delta$, where $\delta$ will be sufficiently small. 

\item Isolate the finitely many points on strictly increasing $h_i$ for which $h_i'(s) = 0$ or $h'_i(s)$ is not defined, in squares of size $\delta$.
\end{enumerate}

Call the collection of these squares by $\mathcal I_{\delta} = \{ I_i \}_{1 \le i \le Q}$. This include points of intersections with the boundary of $R_{0, (x,y)}$. It is fine if these squares overlap, as long as all these problematic points are in their interior.

Away from $\mathcal I_{\delta}$, the discontinuity  curves are isolated so that for all curves we can partition each curve $h_i$ finely enough so that for a given tolerance $\eta$, 
\begin{enumerate}
\item Rectangles $R_{h_i(x_j), h_i(x_{j+1})}$ only contain the discontinuity curve $h_i$. Each rectangle now satisfies Assumption (1) of Lemma \ref{lem:2p}.
\item Assumption (3) in Lemma \ref{lem:2p} holds for any rectangle $R_{h_i(x_j), h_i(x_{j+1})}$. Assumption (2) of Lemma \ref{lem:2p}  is automatically satisfied away from $\mathcal I_{\delta}$. 
\end{enumerate}

Call the collection of these rectangles that cover curve $h_i$ by $\mathcal J_{h_i, \eta} = \{ R_{i,j} = R_{h_i(x_j), h_i(x_{j+1})} \}_j$.

\textbf{Lower Bound:}
Any macroscopic path $\mathbf{x}$ can be viewed as the concatenation of a finite number of segments $\mathbf{x}_j$ so that each segment belongs either in a constant rate region, or in one of the rectangles $\mathcal I_{\delta}$ or in one of the rectangles $\cup_{i} \mathcal J_{h_i, \e}$. Write 
\[
{\bf x}(s) = \sum_{k=1}^Q  {\bf x}(s) \mathbbm 1\{{\bf x}(s) \in I_k \} +  \sum_{k, \ell}  {\bf x}(s) \mathbbm 1\{{\bf x}(s) \in R_{k, \ell} \} + \sum_{k =1}^{D} {\bf x}(s) \mathbbm 1\{{\bf x}(s) \in D_{k} \}. 
\]
Refine the partition further, so that if $\mathbf x : [ s_i, s_{i+1}] \to \R^2 \in D_k$, then the open rectangle $R_{{\mathbf x}( s_i), {\mathbf x} (s_{i+1})} \subseteq D_k$.

Let $(x_1(s), x_2(s))$ a parametrization of the path $\bf x$. Partition the interval $[0,1]$ into  
$\mathcal P = \{ 0 =s_0 < s_1 < s_2 < \ldots< s_K =1 \}$ so that the path segment 
$\mathbf x : [ s_i, s_{i+1}] \to \R^2$ belongs to exactly one $I_k$, $R_{k, \ell}$, or $D_k$. Note that  $I(\mathbf x) = \sum_{i = 0 }^{K-1} \int_{s_i}^{s_{i+1}}\frac{\gamma(\mathbf x'(s))}{c({\mathbf x}(s))}\,ds$. The constant $K = K_{\delta, \eta}$ is the total number of different regions the path touches.

We bound each contribution separately: 
\begin{enumerate}[(1)]
	\item $\mathbf x : [ s_i, s_{i+1}] \to \R^2 \in I_k$. Then, at most, 
	\[
	 \int_{s_i}^{s_{i+1}}\frac{\gamma(\mathbf x'(s))}{c({\mathbf x}(s))} < C \delta.
	\]
	Then for all $n$ large enough
\[
\Big| \frac{G_{[\fl{n {\bf x}(s_i)}, \fl{n{\bf x}(s_{i+1})}]}}{n} -  \int_{s_i}^{s_{i+1}}\frac{\gamma(\mathbf x'(s))}{c({\mathbf x}(s))}\,ds \Big| < C\delta,
\]
since also passage times in these rectangles are bounded by $Cn\delta$.

\item $\mathbf x : [ s_i, s_{i+1}] \to \R^2 \in D_k$, where $D_k$ is the homogeneous region of rate $r_k$.  Fix a small $\theta_1>0$. Then for all $n$ large enough, by the concentration estimates in \cite{Sep-98-mprf-2}
\begin{align*}
\frac{G_{[\fl{n {\bf x}(s_k)}, \fl{n{\bf x}(s_{k+1})}]}}{n} > \frac{ \gamma\big( {\bf x}(s_{k+1}) - {\bf x}(s_k)\big)}{r_k}  - \theta_1
>  \int_{s_i}^{s_{i+1}}\frac{\gamma(\mathbf x'(s))}{c({\mathbf x}(s))}\,ds - \theta_1.
\end{align*}

\item $\mathbf x : [ s_i, s_{i+1}] \to \R^2 \in R_{k, \ell}$.  Define 
\[
s_- = \inf \{s \in [ s_i, s_{i+1}] : {\mathbf x}(s) - h_k =0   \}, \quad s_+ = \sup \{s \in [ s_i, s_{i+1}] : {\mathbf x}(s) - h_k =0 \}.
\]
In words, $\mathbf x(s_-)$ and $\mathbf x(s_+)$ are the points of first and last intersection of $\mathbf x$  with $h_k$ in the rectangle $R_{k, \ell}$. Before $\mathbf x(s_-)$ and after 
$\mathbf x(s_+)$, $\bf x$ stays in a constant-rate region, in this rectangle. Between $\mathbf x(s_-)$ and $\mathbf x(s_+)$, $\bf x$ touches the discontinuity curve. This rectangle has two constant-rate regions. Denote the smallest one of those by $r_{\rm low}$.

We bound in the case where the discontinuity curve in the rectangle is increasing. If it is decreasing, $s_- = s_+$, and the argument simplifies since the path $\bf x$ only intersects the discontinuity at a single point.

Let $ G^{(n), { \mathcal N({\bf x}, \e) }}_{\fl{n {\bf x}(s)}, \fl{n{\bf x}(t)}}$ denote the passage time from $\fl{n {\bf x}(s)}$ to $\fl{n{\bf x}(t)}$, subject to the constraint that paths stay in the strip  $n\mathcal N({\bf x}, \e)$. We assume $\e$ is small enough so that the speed function stays constant on  $n\mathcal N({\bf x}, \e) \cap R(\fl{n {\bf x}(s)},\fl{n{\bf x}(t)})$ except possibly at an $O(\e)$ region near the beginning and end points of the rectangle.

\begin{align} 
\frac{G^{(n)}_{\fl{n {\bf x}(s_i)}, \fl{n{\bf x}(s_{i+1})}}}{n} &\ge \frac{G^{(n), { \mathcal N({\bf x}, \e) }}_{\fl{n {\bf x}(s_i)}, \fl{n{\bf x}(s_{-})}}}{n}+\frac{G^{(n)}_{\fl{n {\bf x}(s_-)}, \fl{n{\bf x}(s_{+})}}}{n}+\frac{G^{(n), \mathcal N({\bf x}, \e) }_{\fl{n {\bf x}(s_+)}, \fl{n{\bf x}(s_{i+1})}}}{n} \notag \\
&\ge \int_{s_i}^{s_{-}}\frac{\gamma(\mathbf x'(s))}{c({\mathbf x}(s))}\,ds  - \theta_1 + \frac{ \gamma\big( {\bf x}(s_{-}) - {\bf x}(s_+)\big)}{r_{\rm low}} - C_{k,\ell} {\rm length}(h_{k} \cap R_{k, \ell}) \eta \notag \\
&\phantom{xxxxxxxxxxxxxxxxxxxxxxxxxxxxxx} +  \int_{s_+}^{s_{i+1}}\frac{\gamma(\mathbf x'(s))}{c({\mathbf x}(s))}\,ds -\theta_1 - O(\e) \label{eq:anb} \\
&\ge  \int_{s_i}^{s_{-}}\frac{\gamma(\mathbf x'(s))}{c({\mathbf x}(s))}\,ds +  \int_{s_-}^{s_{+}}\frac{\gamma(\mathbf x'(s))}{c({\mathbf x}(s))}\,ds+ \int_{s_+}^{s_{i+1}}\frac{\gamma(\mathbf x'(s))}{c({\mathbf x}(s))}\,ds\notag  \\
 &\phantom{xxxxxxxxxxxxxxxxxxxxxxxx}- 2 \theta_1-C_{k,\ell} {\rm length}(h_{k} \cap R_{k, \ell}) \eta - O(\e)\notag \\
&= \int_{s_i}^{s_{i+1}} \frac{\gamma(\mathbf x'(s))}{c({\mathbf x}(s))}\,ds - 2 \theta_1-C_{k,\ell} {\rm length}(h_{k} \cap R_{k, \ell}) \eta - O(\e)\label{eq:upperissue}.
\end{align}
\end{enumerate}

Line \eqref{eq:anb} follows from Lemma \ref{lem:c1cor} for some $\theta_1 >0$ and $n$ large enough.
The line before last follows because either $c({\bf x}(s_k))$ is the largest rate in $R_{i,j}$ or, if it is the smallest of the two, we use Lemma \ref{lem:2p}. The fact that these estimates hold for all large $n$ follows from a Borel-Cantelli argument and the large deviation estimates, as seen in the proof of Lemma \ref{lem:4}.
Constants $C_{k, \ell}$ are the constants given in Lemma \ref{lem:2p}, that show up in bound \eqref{eq:diagbound}. They are all bounded above by some constant 
$\tilde C_{\delta}$ (which also depends on $x, y$), since all points where the derivative of increasing $h_i$ is 0 or undefined are isolated in cubes of side $\delta$.

We are now in a position to bound, for all $n$ large enough
\begin{align*}
G^{(n)}_{\fl{nx}, \fl{ny}} &\ge \sum_{i=0}^{K_{\delta, \eta}-1} G^{(n)}_{\fl{n {\bf x}(s_i)}, \fl{n{\bf x}(s_{i+1})}}  \\
&\ge n \sum_{i=0}^{K_{\delta, \eta}-1}\int_{s_i}^{s_{i+1}} \frac{\gamma(\mathbf x'(s))}{c({\mathbf x}(s))}\,ds- 3 K_{\delta, \eta} n(\theta_1+O(\e)) -  \tilde{C} \delta n - n \eta \tilde C_\delta \sum_{i=1}^Q\text{length}(h_i).
\end{align*}

Divide by $n$, and take the $\varliminf$ as $n \to \infty$ to obtain 
\begin{align}
\varliminf_{n\to \infty} \frac{G^{(n)}_{\fl{nx}, \fl{ny}}}{n} 
&\ge I({ \bf x}) -  3 K_{\delta, \eta} (\theta_1 +  O(\e))-  C \delta  - C_\delta \eta- O(\e).
\end{align}
As the quantifiers go to $0$, $K_{\delta, \eta} $ and $C_\delta $ blow up, so we first send $\theta_1$ to $0$ and $\e \to 0$. 
After that send  $\eta \to 0$  and finally $\delta \to 0$ to obtain that for an arbitrary ${\bf x} \in \mathcal H(x,y)$, 
\[
\varliminf_{n\to \infty} \frac{G^{(n)}_{\fl{nx}, \fl{ny}}}{n} \ge  I({ \bf x}). 
\]
A supremum over the class $\mathcal H(x,y)$ in the right hand-side of the display above gives 
\be
\varliminf_{n\to \infty} \frac{G^{(n)}_{\fl{nx}, \fl{ny}}}{n} \ge \Gamma_c(x,y).
\ee 

\textbf{Upper bound:} For the upper bound we first partition 
$[0,x]\times[0,y]$ into rectangles, so that it is a refinement of the partition used for the lower bound:
This way conditions (1)-(2) are satisfied and all rectangles in $\cup_{i} \mathcal J_{h_i, \eta}$ and $\mathcal I_{\delta}$ are part of this partition. Outside of the union of $\cup_{i} \mathcal J_{h_i, \eta}$ and $\mathcal I_{\delta}$, only the regions of constant rate remain. Divide each one of the constant region into rectangles, of side no longer than $\delta_1 > 0$ and assume $\delta_1 < \delta$.

Enumerate the rectangles in the two-dimensional partition by $Q_{i,j} = [ x_i, x_{i+1} ) \times [y_j, y_{j +1})$ and their total number is $N_{\eta, \delta, \delta_1} < \infty$.

Now, for any $n \in \N$ define the environment according to $c(x,y)$ and consider the maximizing path  $(0,0)$ to $(\fl{nx}, \fl{ny})$ which we denote by $\pi^{\max}_{0,  (\fl{nx}, \fl{ny})}$. The path can be written as a finite concatenation of sub-paths 
\[
\pi^{\max}_{{0,  (\fl{nx}, \fl{ny})}} = \sum_{(x_i,y_j)} \pi_{\fl {nQ_{i,j}}}
\]
where  $\pi_{\fl {nQ_{i,j}}}$ is the segment of the path in the rectangle 
$[\fl{nx_i}, \fl{nx_{i+1}} )\times[\fl{ny_j}, \fl{ny_{j+1}})$.
Some of these segments will be empty. 

We partition the sides of each rectangle $Q_{i,j}$ further: 
Fix a $\delta_2 > 0$ and define partitions 
\[
\mathscr P_{e_1, (i,j)} =\{ {\bf h}^{(i,j)}_k =  (x_i, y_j) + k \delta_2 e_1  \}_{0 \le k \le \frac{x_{i+1} - x_i}{\delta_2}}, \quad
\mathscr P_{e_2, (i,j)} =\{ {\bf v}^{(i,j)}_k =  (x_i, y_j) + k \delta_2 e_2  \}_{0 \le k \le \frac{y_{i+1} - y_i}{\delta_2}}.
\]
These completely define a partition of the boundaries $Q_{i,j}$. Now, the entry point of 
$\pi_{\fl {nQ_{i,j}}}$ into $nQ_{i,j}$ will be between two consecutive partition points, say 
$ {\bf a}_{k}^{(i,j)} \le  {\bf a}_{k+1}^{(i,j)}$ and its exit point will be between $ {\bf b}_{\ell}^{(i,j)} \le  {\bf b}_{\ell+1}^{(i,j)}$. Note that exit point of one rectangle will be the entry point in an adjacent one, and all these points belong to some partition $\mathscr P_{e_k, (i,j)} $. If it so happens and the path enters (or exits) from one of the macroscopic partition points, we set  ${\bf a}_{k}^{(i,j)} =  {\bf a}_{k+1}^{(i,j)}$ (equiv. $ {\bf b}_{\ell}^{(i,j)} = {\bf b}_{\ell+1}^{(i,j)}$).

When  the environment in $Q_{i,j}$ is constant $r_{i,j}$, we have the bound
\begin{align}
G^{(n)}_{\fl{nQ_{i,j}}}(\pi) 
&= \sum_{v \in \pi_{\fl{nQ_{i,j}}}} \tau_v^{(n)} \le G^{(n)}_{n{\bf a}_{k}^{(i,j)}, n{\bf b}_{\ell +1}^{(i,j)}}\le n\Big( \frac{\gamma({\bf b}_{\ell+1}^{(i,j)} - {\bf a}_{k}^{(i,j)})}{r_{i,j}} + \theta_1\Big)\notag \\
& \le n\Big( \frac{\gamma({\bf b}_{\ell}^{(i,j)} - {\bf a}_{k+1}^{(i,j)})}{r_{i,j}} + C_{i,j}\om_{\gamma}(\delta_2) + \theta_1\Big). \label{A}
\end{align}
The second-to-last inequality follows \text{by Theorem 4.2 in \cite{Sep-98-mprf-2}}, for $n$ large enough.

When $c(s,t)$ on $Q_{i,j}$ takes two values, $r_1, r_2$ separated by a curve $h$, we bound as follows. First fix a tolerance $\e$ and find $\delta_3> 0$ so that 
we may define a continuous speed function $c_{\delta_3, h}(s,t)$ as in Lemma \ref{lem:01:20}, with the property $c_{\delta_3, h}(s,t) \le  c(s,t)$ and 
\be\label{eq:5:22}
\sup_{{\bf a}_k, {\bf b}_\ell} (\Gamma_{c_{\delta_3, h}}({\bf a}_k,  {\bf b}_\ell ) - \Gamma_c( {\bf a}_k, {\bf b}_\ell )) < \e.
\ee 
Then, 
\begin{align}
G^{(n)}_{\fl{nQ_{i,j}}}(\pi)&= \sum_{v \in \pi_{\fl{nQ_{i,j}}}} \tau^{(n)}_v \le G^{(c_{\delta_3, h})}_{n{\bf a}_{k}^{(i,j)}, n{\bf b}_{\ell +1}^{(i,j)}}\notag\\
&\le n\big(\Gamma_{c_{\delta_3, h}}( {\bf a}_{k}^{(i,j)}, {\bf b}_{\ell+1}^{(i,j)})+ \theta_1\big) \text{ by a Borel-Cantelli argument and Lemma \ref{lem:newLDP}}, \notag \\
&\le n\big(\Gamma_{c_{\delta_3, h}}({\bf a}_{k+1}^{(i,j)}, {\bf b}_{\ell}^{(i,j)} )+ \om_{\Gamma _c}(2\delta_2)+ \theta_1\big) \text{ by Theorem \ref{cor:conti}},\\
&\le n\big(\Gamma_{c}(  {\bf a}_{k+1}^{(i,j)}, {\bf b}_{\ell}^{(i,j)})+ \e + \om_{\Gamma _c}(2\delta_2)+ \theta_1 \big)  \text{ by equation \ref{eq:5:22}}. \label{B}
\end{align}

Using the estimates \eqref{A} and \eqref{B}, we have total upper bound for the passage time 
\begin{align*}
G^{(n)}_{ \fl{nx}, \fl{ny}} &= \sum_{(i,j)} G^{(n)}_{\fl{nQ_{i,j}}}(\pi) \\
&\le n\sum_{(i,j)}  \Gamma_c({\bf a}_{k+1}^{(i,j)}, {\bf b}_{\ell}^{(i,j)})+ n N_{\eta, \delta, \delta_1}(\max_{(i,j)} C_{i,j} \om_{\gamma}(\delta_2) + \theta_1 + \e +\om_{\Gamma _c}(2\delta_2))+ n C | \mathcal I_{\delta}|\delta\\
&\le n ( \Gamma_c(x,y) + N_{\eta, \delta, \delta_1}(\max_{(i,j)} C_{i,j} \om_{\gamma}(\delta_2) + \theta_1 + \e +\om_{\Gamma _c}(2\delta_2))+  C | \mathcal I_{\delta}|\delta)
\end{align*}
The last line follows from superadditivity of $\Gamma$. 
To finish the bound, divide by $n$ and take the $\varlimsup_{n\to \infty}$. Then, let $\delta_2 \to 0$. This will result to finer $\mathscr P_{e_k, (i,j)}$ partitions, but by modulating $\delta_3$ we can still keep estimate \eqref{eq:5:22} with the same $\e$. Then let  $\theta_1$ and $\e$ tend to 0. These are independent of the other quantifiers $\eta$, $\delta_1$ and $\delta$. Finally send $\delta \to 0$. 
\end{proof}

\begin{proof}[Proof of Theorem \ref{thm:1}] Fix $(x,y)$ and fix an $\epsilon>0$. It is always possible to find piecewise strictly positive constant functions $c_1$ and $c_2$ such that $||c_1-c_2||_\infty\leq\epsilon$ that definitely have the same discontinuity curves as the function $c$ (but perhaps more).  On $[0,x] \times[0,y]$ we can further impose $c_1(x,y)\leq c(x,y)\leq c_2(x,y)$, by defining each $c_i$ on smaller rectangles.

When the weights in \eqref{eq:DSF} are defined via the speed function $c_i$ we write $G^i$ for last passage time and $\Gamma_{c_i}$ for their limits. A coupling using common exponential variables $\{\tau_{i,j}\}$ gives
\[G^{1,(n)}_{\lfloor nx\rfloor,\lfloor ny\rfloor}\geq G^{(n)}_{ \lfloor nx\rfloor,\lfloor ny\rfloor}\geq G^{2,(n)}_{\lfloor nx\rfloor,\lfloor ny\rfloor}.\]

Letting $r_{\min}>0$ denote a lower bound for $c(x,y)$ in the rectangle $[0,x]\times[0,y]$. Then we bound for any $\mathbf{x}\in\mathcal{H}$:
\[\begin{aligned}
0&\leq\int_0^1\bigg\{\frac{\gamma(\mathbf{x}'(s))}{c_1(x_1(s),x_2(s))}-\frac{\gamma(\mathbf{x}'(s))}{c_2(x_1(s),x_2(s))}\bigg\}ds\\
&=\int_0^1\frac{\gamma(\mathbf{x}'(s))(c_2(x_1(s),x_2(s))-c_1(x_1(s),x_2(s)))}{c_1(x_1(s),x_2(s))c_2(x_1(s),x_2(s))}ds\leq\epsilon\int_0^1\frac{\gamma(\mathbf{x}'(s))}{c^2_1(x_1(s),x_2(s))}ds\\
&\leq\epsilon r_{\min}^{-2}\gamma(x,y).
\end{aligned}
\]
As the inequality is uniform across $\bf x$, the bound extends to the suprema
\[0\leq\Gamma_{c_1}(x,y)-\Gamma_{c_2}(x,y)\leq C(x,y)\epsilon.\]
From Proposition \ref{prop:1} we know that the $\Gamma_{c_i}$ are the limits for $G^i$. To obtain Theorem \ref{thm:1}, let $\e \to 0$.
\end{proof}

\appendix
\section{Approximation in \eqref{m2ext}}
In this appendix section we perform all the computations step by step to get \eqref{m2ext}.
From \eqref{eq:-fprime} 
\be
f(a)=f(0)+\int_0^a f'(s)ds=f(0)-2c_{1/2}^{(-)}a^{1/2}-\frac{c}{\gamma+1/2}a^{\gamma+1/2},
\ee
for $a$ small enough. 
Since $m_2$ in \eqref{eq:m_2} is defined as a very complicated function of $a$ we prefer to approximate every addend separately and then put all together.  

Recall
\be\label{Texp}
\frac{1}{c_1x^{\alpha}+c_2x^{\beta}}=\frac{1}{c_1x^\alpha}\frac{1}{1+\frac{c_2}{c_1}x^{\beta-\alpha}}=\frac{1}{c_1x^\alpha}\Big(1-\frac{c_2}{c_1}x^{\beta-\alpha}+O(x^{2(\beta-\alpha)})\Big)\quad\alpha<\beta.
\ee
Use \eqref{Texp} to compute
\begin{align}
\frac{1}{f'(a)}&=\frac{-a^{1/2}}{c_{1/2}^{(-)}+ca^{\gamma}}=-\frac{a^{1/2}}{c_{1/2}^{(-)}}\frac{1}{1+\frac{c}{c_{1/2}^{(-)}}a^{\gamma}}\notag\\
&=-\frac{a^{1/2}}{c_{1/2}^{(-)}}\Big(1- \frac{c}{c_{1/2}^{(-)}}a^{\gamma} +O(a^{2\gamma})\Big)=-\frac{a^{1/2}}{c_{1/2}^{(-)}} +\frac{c}{c_{1/2}^{(-)2}}a^{\gamma+1/2} \label{1f}+O(a^{2\gamma+1/2}).
\end{align}
Since $m_1(a)=f(a)/a=\frac{f(0)}{a}(1-\frac{2r}{r-1}\frac{a^{1/2}}{\sqrt{f(0)}}-\frac{c}{\gamma+1/2}\frac{a^{\gamma+1/2}}{f(0)})$ we then have
\be\label{mf}
\frac{m_1(a)}{f'(a)}=-\frac{(r-1)\sqrt{f(0)}}{r}a^{-1/2}+c\frac{(r-1)^2}{r^2}a^{\gamma-1/2} +2-\frac{2c\gamma}{(\gamma+1/2)c_{1/2}^{(-)}}a^{\gamma}+O(a^{2\gamma+1/2}).
\ee
By the Taylor expansion 
 \be\label{sqe}
\sqrt{1+x}=1+\frac{1}{2}x+O(x^2)
\ee
we obtain
\begin{align*}
\sqrt{m_1(a)}&=\sqrt{\frac{f(0)}{a}}\Big(1-\frac{r}{r-1}\frac{a^{1/2}}{\sqrt{f(0)}}-\frac{c}{2(\gamma+1/2)}\frac{a^{\gamma+1/2}}{f(0)}+O(a)\Big)\\
&=\sqrt{f(0)}a^{-1/2}-\frac{r}{r-1}-\frac{c}{2(\gamma+1/2)}\frac{a^{\gamma}}{\sqrt{f(0)}}+O(a^{1/2})
\end{align*}
and using $\frac{1}{\sqrt{1+x}}=1-\frac{1}{2}x+O(x^2)$ we get
\begin{align}
\frac{1}{\sqrt{m_1(a)}}&=\sqrt{\frac{a}{f(0)}}\Big(1+\frac{r}{r-1}\frac{a^{1/2}}{\sqrt{f(0)}}+\frac{c}{2(\gamma+1/2)f(0)}a^{\gamma+1/2}+O(a)\Big)\notag\\
&=\frac{a^{1/2}}{\sqrt{f(0)}}-\frac{r}{(r-1)f(0)}a-\frac{c}{2(\gamma+1/2)(f(0))^{3/2}}a^{\gamma+1}+O(a^{3/2}).\label{sm}
\end{align}
From \eqref{eq:dioff} we are able to expand $-\frac{1}{f'(a)}-1+D$ which after some rearrangement we can substitute \eqref{1f}, \eqref{mf}, \eqref{sm} in and obtain
\begin{align}
-\frac{1}{f'(a)}&-1+D=(r-1)\Big(\frac{1}{f'(a)}+1\Big)+r\frac{1}{\sqrt{m_1(a)}}\Big( \frac{m_1(a)}{f'(a)}+1\Big)\\
&=(r-1)-\frac{(r-1)}{c_{1/2}^{(-)}}a^{1/2} +c\frac{(r-1)}{c_{1/2}^{(-)2}}a^{\gamma+1/2} +O(a^{2\gamma+1/2})+\Big(\frac{ra^{1/2}}{\sqrt{f(0)}}-\frac{r^2}{(r-1)f(0)}a\notag\\
&\hspace{0.3cm}-\frac{rc}{2(\gamma+1/2)(f(0))^{3/2}}a^{\gamma+1}+O(a^{3/2})\Big)\Big(-\frac{(r-1)\sqrt{f(0)}}{r}a^{-1/2}+c\frac{(r-1)^2}{r^2}a^{\gamma-1/2} +3\notag\\
&\hspace{9.5cm}-\frac{2c\gamma}{(\gamma+1/2)c_{1/2}^{(-)}}a^{\gamma}+O(a^{2\gamma+1/2})\Big)\notag\\
&=-\frac{3r^2}{(r-1)f(0)}a+\frac{3r^2+2r-1}{r\sqrt{f(0)}}a^{1/2} +c\frac{(r-1)^2}{r\sqrt{f(0)}}a^{\gamma} \label{br}\\
&\hspace{0.3cm}-c\Big(2-\frac{(r-1)^2}{r^2}+\frac{\gamma-1}{\gamma+1/2}\Big)\frac{r-1}{f(0)}a^{\gamma+1/2}+c\frac{r(4\gamma-1)}{2(\gamma+1/2)f(0)^{3/2}}a^{\gamma+1}+O(a^{2\gamma+1/2}).\notag
\end{align}
To know at which order of $a$ we can approximate we split out analysis into two cases according to the value of $\gamma$ 
\begin{enumerate}
\item $\gamma\in(0,1/2)$,
\item $\gamma\in[1/2,\infty)$.
\end{enumerate}
If $\gamma\in(0,1/2)$, from \eqref{br}
\be\label{sbr}
\Big(-\frac{1}{f'(a)}-1+D\Big)^2=c^2\frac{(r-1)^2}{c_{1/2}^{(-)2}}a^{2\gamma}+2c\frac{(3r^2+2r-1)}{c_{1/2}^{(-)2}}a^{\gamma+1/2}\\+O(a^{2\gamma+1/2}).
\ee
Substitute \eqref{sbr} into the following expression
\begin{align*}
&\sqrt{\Big(-\frac{1}{f'(a)}-1+D\Big)^2-4\frac{1}{f'(a)}}=\Big(c^2\frac{(r-1)^2}{c_{1/2}^{(-)2}}a^{2\gamma}+\frac{4}{c_{1/2}^{(-)}}a^{1/2} +c\Big(-r^2(4\gamma-1)\\
&\hspace{1cm}-4r(\gamma+1/2)+2(\gamma+1/2)\Big)\frac{(r-1)(3r^2+2r-1)}{r^3(\gamma+1/2)f(0)^{3/2}}a^{\gamma+1/2}  +O(a^{2\gamma+1/2})\Big)^{1/2}\\
&=c\frac{r-1}{c_{1/2}^{(-)}}a^{\gamma}\Big(1+4\frac{c_{1/2}^{(-)}}{c^2(r-1)^2}a^{-(2\gamma-1/2)} \\
&\hspace{0.7cm}+\Big(-r^2(4\gamma-1)-4r(\gamma+1/2)+2(\gamma+1/2)\Big)\frac{(3r^2+2r-1)}{c(\gamma+1/2)(r-1)^4c_{1/2}^{(-)}}a^{-\gamma+1/2}  +O(a^{1/2})\Big)^{1/2}.
\end{align*}
and by \eqref{sqe} we can Taylor expand  
\begin{align*}
&\sqrt{\Big(-\frac{1}{f'(a)}-1+D\Big)^2-4\frac{1}{f'(a)}}\\
&\hspace{1cm}=c\frac{r-1}{c_{1/2}^{(-)}}a^{\gamma}+\frac{2}{c(r-1)}a^{-\gamma+1/2}\\
&\hspace{1cm}\quad+\Big(-r^2(4\gamma-1)-4r(\gamma+1/2)+2(\gamma+1/2)\Big)\frac{(3r^2+2r-1)}{(2\gamma+1)(r-1)^3c_{1/2}^{(-)2}}a^{1/2}  +O(a^{\gamma+1/2}).
\end{align*}
In the end, putting all estimates together, we approximate  \eqref{eq:m_2}
\begin{align}
&\frac{1}{\sqrt{m_2(\e)}}=\frac{1}{2}\Big|-\frac{1}{f'(a)}-1+D+\sqrt{\Big(-\frac{1}{f'(a)}-1+D\Big)^2-4\frac{1}{f'(a)}}\Big|\notag \\
&=\frac{1}{2}\Big|\frac{3r^2+2r-1}{r\sqrt{f(0)}}a^{1/2} -c\frac{(r-1)^2}{r\sqrt{f(0)}}a^{\gamma}+c\frac{(r-1)^2}{r\sqrt{f(0)}}a^{\gamma}+\frac{2}{c(r-1)}a^{-\gamma+1/2} \notag \\
&\hspace{0.5cm}+\Big(-r^2(4\gamma-1)-4r(\gamma+1/2)+2(\gamma+1/2)\Big)\frac{(3r^2+2r-1)}{(2\gamma+1)(r-1)^3c_{1/2}^{(-)2}}a^{1/2}  +O(a^{\gamma+1/2})\Big|\notag \\
&=\frac{1}{2}\Big|\Big((2\gamma+1)\Big(2r-1+r(r-1)\sqrt{f(0)}\Big)+r^2(4\gamma-1)\Big)\frac{(3r^2+2r-1)}{(2\gamma+1)(r-1)^3c_{1/2}^{(-)2}}a^{1/2}\notag \\ 
&\hspace{9.5cm}+\frac{2}{c(r-1)}a^{-\gamma+1/2}+O(a^{\gamma+1/2})\Big|. \label{eq:app1}
 \end{align}
If $\gamma\in[1/2,\infty)$, from \eqref{br}
\be\label{sbr2}
\Big(-\frac{1}{f'(a)}-1+D\Big)^2=\frac{(3r^2+2r-1)^2}{r^2f(0)}a-2c\frac{(3r^2+2r-1)}{c_{1/2}^{(-)2}}a^{\gamma+1/2}+O(a^{\gamma+1}).
\ee
Use \eqref{sbr2} to obtain
\begin{align*}
&\sqrt{\Big(-\frac{1}{f'(a)}-1+D\Big)^2-4\frac{1}{f'(a)}}=\Big(\frac{4}{c_{1/2}^{(-)}}a^{1/2}+\frac{(3r^2+2r-1)^2}{r^2f(0)}a \\
&\hspace{1cm}+c\Big(-r^2(4\gamma-1)-4r(\gamma+1/2)+2(\gamma+1/2)\Big)\frac{(r-1)(3r^2+2r-1)}{r^3(\gamma+1/2)f(0)^{3/2}}a^{\gamma+1/2}  +O(a^{\gamma+1})\Big)^{1/2}\\
&=2\sqrt{c_{1/2}^{(-)}}a^{1/4}\Big(1+\frac{(3r^2+2r-1)^2}{4r(r-1)\sqrt{f(0)}}a^{1/2}\\
&\hspace{1cm}+c\Big(-r^2(4\gamma-1)-4r(\gamma+1/2)+2(\gamma+1/2)\Big)\frac{3r^2+2r-1}{4r^2(\gamma+1/2)f(0)}a^{\gamma}  +O(a^{\gamma+1/2})\Big)^{1/2}.
\end{align*}
By \eqref{sqe} 
\begin{align*}
&\sqrt{\Big(-\frac{1}{f'(a)}-1+D\Big)^2-4\frac{1}{f'(a)}}=2\sqrt{c_{1/2}^{(-)}}a^{1/4}+\frac{(3r^2+2r-1)^2}{2r^{3/2}\sqrt{r-1}f(0)^{3/4}}a^{3/4}\\
&\hspace{1cm}+c\Big(-r^2(4\gamma-1)-4r(\gamma+1/2)+2(\gamma+1/2)\Big)\frac{(3r^2+2r-1)\sqrt{r-1}}{2r^{5/2}(\gamma+1/2)f(0)^{5/4}}a^{\gamma+1/4}  +O(a^{\gamma+3/4})\Big).
\end{align*}
Finally, combining the estimates we have
\begin{align}
\frac{1}{\sqrt{m_2(\e)}}&=\frac{1}{2}\Big|-\frac{1}{f'(a)}-1+D+\sqrt{\Big(-\frac{1}{f'(a)}-1+D\Big)^2-4\frac{1}{f'(a)}}\Big|\notag \\
&=\frac{1}{2}\Big|2\sqrt{c_{1/2}^{(-)}}a^{1/4}+\frac{3r^2+2r-1}{r\sqrt{f(0)}}a^{1/2} +c\frac{r-1}{c_{1/2}^{(-)}}a^{\gamma}+O(a^{3/4})\Big|.\label{eq:app2}
\end{align}

Equation \eqref{m2ext}, follows from \eqref{eq:app1} and \eqref{eq:app2}. \qed

\bibliographystyle{ieeetr}

\end{document}